\documentclass[10pt]{article}

\usepackage{amssymb,latexsym,amsmath,verbatim,amsfonts,amsthm}
\usepackage{enumerate}
\usepackage[shortlabels]{enumitem}
\usepackage{color} 
\usepackage{algorithm}
\usepackage{mathtools}
\usepackage{subcaption}

\usepackage{url}
\usepackage[pdftex,
plainpages=false,
bookmarks,
bookmarksnumbered,
colorlinks=true,
linkcolor=blue,
citecolor=blue,
urlcolor=blue,
filecolor=black,
hypertexnames=false
]
{hyperref}
\usepackage[colorinlistoftodos,prependcaption,textsize=footnotesize]{todonotes}
\usepackage[framemethod=tikz]{mdframed}
\mdfsetup{%
  linewidth=1.25pt,
  backgroundcolor=gray!5,
  userdefinedwidth=\textwidth,
  roundcorner=10pt,
}

\newtheorem{lemma}{Lemma}[section]
\newtheorem{definition}[lemma]{Definition}

\newtheorem{theorem}[lemma]{Theorem}
\newtheorem{corollary}[lemma]{Corollary}
\newtheorem{proposition}[lemma]{Proposition}
\newtheorem{remark}[lemma]{Remark}

\def\RV{{\rm RV}}
\def\N{\mathbb{N}}
\def\F{{\rm Fix}}
\def\E{\mathcal{E}}

\def\RVz{{\rm RV}^0}
\def\B{{\mathbb{B}}}
\def\V{{\mathcal{V}}}
\def\hefun{{\gamma}}

\def\R{{\rm I\!R}}
\def\dist{{\rm dist}}
\def\inf{{\rm inf}}

\def\dom{{\rm dom}}

\newcommand{\RVzp}[1]{{\rm RV}^0_{#1}}
\newcommand{\RVp}[1]{{\rm RV}_{#1}}
\newcommand{\KRV}{{\rm KRV}_{\infty}}
\newcommand{\expCone}{\ensuremath{K_{\exp}}}

\newcommand{\inProd}[2]{\langle #1 , \, #2 \rangle }
\newcommand{\norm}[1]{\|#1\|}

\DeclareMathOperator{\interior}{int}
\DeclarePairedDelimiter\abs{\lvert}{\rvert}%

\title{\sf   Concrete convergence rates for common fixed point problems under Karamata regularity}

\author{Tianxiang Liu\thanks{Institute of Systems and Information Engineering, University of Tsukuba,  Japan. This author was supported in part by ACT-X, Japan Science and Technology Agency (Grant No. JPMJAX210Q) and the JSPS KAKENHI (Grant Number:25K15000). (\href{liutx@sk.tsukuba.ac.jp}{liutx@sk.tsukuba.ac.jp})} \and Bruno F. Louren\c{c}o\thanks{Department of Fundamental Statistical Mathematics, Institute of Statistical Mathematics, Japan. This author was supported partly by the JSPS Grant-in-Aid for Early-Career Scientists  23K16844. (\href{bruno@ism.ac.jp}{bruno@ism.ac.jp})} }
\numberwithin{equation}{section}
\allowdisplaybreaks

\begin{document}
\maketitle

\begin{abstract}
We introduce the notion of Karamata regular operators, which is a notion of regularity 
that is suitable for obtaining concrete convergence rates for common fixed point problems.
This provides a broad framework that includes, but goes beyond, H\"olderian error bounds and H\"older regular operators.
By \emph{concrete}, we mean that the rates we obtain are explicitly expressed in terms of a function of the iteration 
number $k$ instead, of say, a function of the iterate $x^k$. 
While it is well-known that under H\"olderian-like assumptions many algorithms converge 
linearly/sublinearly (depending on the exponent), little it is known when the underlying 
problem data does not satisfy H\"olderian assumptions, which may happen if a problem 
involves exponentials and logarithms. Our main innovation is the usage of the theory 
of regularly varying functions which we showcase by obtaining concrete convergence rates for quasi-cylic algorithms in non-H\"olderian 
settings.  This includes certain rates that are neither sublinear nor linear but sit somewhere in-between, including a case where the rate is expressed via the Lambert W function. 
Finally, we connect our discussion to o-minimal geometry and show that, under mild assumptions, definable operators in any o-minimal structure are always Karamata regular.
\end{abstract}

\section{Introduction}
In this paper, we consider the following common fixed point problem:
\begin{equation}\label{fcp_prob}
{\rm find}\ \ x\in F\coloneqq \bigcap_{i=1}^m\F\,T_i,
\end{equation}
where each $T_i$ ($i=1,\ldots,m$) is an \emph{$\alpha$-averaged} ($\alpha \in (0,1)$) operator (see definition in Section~\ref{sec:notation}) on a finite dimensional real vector space $\E$. We assume that $F$ is nonempty and $F\neq\E$. 
Many interesting problems can be reformulated as in \eqref{fcp_prob} and two notable examples are convex feasibility problems and certain variational inequality problems.  There are many algorithms for solving \eqref{fcp_prob} and one particularly broad class of methods corresponds to the family of \emph{quasi-cyclic algorithm}, considered in \cite[Theorem~6.1]{BNP15} and further analyzed in \cite{BLT17}.

Our main goal in this paper is to obtain \emph{concrete} convergence rates for quasi-cyclic algorithms for \eqref{fcp_prob}. Here, we emphasize that, by \emph{concrete} we mean that the convergence rate should be given in terms of an explicit function \emph{depending only on the iterate number}. 
More precisely, if we denote the $k$-th iterate of an algorithm by $x^k$ our desired convergence rates should have the form
\[
\dist(x^k,\, F) \leq R(k),
\]
where $R$ is some function of $k$.
We recall that if $R(k) = Mc^{-k}$ for some $c > 1$ and $M > 0$ the convergence rate is often said to be \emph{linear}. If $R(k) = Mk^{-r}$ for some $r > 0$ and $M > 0$ the rate is said to be \emph{sublinear}. 

In order to obtain a concrete converge rate (i.e., to obtain $R(k)$) typically some assumptions on the operators $T_i$ and their fixed point sets are required. 
As far as we know, the only way to obtain $R(k)$ so far is to make use of certain \emph{H\"olderian assumptions} such as assuming that the fixed point sets have a \emph{H\"olderian error bound} and the operators are \emph{H\"older regular} (see \cite{BLT17} or Remark~\ref{rmk_def} below).

A H\"olderian assumption for an operator typically takes the form of asking that over a bounded set, 
some \emph{power} of the residual $\norm{x-Tx}$ should be an upper bound to the true distance between $x$ and the fixed point set of $T$. That is, given a bounded set $B$, there should exist $\rho \in (0,\,1]$ and a constant $\kappa > 0$, such that $\dist(x,\, \F\,T) \leq \kappa \norm{x-Tx}^\rho$ holds, for $x \in B$.
Or this assumption can be taken jointly, by requiring that $\dist(x,\, F) \leq \kappa ({\max_{i=1}^m\norm{x-T_ix}})^\rho$ holds over $B$. In particular, this recovers the notion of H\"olderian error bound when each $T_i$ is a projection onto a given convex set $C_i$.

Under certain H\"olderian assumptions, it was shown in \cite[Theorem~3.1]{BLT17} that the quasi-cyclic algorithm converges at least linearly or sublinearly with a rate whose asymptotic behavior is controlled by  the powers appearing in the assumed H\"olderian conditions.

Here, however, we will consider the problem of obtaining concrete convergence rates when the underlying problem \emph{does not} necessarily satisfy H\"olderian assumptions.
This is motivated by the fact that there are certain problems for which their regularity properties are better expressed under more general conditions. For example, in \cite{LLP20}, it was shown that certain intersections of the exponential cone never admit H\"olderian error bounds \cite[Example~4.20]{LLP20}. Or, even when a H\"olderian error bound holds it may be the case that a tighter error bound can be obtained by making use of a non-H\"olderian error bound as in \cite[Remark~4.14 (a)]{LLP20}.
Other examples were found in \cite[Section~5.1]{LL20} and in the study of error bounds for log-determinant cones \cite{LLLP24}.

A situation where one \emph{can} actually expect H\"olderian assumptions to hold is when the problem data is semialgebraic, thanks to results such as the \L{}ojasiewicz inequality as used in \cite[Proposition~4.1]{BLT17}. 
Unfortunately,  whenever the problem data is related to exponentials and logarithms\footnote{As it is the case if the description of the graph of some operator in \eqref{fcp_prob}  requires exponentials and logarithms, for example. 
	This may happen if an operator is only definable in $o$-minimal structures containing the graph of the exponential function, see Section~\ref{sec: o-minimal} for more details. 
} (as it is in \cite{LLP20} and \cite{LLLP24}) we cannot typically ensure that the underlying operators  satisfy H\"olderian assumptions. 
This boils down to the fact that functions involving exponentials and logarithms are typically not semialgebraic.

As we move away from H\"olderian assumptions, obtaining concrete convergence rates becomes quite challenging. 
Here we should remark that, as discussed extensively in \cite[Section~7.1]{LL20}, certain  important previous works based on the Kurdya-\L{}ojasiewicz property have results concerning convergence rates for certain algorithms under general desingularizing functions (e.g., \cite[Theorem~24]{bdlm10} and \cite[Theorem~14]{BNPS17}), but these results do not lead to concrete rates since they are expressed in terms of the iterates $x^k$ instead of just as functions of $k$. Under KL theory, the cases where one obtains concrete convergence rates are typically restricted to the situation where the desingularizing function is a power function which implies the existence of the so-called \emph{KL exponents}, see more details in \cite{LP18,YLP21}.  
As such, the case of KL exponents can be seen as another kind of H\"older-type assumption.
This is particularly more pronounced in the convex case, in view of results such as \cite[Theorem~5]{BNPS17} as used, say, in \cite[Proposition~4.21]{LLP20} to connect a H\"olderian error bound to the  KL exponent of a certain function and vice-versa. 

Similarly, in the analysis of set-valued mappings and fixed points of operators, several general notions of regularity have been proposed. For example, Ioffe suggested the usage of gauge functions\footnote{One must be careful that gauge functions here are not the same gauge functions considered in, say, \cite[Section~15]{rockafellar}.} to express generalized versions of metric subregularity and other notions, see \cite[Section~2]{Io13}.
This was also used in subsequent works, e.g., \cite{LTT18,LTT18_2}. 
But, again, a challenge that seems to remain is getting concrete convergence rates for algorithms when the considered regularity notion is no-longer H\"olderian.

Part of the difficulty is that no matter how one frames a certain generalized notion of regularity, say through some function $\mu$ satisfying some key inequality, if we wish to obtain a convergence rate for some algorithm, we typically need to sum the effect of $\mu$ over the course of the algorithm and solve a recurrence inequality in $k$ (the iteration number) in order to get a rate. 
Solving general recurrence relations is a notoriously ad-hoc endeavour and this also adds to the difficulty of reasoning about rates beyond the H\"olderian case.

To the best of our knowledge, the first work to show examples of concrete rates under more general regularity conditions in a systematic way was \cite{LL20}, under the framework of \emph{consistent error bound functions}.
In particular, it was shown in \cite[Proposition~6.9]{LL20} that when the alternating projections algorithm is applied to the exponential cone and a certain subspace, it may happen that the rate is given by a function $R(k)$ that is proportional to $\frac{1}{\ln(k)}$. Also, for another choice of subspace the rate is ``almost linear'' in the sense that is faster than any sublinear rate but it may be slower than any linear rate. A key point in \cite{LL20} was the notion of \emph{regularly varying functions} \cite{Se76,BGT87}, a relatively known tool in probability theory but virtually unused in optimization theory. 
Again, to the best of our knowledge, \cite{LL20} was the first work in optimization to make systematic use of regular variation to study convergence rates.

The analysis done in \cite{LL20} concerns only convex feasibility problems and one of our goals here is to extend it to the more general problem class expressed in \eqref{fcp_prob}.
For that, we will introduce the notion of \emph{Karamata regular operator}, which is suitable for working with regular variation techniques. 
We will also deepen our usage of regular variation toolbox and address certain inelegancies and superfluous assumptions in \cite{LL20}. Our results will also  lead to sharper rates than the ones obtained in \cite{LL20}. In particular, for certain convex feasibility problems our rates here will be better than the ones that could be obtained by invoking the results in \cite{LL20}.

The basic idea here is to obtain rates via a series of functional transformations performed onto the regularity function that appears in the definition of Karamata regularity. Then, regular variation will be  useful because it will help us to analyze the asymptotic behaviour as we perform those functional transformations without the need of actually computing some of them. In this way, we will avoid, for instance, the need to evaluate certain hard integrals.

Under Karamata regularity, we will show how concrete convergence rates can be obtained through two different techniques. The first only requires knowledge of the index of regular variation, see Theorem~\ref{theo:rv}. Here we remark that the index of regular variation is a quantity that can be easily obtained from a simple limit computation, provided that we have the underlying regularity function, see Definition~\ref{def:rv}. The second technique requires more work to be applied but leads to tighter rates as we will see in Theorem~\ref{theorem_int} and Section~\ref{sec:examples}.

At the very end, we will connect the notion of Karamata regularity to definable sets in $o$-minimal structures and will show that, under mild assumptions, operators that correspond to definable functions can always be taken to be Karamata regular and are, therefore, under the scope of the techniques described in this paper.

Besides the specific goal of obtaining convergence rates for algorithms for \eqref{fcp_prob} and showing the applicability of Karamata regularity, we hope that this work will also inspire others to investigate other applications of regular variation in optimization, especially when the problem data is not semialgebraic and/or involves exponentials and logarithms.
With this goal in mind we tried to present a survey-like overview of regular variation in Section~\ref{sec:notation} with a view towards optimization applications. 
Also, with the exception of a single result in \cite{BGO06}, we confine all other references to results on regular variation to the classical textbook by Bingham, Goldie and Teugels \cite{BGT87}.

\subsection{Our contributions}

Our contributions are as follows.
\begin{itemize}
\item[-] We introduce the notion of (joint) Karamata regularity for operators in Definition~\ref{def_jcrv}, which is a regularity notion suitable for using tools from regular variation. We then discuss its connections with previous considered notions and prove a calculus rule in Proposition~\ref{prop_psi}.


\item[-] We prove an abstract convergence rate result for quasi-cyclic algorithms for common fixed point problems in Theorem~\ref{theo:rate}. Admittedly, applying directly Theorem~\ref{theo:rate} is hard because it requires \emph{inverting} an already relatively complicated integral. 
However, by making use of regular variation, we show how to bypass the evaluation of the complicated expression in Theorem~\ref{theo:rate}.
This is done through either the index of regular variation  (Theorem~\ref{theo:rv}) of the regularity function 
associated to Karamata regularity or through a sharper result that requires a bit more computation in 
Theorem~\ref{theorem_int}.
Several application examples are given in Section~\ref{sec:examples}, with a focus on cases having non-H\"olderian behavior.


\item[-] We explore the class of Karamata regular operators, and show that continuous quasi-nonexpansive operators defined on an $o$-minimal structure can always be taken to be jointly Karamata regular, see Theorem~\ref{theo:cvo}. As this includes the case of certain large $o$-minimal structures containing the graph of the exponential function, this shows that theory developed in this paper is applicable quite broadly. 
We also explore certain consequences of Theorem~\ref{theo:cvo} and show, for example, that the consistent error bound functions considered in \cite{LL20} can also be taken to be regularly varying, provided that the problem data is definable.
  At the end, we examine the particular case of polynomially bounded $o$-minimal structures and show that convergence rates are at least sublinear in those cases, see Corollary~\ref{col:poly_conv}.

\end{itemize}

This paper is organized as follows. In Section~\ref{sec:notation}, we introduce the notation and discuss the necessary notion from the theory of regular variation. In Section~\ref{sec:ca}, we  introduce the notion of joint Karamata regularity and establish an abstract convergence result for quasi-cyclic algorithms.  With the aid of regular variation, we further study the asymptotic properties of the convergence rates obtained. Later in Section~\ref{sec:examples}, we establish explicit convergence rates under a number of scenarios.
Finally, the class of Karamata regular operators is discussed in Section~\ref{sec: o-minimal} in the context of $o$-minimal structures.

\section{Preliminaries and basic notions from Karamata theory}
\label{sec:notation}
Let $\E$ be a finite-dimensional Euclidean space equipped with an inner product $\inProd{\cdot}{\cdot}$ and a corresponding norm $\norm{\cdot}$.
We denote the ball of radius $r$ centered at the origin  by $\B_r \coloneqq \{x\in\E\mid \|x\|\le r\}$.
Given a closed set $C \subseteq \E$ and $x \in \E$, we denote the projection of $x$ onto $C$  and the distance of $x$ to $C$ by $P_{C}(x)$ and $\dist(x,\,C)$, respectively.

We say an operator $T$ is \emph{$\alpha$-averaged} ($\alpha \in (0,1)$) if there exists a nonexpansive operator $R$ such that $T = (1 - \alpha)I + \alpha R$, where $I$ is the identity operator.
In particular, since the $T_i$'s  in \eqref{fcp_prob} are $\alpha$-averaged, each $T_i$ is nonexpansive and $\F\,T_i$ is convex, thanks to 
 \cite[Remark~4.24 and Proposition~4.13]{BC11}.  
 Moreover, the following property of $\alpha$-averaged operators follows from \cite[Proposition~4.25]{BC11}.

\begin{lemma}[$\alpha$-averaged operator]
\label{lm_alpha_av}
Let $T$ be an $\alpha$-averaged ($\alpha \in (0,1)$) operator on $\E$. Then it satisfies
\begin{equation*}
    \|T(x) - T(y)\|^2 + \frac{1 - \alpha}{\alpha}\big\|(I - T)(x) - (I - T)(y) \big\|^2 \le \|x - y\|^2, \ \ \ \forall\ x,\,y\in\E.
\end{equation*}
\end{lemma}
Next, we introduce some notation and preliminaries on the theory of regular variation, which we will use to conduct convergence analysis of algorithms for solving \eqref{fcp_prob}.  
More details on regular variation can be found in \cite{Se76,BGT87}.
We start with the notion of regularly varying functions. 
\begin{definition}[Regularly varying functions]\label{def:rv}
	A function $f: [a,\,\infty)\rightarrow (0,\,\infty)\, (a > 0)$
	is said to be \emph{regularly varying at infinity} if it is (Lebesgue) measurable and there exists a real number $\rho$ such that for all $\lambda > 0$ we have
	\begin{equation}\label{eq:rv}
	\lim_{x\rightarrow\infty}\frac{f(\lambda x)}{f(x)} = \lambda^{\rho}.
	\end{equation}
	In this case, we write  $f\in \RV_{\rho}$. 
	Similarly, a measurable function $f:(0,\,a]\rightarrow(0,\,\infty)$ is said to be regularly varying at $0$ if there exists a real number $\rho$ such that for all $\lambda > 0$ we have
	\begin{equation}\label{eq:rvz}
	\lim_{x\rightarrow 0_+}\frac{f(\lambda x)}{f(x)} = \lambda^{\rho},
	\end{equation}
	in which case we write $f \in \RVz_{\rho}$. The $\rho$ in \eqref{eq:rv} and \eqref{eq:rvz} is called the \emph{index} of regular variation. 
	
If the limit on the left hand side of \eqref{eq:rv} is $0$, $1$ and $+\infty$ for $\lambda$ in $(0,\,1)$, $\{1\}$ and $(1,\,\infty)$, respectively, then $f$ is said to be a function of \emph{rapid variation of index $\infty$} and 
we write $f \in \RV_{\infty}$. If $1/f \in \RV_{\infty}$, we say that $f$ is a function of \emph{rapid variation of index $-\infty$} and write $f \in \RV_{-\infty}$. $\RVzp{-\infty}$ and $\RVzp{\infty}$ are defined analogously.
\end{definition}
We denote by $\RVz$ the set of
of regularly varying functions at zero with index $\rho \in \R$, i.e., 
$\RVz \coloneqq \cup _{\rho \in \R} \RVz_{\rho}$. 
$\RV$ is defined analogously as $ \cup _{\rho \in \R} \RV_{\rho}$.
The functions in $\RVzp{0},\, \RVp{0}$ are said to be \emph{slowly varying}.
Regular variation at $0$ and at $\infty$ are naturally linked and we will use the following relation several times throughout this paper:
\begin{equation}\label{eq:rv_rvz}
f \in \RVzp{\rho} \quad \Longleftrightarrow \quad 
f(1/\cdot) \in \RVp{-\rho},
\end{equation}
where $\rho \in \R \cup \{-\infty,\,\infty\}$.

We also need some auxiliary definitions.
We say that a nonnegative function  $f$ defined  on a subset $C$ of the real line is \emph{locally bounded} if its restriction to each compact subset $K \subseteq C$ is bounded. If the restriction of $f$ to each compact $K \subseteq C$ satisfies $\inf_{t\in\, K}f(t) > 0$, then we say that $f$ is \emph{locally bounded away from zero}. Finally, we say that $f$ \emph{locally integrable}, if $\int _{K} f$ is finite for each compact $K \subseteq C$. 

An important fact is that we can always adjust the domain in order to ensure local boundedness. More precisely, if $f:[a,\,\infty) \to (0,\,\infty)$ belongs to $\RV$, then there exists $b \geq a$ such that the restriction of $f$ and $1/f$ to $[b,\,\infty)$ are both locally bounded, see \cite[Corollary~1.4.2]{BGT87}.

Analogously, if $f:(0,\,a] \to (0,\,\infty)$ belongs to $\RVz$, then $1/f(1/\cdot): [1/a,\,\infty)\to (0,\,\infty)$ belongs to $\RV$ by \eqref{eq:rv_rvz}. Then there exists some $b\in(0,\,a]$ (hence $1/b\ge 1/a$) such that the restriction of $1/f(1/\cdot)$ to $[1/b,\,\infty)$ is locally bounded. This implies that $f$ is locally bounded away from zero over  $(0,\,b]$. 

We note that if $f$ is a positive function on $C$ and it is monotone (either nondecreasing or nonincreasing) then it is both \emph{locally bounded} and \emph{locally bounded away from zero}.
For the sake of preciseness, we emphasize that $f$ is nondecreasing (resp.~increasing) if 
$f(t_1) \leq f(t_2)$ (resp.~$f(t_1) < f(t_2)$) holds when $t_1,\,t_2 \in C$ satisfies $t_1 < t_2$.
Nonincreasing/decreasing are defined analogously.

\paragraph{Calculus rules.}
For $f_1 \in \RVp{\rho_1},\, f_2 \in \RVp{\rho_2}$ with $\rho_1,\,\rho_2,\,\alpha \in \R$ we have the following calculus rules, see \cite[Proposition~1.5.7]{BGT87}:
\begin{align}\label{eq:rv_calc}
f_1f_2 \in \RVp{{\rho_1+\rho_2}},\quad f_1+f_2 \in \RVp{\max\{\rho_1,\,\rho_2\}},\quad f_1^\alpha \in \RVp{\alpha\rho_1}, \quad f_1\circ f_2 \in \RVp{\rho_1\rho_2},
\end{align}
where the last relation requires the additional hypothesis that 
$f_2(x) \to \infty$ as $x \to \infty$.
From \eqref{eq:rv_rvz} and \eqref{eq:rv_calc} we see that if $f_1 \in \RVzp{\rho_1},\, f_2 \in \RVzp{\rho_2}$, then:
\begin{align}\label{eq:rvz_calc}
f_1f_2 \in \RVzp{{\rho_1+\rho_2}},\quad f_1+f_2 \in \RVzp{\min\{\rho_1,\,\rho_2\}},\quad f_1^\alpha \in \RVzp{\alpha\rho_1}, \quad f_1\circ f_2 \in \RVzp{\rho_1\rho_2},
\end{align}
where the last relation requires the additional hypothesis that 
$f_2(x) \to 0$ as $x \to 0_+$. 

\begin{remark}[About the function domain and image]\label{rem:dom}
The literature on regular variation treats the domain of functions in a somewhat loose fashion. If $f:[a,\,\infty) \to (0,\,\infty)$ is in $\RV$, we can freely restrict $f$ to $[c,\,\infty)$ ($c > a$) or extend $f$ to $[b,\,\infty)$ ($b < a$) by letting $f$ take arbitrary positive values on $[b,\,a)$. More extremely, we can change the value of $f$ in a single bounded interval $[b,\,c]$ and none of these operations would change the asymptotic properties of $f$ at infinity nor the index of regular variation of $f$. So the calculus rules in \eqref{eq:rv_calc} (and much of this paper, in fact) should be seen under this light: while $f_1,\,f_2 \in \RV$ might have different domains of definition, we can restrict/extend their domains until $f_1+f_2$,\, $f_1f_2,\, f_1 \circ f_2$  are well-defined.
A similar comment applies to regular variation at $0$ so that if $f \in \RVz$ is defined over $(0,\,a]$, we can  arbitrarily change the value of $f$ in a single interval of 
the form $[b,\,c]$ with $b > 0$ and $c \in (b,\,\infty)\cup \{\infty\}$.

Due to aforementioned flexibility of restricting the function domain, we also treat the image of functions in a loose fashion. For a function $f$ whose image falls out of $(0,\,\infty)$, we still write $f\in\RV$ or $f\in\RVz$, if there exists some $a > 0$ such that the image of the restriction $f|_{[a,\,\infty)}$ or $f|_{(0,\,a]}$ is contained in $(0,\,\infty)$, respectively. Again, this is because only the asymptotic property of $f$ at infinity or $0$ matters.
\end{remark}

\begin{remark}[Examples of regularly varying functions]
	The function  mapping a nonnegative $x$ to $x^\rho$ belongs to both $\RV_{\rho}$ and $\RVz_{\rho}$.
With Remark~\ref{rem:dom} in mind, the natural logarithm function $\ln(\cdot)$ is in $\RV_{0}$ as it satisfies \eqref{eq:rv} and $\ln(x)$ is positive for $x > 1$, but it is not in $\RVz$, since $\ln(x)$ is negative for $x$ close to zero. 
	On the other hand, $-\ln(\cdot)$ is indeed a function in $\RVz_0$, as it satisfies \eqref{eq:rvz} and $-\ln(x)$ is positive for sufficiently small $x > 0$. 
	In view of \eqref{eq:rvz_calc}, 
	the function mapping a positive $x$ to $-x^\rho\ln(x)$ belongs to $\RVz_{\rho}$.
	Finally, the function mapping $x$ to $e^x$ is an example of a function in $\RV_{\infty}$.
	
\end{remark}

\paragraph{Potter's bounds.}
A great deal of information on the asymptotic behavior of a function can be extracted from its index of regular variation. A result known as \emph{Potter bounds} states that if $f \in \RVp{\rho}$, then for every $A > 1,\, \epsilon > 0$, there exists a constant $M$ such that 
$x \geq M,\, y \geq M$ implies
\begin{equation}\label{eq:potter_rv}
\frac{f(x)}{f(y)} \leq A\max\left \{\left(\frac{x}{y}\right)^{\rho -\epsilon}, \left(\frac{x}{y}\right)^{\rho +\epsilon}  \right\},
\end{equation}
see \cite[Theorem~1.5.6]{BGT87}.
Now, if $f \in \RVzp{\rho}$, then $\hat f$ such that $\hat f(t) \coloneqq 1/f(1/t)$ belongs to $\RVp{\rho}$, by \eqref{eq:rv_rvz} and \eqref{eq:rv_calc}. Applying \eqref{eq:potter_rv}, we see that
for any $A > 1,\, \epsilon > 0$, there exists a constant $M$ such that 
\begin{equation}\label{eq:potter}
\frac{f(t)}{f(s)} \leq A\max\left \{\left(\frac{t}{s}\right)^{\rho -\epsilon}, \left(\frac{t}{s}\right)^{\rho +\epsilon}  \right\},
\end{equation}
whenever $t \leq M,\, s\leq M$.
We note that taking $\epsilon = |\rho|/2$ in \eqref{eq:potter_rv}, fixing $x$ (if $\rho > 0$) or $y$ (if $\rho < 0)$ and taking limits, leads to the following conclusions:
\begin{align}
f\in \RVp{\rho} \text{ and } \rho > 0 \quad& \Rightarrow\quad   \lim _{x \to \infty} f(x) = +\infty, \label{eq:rv_inf} \\
f\in \RVp{\rho} \text{ and } \rho < 0 \quad& \Rightarrow\quad   \lim _{x \to \infty} f(x) = 0, \label{eq:rv_inf2} 
\end{align}
see also \cite[Proposition~1.3.6, item~($v$)]{BGT87}. Similarly, we have
\begin{align}
f\in \RVz_{\rho} \text{ and } \rho > 0 \quad& \Rightarrow\quad   \lim _{x \to 0_+} f(x) = 0, \label{lim_rv0_pos} \\
f\in \RVz_{\rho} \text{ and } \rho < 0 \quad& \Rightarrow\quad   \lim _{x \to 0_+} f(x) = +\infty. \label{lim_rv0_neg} 
\end{align}
There is also an analogous result for rapidly varying function. Bingham, Goldie and Omey proved that if $f \in \RVp{-\infty}$, then given any $r > 0$ there exists a constant $M > 0$ such that $x \geq M$ implies
\begin{equation}\label{eq:potter_rapid}
f(x) \leq x^{-r},
\end{equation}
see \cite[Lemma~2.2]{BGO06}, in particular $f(x) \to 0$ as $x \to \infty$. 
Therefore, if 
$f \in \RVp{\infty}$ (i.e., $1/f \in  \RVp{-\infty}$) then, there exists $M > 0$ such that $x \geq M$ implies
\begin{equation}\label{eq:potter_rapid2}
x^r \leq f(x),
\end{equation}
in particular $f(x) \to \infty$ as $x \to \infty$. 
\paragraph{Generalized inverses.}
Let $f:[a,\,\infty) \to \R$ be such that $f(x)$ tends to $\infty$ as $x \to \infty$. Then, we define its ``arrow'' generalized inverse as
\begin{equation}\label{eq:arrow}
f^{\leftarrow}(y) \coloneqq \inf \{ x \in [a,\,\infty)\mid  f(x) > y\},    
\end{equation}
see \cite[equation~(1.5.10)]{BGT87}. The function $f^{\leftarrow}$ is
nondecreasing and well-defined over $(0,\,\infty)$ even if the usual inverse of $f$ does not exist as it may happen if $f$ is nondecreasing but not necessarily increasing.
We will see shortly in Proposition~\ref{lb_proposition} that when $f^{-1}$ does in fact exist, it coincides with  $f^{\leftarrow}$ under suitable assumptions.

Similarly, for a function $f:(0,\,a]\to\R$ with $\lim_{x\to 0_+}f(x) = 0$, we define its ``minus'' generalized inverse as
\begin{equation}\label{inv_fun}
\begin{split}
    f^{-}(y)& : =  \sup \{ x \in (0,\,a] \mid   f(x) < y\},
\end{split}
\end{equation}
which is well-defined over $(0,\,\infty)$. Before we proceed, we note the following useful properties of the minus inverse.  
\begin{lemma}\label{inv_lemma}
	Let $f: (0,\,a]\to (0,\,\infty)$ satisfy $\lim _{x\to 0_+}f(x) = 0$. Then $f^{-}$ is  nondecreasing. If $f$ is  nondecreasing, then $0 < s\le f(t)$ implies that $f^{-}(s)\le t$.  
\end{lemma}

\begin{proof}
	The monotonicity of $f^{-}$ follows directly from its definition. 
	Next, let $0 < s\le f(t)$ and suppose that $f$ is nondecreasing. For all $y\in(0,\,a]$ satisfying $f(y) < s$, we have $f(y) < s \le f(t)$. 
	Then, the monotonicity of $f$ implies $y < t$. 
	Therefore, $f^{-}(s) =  \sup \{ y \in (0,\,a]\mid   f(y) < s\} \le t$. This completes the proof.
\end{proof}



The two generalized inverses are related as follows. Suppose that $f:(0,\,a]\to(0,\,\infty)$ satisfies $\lim_{x\to 0_+}f(x) = 0$. Let $g \coloneqq 1/f(1/\cdot)$. Then $g(x)$ tends to $\infty$ as $x\to\infty$. Moreover,
\begin{equation}\label{eq:rvz_inv_def}
\begin{split}
   f^{-}(y) & = \sup \{ x \in (0,\,a] \mid   f(x) < y\} = \frac{1}{\inf \{ u \in [1/a,\,\infty) \mid  f(1/u) < y\}}\\
  & =  \frac{1}{\inf \{ u \in [1/a,\,\infty) \mid   g(u) > 1/y\}} = \frac{1}{g^
{\leftarrow}(1/y)}.  
\end{split}
\end{equation}
In the spirit of Remark~\ref{rem:dom}, we will observe that under local boundedness,  the value of the generalized inverse $f^{\leftarrow}(y)$ does not depend on $a$ for sufficiently large $y$. Also, it may happen that  a non-monotone function $f(y)$ is increasing and continuous for sufficiently large $y$.
Nevertheless, this will still be enough to conclude that the generalized inverse will eventually coincide with the usual inverse.
For the sake of preciseness, in what follows, given a function $f: C \to \R$ and $S \subseteq \R$ we say that \emph{$f^{-1}$ is well-defined over $S$} if for every $y \in S$ there exists a unique $x \in C$ such that $f(x) = y$ holds. In this case, for $y \in S$, we can define $f^{-1}(y) \coloneqq x$ without ambiguity.

\begin{proposition}\label{lb_proposition}
	Suppose that $f: [a,\,\infty)\to\R$ is locally bounded and $f(x)\to\infty$ as $x\to\infty$. Let $b \geq a$, $\widehat{f}:= f|_{[b,\,\infty)}$ and  $M \coloneqq \sup_{x \in [a,\,b]} f(x)$. Then $f^{\leftarrow}(y) = \widehat{f}^{\leftarrow}(y)$  holds for $y\geq M$. In addition, if $f$ is continuous and increasing on $[b,\,\infty)$, then $f^{-1}$ is well-defined over $(M,\, \infty)$ and $f^{-1}(y) = f^{\leftarrow}(y)$ holds for $y> M$.
\end{proposition}
\begin{proof}
	We first observe that $M$ is finite because $f$ is locally bounded.
	Also, for any $y \ge  {M}$, if $f(x) > y$ holds we must have $x > b$, which together with $\widehat{f}:= f|_{[b,\,\infty)}$ implies the inclusion and the equality respectively:
	\[
	\left\{x\in [a,\,\infty) \mid f(x) > y\right\} \subseteq \left\{x\in [b,\,\infty) \mid f(x) > y\right\} = \{x\in [b,\,\infty) \mid \widehat{f}(x) > y\}.
	\]
	Since $b \geq a$, the converse inclusion $\left\{x\in [b,\,\infty) \mid f(x) > y\right\} \subseteq \left\{x\in [a,\,\infty) \mid f(x) > y\right\}$ holds, which together with the above relation implies that $\left\{x\in [a,\,\infty) \mid f(x) > y\right\} = \{x\in [b,\,\infty) \mid \widehat{f}(x) > y\}$. In view of the definition of the arrow inverse, we then have $f^{\leftarrow}(y) = \widehat{f}^{\leftarrow}(y)$. 
	This proves the first half.

	Now, we show the remaining half, where we assume that $f$ is continuous and increasing on $[b,\, \infty)$.
	If $x_1,x_2$ and $y > M$ are such that $f(x_1) = y = f(x_2)$ holds, then, by the definition of 
	$M$, we must have $x_1, x_2 \in (b,\, \infty)$. Since the restriction of $f$ to 
	$[b,\,\infty)$ is increasing, this implies that $x_1 = x_2$. Additionally, since $f$ is continuous and 
	goes to $\infty$ as $x\to \infty$, for every $y > M$ there exists at least one $x$ satisfying $f(x) = y$, which is a consequence of $M \geq f(b)$ and the intermediate value theorem. 
	We conclude that $f^{-1}$ is well-defined over $(M,\, \infty)$.
	
	Finally, let $y > M$ be arbitrary. 	If $f(x) > y$ holds, then $x > b$ holds. Also, since $f(f^{-1}(y)) = y > M$, 
	we have  $f^{-1}(y) > b$ too. So $f(x) > y = f(f^{-1}(y))$ implies $x > f^{-1}(y) $, since 
	$f$ is increasing on $(b, \, \infty)$. Therefore, we have the inclusion
	\[
	 \{x \in [a,\,\infty) \mid f(x) > y \} \subseteq \{x \in [b,\,\infty) \mid x > f^{-1}(y) \}.
	\]
	However, if $x \in [b,\,\infty)$ and $x > f^{-1}(y)$ holds, since $f$ is increasing on $[b,\,\infty)$, 
	we have $f(x) > f(f^{-1}(y)) = y$. Therefore, both sets coincide and we have
	\[
	f^{\leftarrow}(y) = \inf \{x \in [a,\,\infty) \mid f(x) > y \} = \inf \{x \in [b,\,\infty) \mid x > f^{-1}(y) \} = f^{-1}(y). 
	\]
This completes the proof. 
\end{proof} 
We observe that in the second half of the proof of Proposition~\ref{lb_proposition}, if $a = b$ holds, then $f^{-1}$ is also well-defined at $f(b)$ and a direct computation shows that $f^{\leftarrow}(f(b)) = b = f^{-1}(f(b))$. 
In particular, if $f:[a,\,\infty) \to (0,\,\infty)$ is continuous, increasing and satisfies $f(x) \to \infty$ as $x\to\infty$, then $f^{-1} = f^{\leftarrow}$ holds over $[f(a),\,\infty)$. Similarly, if $f:(0,\,a] \to (0,\,\infty)$ is continuous, increasing and satisfies $f(x) \to 0$ as $x \to 0_{+}$, then $f^{-1} = f^{-}$ holds over $(0,\,f(a)]$.

Moving on, an important result is that if $f$ is locally bounded on $[a,\,\infty)$ then:
\begin{equation}\label{eq:rv_inv}
f \in \RVp{\rho},\,\, \rho > 0 \quad \Rightarrow \quad f^{\leftarrow} \in \RVp{1/\rho},
\end{equation}
see \cite[Theorem~1.5.12]{BGT87} and this footnote\footnote{For an example of what can go awry if $f$ is not locally bounded, let $f:[1,\,\infty) \to (0,\,\infty)$ be such that $f(t) \coloneqq t$, for $t \geq 2$ and 
	$f(t)\coloneqq 1/(2-t)$ for $t \in [1,\,2)$. We have $f \in \RV_{1}$, but $f^{\leftarrow}(x) = \inf\{y \in [1,\,\infty) \mid f(y) > x\}$ does not go to $\infty $ as $x \to \infty$, so, in particular, $f^{\leftarrow} \not \in \RV_{1}$ (in view of \eqref{eq:rv_inf}). Still, we can always adjust the domain of a regularly varying function in order to ensure local boundedness as discussed previously. In this case, it is enough to restrict $f$ to $[2,\,\infty)$.}.


In order to describe the behavior of the arrow inverse when the index is $0$, we need an extra definition.
\begin{definition}\label{def:krv}
A positive measurable function $f$ belongs to the class of Karamata rapidly varying functions (denoted by $\KRV$) if and only if $f$ can be restricted or extended to the interval $[1,\,\infty)$ in such a way that
\[
f(x) = \exp\left\{z(x) + \eta(x) + \int_{1}^x \xi(t) \frac{dt}{t} \right\}, \quad x \geq 1,
\]
where $z,\,\eta,\,\xi$ are measurable functions such that $z$ is nondecreasing, $\eta(x) \to 0$ and $\xi(x) \to \infty$ as $x \to \infty$.
\end{definition}
The definition of $\KRV$ in \cite[Section~2.4]{BGT87} uses the so-called \emph{Karamata indices}, but thanks to \cite[Theorem~2.4.5]{BGT87} we can  equivalently use Definition~\ref{def:krv}.
We have the following implications:
\begin{align}
f \in \KRV \quad & \Rightarrow \quad f \in \RVp{\infty} \label{eq:krv},\\
f \in \RVp{\infty} \text{ and } f \text{ is  nondecreasing } \quad & \Rightarrow\quad f \in \KRV \label{eq:krv_mon},\\
f \in \KRV,\, g(x) \coloneqq x^\alpha,\, \alpha \in \R\quad  & \Rightarrow \quad gf \in \KRV, \label{eq:krv_power}
\end{align}
where \eqref{eq:krv} and \eqref{eq:krv_mon} follows from \cite[Proposition~2.4.4, item(iv)]{BGT87}. We now check \eqref{eq:krv_power}. Writing $g(x)$ as $\exp\left\{\int _{1}^x \alpha/t dt\right\}$, we see that $gf$ admits a representation as in Definition~\ref{def:krv} where $\xi(t) + \alpha$ appears instead of just $\xi(t)$. Since $\xi(t) +\alpha$ still goes to $\infty$ as $x \to \infty$, this shows that $gf \in \KRV$.

With that we have the following results. If $f$ is locally bounded and $f(x)$ goes to $\infty$ as $x \to \infty$, then 
\begin{align}
   f \in \RVp{0} \quad& \Rightarrow \quad f^{\leftarrow} \in \KRV, \label{eq:rav_inv}\\
   f \in \RVp{\infty} \quad& \Rightarrow \quad f^{\leftarrow} \in \RVp{0}, \label{eq:rav_inv2}
\end{align}
see \cite[Theorem~2.4.7]{BGT87}.

For regular varying functions at $0$ it will be more convenient to use the minus inverse. 
Suppose that $f:(0,\,a] \to (0,\,\infty) \in \RVzp{\rho}$ is such that $f(x)$ goes to $0$ as $x\to 0_+$, $\rho \geq 0$ and $f$ is locally bounded away from zero. Then, $g \coloneqq 1/f(1/\cdot)$ belongs to $\RVp{\rho} $, $g(t)$ goes to $\infty$ as $t \to \infty$ and $g$ is locally bounded on its domain 
$[1/a,\,\infty)$. 
This together with \eqref{eq:rvz_inv_def}, \eqref{eq:rv_inv} and \eqref{eq:rav_inv} allows us to conclude that if $f$ is bounded away from zero, then 
\begin{equation}\label{eq:rvz_inv}
   f \in \RVzp{\rho},\,\, \rho > 0 \quad \Rightarrow \quad f^{-} \in \RVzp{1/\rho},
\end{equation}
and in case of  $\rho = 0$ we have
\begin{equation}\label{eq:ravz_inv}
   f \in \RVzp{0} \quad \Rightarrow \quad f^{-} \in \RVzp{\infty}.
\end{equation}
\paragraph{Asymptotic equivalence.}
We say that two functions $f:(0,\,a]\to(0,\,\infty)$ and $g:(0,\,a]\to(0,\,\infty)$ are \emph{asymptotically equivalent up to a constant} if  there is a constant $\mu > 0$ such that 

\begin{equation}\label{eq:asy_def}
f(t) - \mu g(t) = o(g(t)), \text{ as } t\to 0_+.
\end{equation}
In this case, we write \emph{$f(t)\overset{c}{\sim} g(t)$ as $t\to0_+$}, or we may simply 
write \emph{$f\overset{c}{\sim} g$} if it is clear from context what is meant. If $\mu = 1$, we say that 
$f$ and $g$ are \emph{asymptotically equivalent}  and write  \emph{$f(t)\sim g(t)$ as $t\to0_+$} or $f \sim g$. Then, for measurable functions $f$ and $g$ we have the following implication:
\begin{equation}\label{asym_eq}
f \overset{c}{\sim} g, \ \ f\in\RVz_{\rho} \quad \Rightarrow \quad g\in\RVz_{\rho}.
\end{equation}
Indeed, by the definition one has $\lim_{x\to 0_+}\frac{f(x)}{g(x)} = \mu + \lim_{x\to 0_+}\frac{f(x) - \mu g(x)}{g(x)} = \mu$ and therefore,
\begin{equation*}
\lim_{x\to0_+}\frac{g(\lambda x)}{g(x)} = \lim_{x\to0_+}\frac{g(\lambda x)}{f(\lambda  x)}\frac{f(\lambda x)}{f(x)}\frac{f(x)}{g(x)} = \lim_{x\to0_+}\frac{g(\lambda x)}{f(\lambda  x)}\cdot\lim_{x\to 0_+}\frac{f(\lambda x)}{f(x)}\cdot\lim_{x\to 0_+}\frac{f(x)}{g(x)} = \frac{1}{\mu}\lambda^{\rho}\mu = \lambda^{\rho}.
\end{equation*}
For multiple functions, we called them \emph{pairwise asymptotically equivalent up to a constant}  (resp. \emph{pairwise asymptotically equivalent}) if any two functions among them 
are \emph{asymptotically equivalent up to a constant} (resp.~\emph{asymptotically equivalent}).

The notion in \eqref{eq:asy_def} corresponds to asymptotic equivalence at $0_+$, but similarly, we can define asymptotic equivalence at infinity, e.g., $f,g$ are \emph{asymptotically equivalent up to a constant} (at infinity) if $f(t) - \mu g(t) = o(g(t)), \text{ as } t\to +\infty$. For simplicity, we will use the same notation as it will be clear from context if asymptotic equivalence is meant at $0_+$ or at $\infty$.
Similarly, for measurable functions $f$ and $g$ we have 
\begin{equation}\label{asym_eq_RV}
f \overset{c}{\sim} g, \ \ f\in\RV_{\rho} \quad \Rightarrow \quad g\in\RV_{\rho}.
\end{equation}

\section{Karamata regularity and convergence rates}\label{sec:ca}

In this section, we will explore the convergence of a family of algorithms for the common fixed point problem \eqref{fcp_prob}. Naturally, this will be done under certain assumptions on the operators $T_i$.
We start by introducing the following definition.
\begin{definition}[Karamata regularity]\label{def_jcrv}
Let $L_i: \E\rightarrow\E$ $(i = 1,\ldots,n)$ be operators with $C \coloneqq \bigcap_{i=1}^n\F\,L_i\neq\emptyset$ and $B \subset \E$ a given bounded set. The $L_i$ are said to be \emph{jointly  Karamata regular (JKR) over $B$} if there exists a function $\psi_B:\R_{+}\to \R_{+}$ 
such that the following properties are satisfied.
\begin{enumerate}[$(i)$]
	\item \label{def_jcrv:1} The following error bound condition holds:
	\begin{equation}\label{eq:def_jcrv}
	\dist(x,\,C) \le \psi_B\Big(\max_{1\le i\le n}\|x - L_i(x)\|\Big), \ \ \forall\ x\in B.
	\end{equation}
	\item \label{def_jcrv:2}
	$\psi_B$ is  nondecreasing and  satisfies $\lim_{t\to0_+}\psi_B(t) = \psi_B(0) = 0$. 
	\item \label{def_jcrv:3}
	For some $a > 0$, it holds that $\psi_B|_{(0,\,a]}\in\RVz_{\rho}$ with $\rho\in[0,\,1]$. 
\end{enumerate}
We will refer to $\psi_B$ as a \emph{regularity function for the $L_i$'s  over $B$}.
If the operators $L_i$ are JKR over all bounded sets $B$ in such a way that the regularity functions $\psi_B$ can be taken to be  pairwise asymptotically equivalent up to a constant, we call them \emph{uniformly jointly Karamata regular (UJKR)}. 

In particular, when $n = 1$, we will drop the qualifier ``jointly'' and call the single operator $L$ \emph {Karamata regular (KR) over $B$} and \emph{uniformly  Karamata regular (UKR)}, respectively.

\end{definition}

\begin{remark}[Domain of $\psi_B$ and positivity]\label{dom_rv0}
In item~\ref{def_jcrv:3} of Definition~\ref{def_jcrv},  as far as regular variation at zero is concerned, the actual value of $a$ does not matter since only the behavior of $\psi_B$ as it approaches zero is relevant.
Nevertheless, even if we are flexible with the domain and image as in Remark~\ref{rem:dom},
a function in $\RVz_{\rho}$ must, at the very least, be positive close to zero.
Therefore, the requirement  that the restriction of $\psi_B$ to some $(0,\,a]$ is in $\RVz_{\rho}$ together with monotonicity implies $\psi_B(t) > 0$ for $t \neq 0$. 
\end{remark}

\begin{remark}[Connection with existing concepts]\label{rmk_def}
 Definition~\ref{def_jcrv} is closely related to several existing definitions. Later in Section~\ref{sec: o-minimal} we show that, under mild assumptions, operators whose fixed point sets intersect are always JKR over any bounded set.
 \begin{enumerate}
 
 \item[{\rm (i)}] \emph{(Bounded H\"{o}lder regular intersection)} Definition~\ref{def_jcrv} extends the definition of bounded H\"{o}lder regular intersection in \cite[Definition~2.2]{BLT17} as follows. 
 Let $C_1, \ldots, C_n \subseteq \E$ be convex sets and let $L_i \coloneqq P_{C_i}$ denote the projection operator onto $C_i$. 
With that, $C_1, \ldots, C_n \subseteq \E$ has a bounded H\"{o}lder regular intersection if and only if for every bounded set $B$ the $L_i$'s are JKR over $B$ and there exists $c_B > 0$ and $\gamma_B \in (0,\,1]$ such that  \eqref{eq:def_jcrv} holds with regularity function $\psi_B(\cdot) \coloneqq c_B(\cdot)^{\gamma_B}$.
In particular, if $\gamma_B\equiv\gamma$ does not depend on $B$ (in this case, we have that operators $L_i$ are UJKR), then the collection $\left\{C_i\right\}$ is bounded  H\"{o}lder regular with uniform exponent $\gamma$. 
We remark that the notion of bounded  H\"{o}lder regularity  coincides with the notion of H\"olderian error bound.
 
 \item[{\rm (ii)}]  \emph{(Bounded H\"{o}lder regular operators)} An operator  $L$ is  bounded H\"{o}lder regular (as in \cite[Definition~2.4]{BLT17}) if and only if 
 for every bounded set $B$, $L$ is KR over $B$ and the regularity function $\psi_B$ can be taken to be of the form $\psi_B(\cdot) = c_B(\cdot)^{\gamma_B}$ with $c_B >0$ and $\gamma_B \in (0,\,1]$. 
 The exponent $\gamma_B\equiv\gamma$ does not depend on $B$ (in this case, we have that $L$ is UKR) if and only if $L$ is bounded  H\"{o}lder regular with uniform exponent $\gamma$.

 \item[{\rm (iii)}] \emph{(Consistent error bounds)} For closed convex sets $C_i\subseteq\E$ $(i =1,\ldots,n)$ with non-empty intersection, a \emph{consistent error bound function} $\Phi$ (\cite[Definition~3.1]{LL20}) is a two-parameter function on $\R_+^2$, which is nondecreasing with respect to each variable and satisfies $\lim_{a\to 0_+}\Phi(a,\,b) = \Phi(0,\,b) = 0$ for all $b\ge 0$ and 
 \begin{equation}\label{eb_def_fcm}
\dist\Big(x,\,\bigcap_{i=1}^nC_i\Big) \le \Phi\Big(\max_{1\le i\le n}\dist(x,\,C_i),\, \|x\|\Big), \ \ \forall\ x\in\E.
\end{equation}
If \eqref{eb_def_fcm} holds, and for any $b > 0$ there exists some $\rho\in[0,\,1]$ such that $\Phi(\cdot,\,b)|_{(0,\,a]}\in\RVz_{\rho}$ for some $a > 0$, then the operators $L_i := P_{C_i}$ are JKR over the ball $\B_{b}$ of radius $b$ for any $b$. In addition, for any bounded set $B$, there exists some $r_B > 0$ such that $B\subseteq \B_{r_B}$. Let $\psi_{B}(\cdot):= \Phi(\cdot,\,r_B)$. Then we have from \eqref{eb_def_fcm} and the monotonicity of $\Phi$ with respect to the second variable that
\begin{equation*}
\begin{split}
\dist\Big(x,\,\bigcap_{i=1}^n\F\,L_i\Big) & = \dist\Big(x,\,\bigcap_{i=1}^nC_i\Big)  \le \Phi\Big(\max_{1\le i\le n}\dist(x,\,C_i),\, r_B\Big) \\
& = \psi_B\Big(\max_{1\le i\le n}\dist(x,\,C_i)\Big) = \psi_B\Big(\max_{1\le i\le n}\|x - L_i(x)\|\Big),\ \forall\ x\in B.
\end{split}
\end{equation*}
 Note that $\lim_{t\to 0_+}\psi_B(t) = \lim_{t\to 0_+}\Phi(t,\,r_B) = 0$ and $\psi_B(0) = \Phi(0,\,r_B) = 0$. In addition, $\psi_B$ is nondecreasing, thanks to the monotonicity of $\Phi$ with respect to the first variable. Moreover, the property of $\psi_B$ in item (iii) of Definition~\ref{def_jcrv} directly follows from the assumption on $\Phi$. 
 
 The summary is that if we have a consistent error bound function $\Phi$ for the $C_i$'s where for sufficiently large $b > 0$ the functions $\Phi(\cdot,\,b)$ are regularly varying with index $\rho _b \in [0,\,1]$, then the $P_{C_i}$'s are JKR over any bounded set $B$ and the regularity function can be taken to be  $\psi_{B}(\cdot):= \Phi(\cdot,\,b)$ for any $b$ satisfying $b \geq r_B$.
 Later in Corollary~\ref{col:consistent}, we will see that if the $C_i$'s are definable over an $o$-minimal structure and have non-empty intersection, we can always construct such a consistent error bound function.

 \end{enumerate}
 
\end{remark}

A typical situation in applications is  having some 
regularity condition on each \emph{individual operator} (e.g., as in Remark~\ref{rmk_def}~(ii)) and some error bound condition on the fixed point sets (e.g., as in Remark~\ref{rmk_def}~(iii)).
The next results indicates how to aggregate these individual results and establish joint Karamata regularity for the operators.


\begin{proposition}[Calculus of $\psi_B$]\label{prop_psi}
 Suppose that $L_i: \E\rightarrow\E$ $(i = 1,\ldots,n)$  are nonexpansive, closed operators such that $\bigcap_{i=1}^n\F\,L_i\neq\emptyset$ holds. 
Let $\Phi$ be a consistent error bound function for the sets $\F\,L_i$, $B \subseteq \E$ a bounded set and suppose that $\Phi$ and $L_i$ are as follows:
\begin{enumerate}
\item[{\rm (i)}] for any $b > 0$, there exists some $\theta_b\in[0,\,1]$ such that $\Phi(\cdot,\,b)|_{(0,\,a]}\in\RVz_{\theta_b}$ for some $a > 0$;
\item[{\rm (ii)}] each $L_i$ is Karamata regular over $B$.
\end{enumerate}
Then, the following statements hold.
\begin{enumerate}[$(a)$]
	\item The operators $L_i$ $(i = 1,\ldots,n)$ are jointly Karamata regular over $B$. 
	\item Let $\Gamma_B^i$ be the regularity function for each $L_i$ over $B$. Assume that $\Gamma_B^i|_{(0,\,a]}\in\RV_{\rho_i}^0$.
	Then, the function defined by
	 \[\psi_B:= \Theta_B\circ\Gamma_B\]
	 satisfies \eqref{eq:def_jcrv} and $\psi_B|_{(0,\,a]}\in\RVz_{\rho}$, where $\rho = \theta_{b}\min\limits_{1\le i\le n}\rho_i$, $\Theta_B(\cdot):=\Phi(\cdot,\,b)$, $\Gamma_B:=\sum\limits_{i=1}^n\Gamma_B^i$ and $b$ is such that $B\subseteq \B_{b}$.
	 \end{enumerate}

\end{proposition}

\begin{proof}
We will prove item~(a) and (b) together. 
Since each $L_i$ is nonexpansive and closed,  each $\F\,L_i$ must be closed and convex, e.g.,  \cite[Proposition~4.13]{BC11}.  
By assumption, $\bigcap_{i=1}^n\F\,L_i\neq\emptyset$ and $\Phi$ is a  consistent error bound function  for the sets $\F\,L_i$ $(i = 1,\ldots,n)$. Therefore,
\begin{equation*}
\dist\Big(x,\,\bigcap_{i=1}^n\F\,L_i\Big) \le \Phi\Big(\max_{1\le i\le n}\dist(x,\,\F\,L_i),\, \|x\|\Big), \ \forall\ x\in\E.
\end{equation*}
This together with $B\subseteq \B_{b}$, the monotonicity of $\Phi$ and the definition of $\Theta_B$ further implies
\begin{equation}\label{eb_sets}
	\dist\Big(x,\,\bigcap_{i=1}^n\F\,L_i\Big)  \le \Phi\Big(\max_{1\le i\le n}\dist(x,\,\F\,L_i),\, b\Big) = \Theta_B\Big(\max_{1\le i\le n}\dist(x,\,\F\,L_i)\Big),\ \forall\ x\in B.
\end{equation}
On the other hand, we see from assumption~(ii) and the assumption on the $\Gamma_B^i$'s that
\begin{equation}\label{pr_sing}
\dist(x,\,\F\,L_i) \le \Gamma_B^i\left(\|x - L_i(x)\|\right),  \ \forall\ x\in B.
\end{equation}
Combining \eqref{eb_sets} and \eqref{pr_sing}, we have from the monotonicity of $\Theta_B$ and the nonnegativity of $\Gamma_B^i$  that for all $x\in B$,
\begin{equation*}
\begin{split}
\dist\Big(x,\,\bigcap_{i=1}^n\F\,L_i\Big) & \le \Theta_B\Big(\max_{1\le i\le n}\Gamma_B^i\left(\|x - L_i(x)\|\right)\Big) \le \Theta_B\Big(\max_{1\le i\le n}\Gamma_B\left(\|x - L_i(x)\|\right)\Big)\\
& = \Theta_B\Big(\Gamma_B\Big(\max_{1\le i\le n}\|x - L_i(x)\|\Big)\Big) = \psi_B\Big(\max_{1\le i\le n}\|x - L_i(x)\|\Big).
\end{split}
\end{equation*}
Note that $\Theta_B(\cdot)=\Phi(\cdot,\,b)$ is  nondecreasing and \[
\lim_{t\to 0_+}\Theta_B(t) = \lim_{t\to 0_+}\Phi(t,\,b) = 0 = \Theta_B(0).\]
Also, $\Gamma_B$ is  nondecreasing and it holds that $\lim_{t\to 0_+}\Gamma_B(t) = \Gamma_B(0) = 0$, thanks to $\Gamma_B=\sum\limits_{i=1}^n\Gamma_B^i$ and the properties of $\Gamma_B^i$. Consequently, $\psi_B= \Theta_B\circ\Gamma_B$ is  nondecreasing and satisfies $\lim_{t\to 0_+}\psi_B(t) = \psi_B(0) = 0$.
Moreover, since $\Theta_B|_{(0,\,a]}\in\RVz_{\theta_b}$ and $\Gamma_B^i|_{(0,\,a]}\in\RVz_{\rho_i}$, using \eqref{eq:rvz_calc} we have $\psi_B|_{(0,\,a]} = (\Theta_B\circ\Gamma_B)|_{(0,\,a]}\in\RVz_{\rho}$ with $\rho = \theta_b\min\limits_{1\le i\le n}\rho_i$. This completes the proof.
\end{proof}


\subsection{General convergence theory}
The goal of this section is to present a general result that connects the error bound function $\psi_B$ appearing in Definition~\ref{def_jcrv} to the convergence rate of a sequence generated by the quasi-cyclic algorithm described in \cite{BLT17}. This will be accomplished by applying a series of functional transformations to $\psi_B$ which will culminate in Theorem~\ref{theo:rate}, the main result of this section. 
Theorem~\ref{theo:rate} is an abstract result that is hard to apply directly, but in later sections we will show how to use regular variation to better estimate the asymptotic properties of the function appearing in Theorem~\ref{theo:rate} without  explicitly computing it.

\paragraph{Quasi-cylic algorithms.}
We recall the common fixed point problem in  \eqref{fcp_prob}. In \cite[Section~3]{BLT17}, 
Borwein, Li and Tam analysed the framework 
of \emph{quasi-cyclic algorithms} which was considered earlier by Bauschke, Noll and Phan in \cite{BNP15}.
A quasi-cylic algorithm is given by iterations that are as follows:
\begin{equation}\label{alg_iter}
x^{k + 1} = \sum_{i=1}^m w_i^kT_i(x^k),
\end{equation}
where each $T_i$ is $\alpha$-averaged, the weight parameters $w_i^k \ge 0$ satisfy $\sum_{i=1}^mw_i^k = 1$ and $\nu:=\mathop{\inf}\limits_{k\in\N}\min\limits_{i\in I_+(k)}w_i^k >0$, where $I_+(k):=\{1\le i\le m\mid w_i^k > 0\}$.  
The algorithm framework \eqref{alg_iter} covers a number of projection algorithms, the Douglas-Rachford splitting method as well as the forward-backford splitting method, see \cite{BLT17}.
Later, in Theorem~\ref{theo:rate} we will impose additional conditions on the $I_+(k)$ in order to ensure convergence. 
This amounts to requiring that each operator $T_i$ is ``activated'' (i.e., $w_i^k > 0$) at least every $s > 0$ iterations, which is an idea that harkens back in different forms to several earlier works, e.g., see \cite[Section~3]{ABC83}, \cite[Section~2]{Censor84}, \cite[pg.~435]{Ottavy88}, \cite[section~2]{FZ90}, \cite[pg.~382]{BB96} and many others.

\paragraph{A general convergence rate result.}
In \cite{BLT17}, the authors derived convergence 
rate results for the iteration \eqref{alg_iter} under certain H\"olderian assumptions. 
Here, one of our main goals is to prove convergence 
rates under the more general Karamata regularity condition as 
in Definition~\ref{def_jcrv}.

Let $f: (0,\,a]\to (0,\,\infty)$ satisfy $\lim _{x\to 0_+}f(x) = 0$. Fix $\delta > 0$ and consider the following  integral.
\begin{equation}\label{pf_f}
    \Phi_f(x):= \int_x^{\delta}\frac{1}{f^{-}(t)}dt, \ \ x > 0,
\end{equation}
which is well-defined for 
$x \in (0,\infty)$.
The integral in \eqref{pf_f} plays a fundamental role in this paper and its \emph{inverse} is related to the convergence rate of algorithms as will be described in Theorem~\ref{theo:rate}.
In the next lemma, we will check some properties of $\Phi_f$.

\begin{lemma}\label{pf_inv}
Let $f: (0,\,a]\to (0,\,\infty)$ satisfy $\lim _{x\to 0_+}f(x) = 0$ and let $\Phi_f$ be defined as in \eqref{pf_f}. Then the following statements hold.
\begin{enumerate}[$(i)$]
	\item \label{pf_inv_i} $\Phi_f$ is continuous and  decreasing;
	\item \label{pf_inv_ii} Suppose that $f$ is  nondecreasing and one of the conditions below is satisfied:
	\begin{enumerate}[$(a)$]
		\item \label{case_a} $f\in\RVz_{\rho}$ with $\rho\in[0,\,1)$;
		
		\item  \label{case_b} $f(x)\ge cx$ holds for some $c > 0$ as $x\to 0_+$.
	\end{enumerate}
	Then, $\Phi_f(x)\to\infty$ as $x\to 0_+$.
\end{enumerate}
\end{lemma}

\begin{proof}
First, we see from the definition of $f^{-}$ in \eqref{inv_fun} that $f^{-}(x) >0$ for all $x >0$. Then $\Phi_f$ is  decreasing. Next, we show the continuity of $\Phi_f$ on $(0,\,\infty)$. For any fixed $\widehat{x}\in(0,\,\infty)$, one can find a compact interval $[c,\,d]$ such that $\widehat{x}\in[c,\,d]$ and $c > 0$. Then for any $x\in[c,\,d]$, one has
\begin{equation*}
\Phi_f(x) = \int_x^{\delta}\frac{1}{f^{-}(t)}dt = \int_c^{\delta}\frac{1}{f^{-}(t)}dt + \int_{x}^c\frac{1}{f^{-}(t)}dt = \Phi_f(c) + \int_{c}^x\frac{-1}{f^{-}(t)}dt.
\end{equation*}
Let $g(t):=\frac{-1}{f^{-}(t)}$. By  Lemma~\ref{inv_lemma}, we see that $g$ is  nondecreasing, and therefore $g|_{[c,\,d]}$ is measurable and integrable. Using \cite[Theorem~4.4.1]{al06}, we then have that $\Phi_f|_{[c,\,d]}$ is absolutely continuous. This together with the arbitrariness of $\widehat{x}$ proves the continuity of $\Phi_f$ on $(0,\,\infty)$ and completes proof of item~\ref{pf_inv_i}.

Now we prove $\Phi_f(x)\to\infty$ as $x\to 0_+$ in two cases.
\begin{itemize}
\item In case \ref{case_a}, from \eqref{eq:rvz_calc} we have $f(t)/t\in\RVz_{\rho - 1}$. Since $\rho - 1 < 0$, we see from \eqref{lim_rv0_neg} that $f(t)/t\to\infty$ as $t\to 0_+$, which further implies $f(t) \ge t$ when $t\in (0,\,\varepsilon]$ for some $\varepsilon\in(0,\,\delta)$.
By the monotonicity of $f$ and Lemma~\ref{inv_lemma}, it holds $f^{-}(t) \le t$ for $t\in (0,\,\varepsilon]$. Using this and $f^{-} > 0$, we have for all $x\in(0,\,\varepsilon)$ that
\begin{equation*}
    \Phi_f(x):= \int_x^{\delta}\frac{1}{f^{-}(t)}dt \ge \int_x^{\varepsilon}\frac{1}{f^{-}(t)}dt \ge \int_{x}^{\varepsilon}\frac{1}{t}dt = \ln(\varepsilon) - \ln(x),
\end{equation*}
which proves $\Phi_f(x)\to\infty$ as $x\to 0_+$.

\item In case \ref{case_b}, we see that there exists some  $\varepsilon\in(0,\,\delta/c]$ such that $cx \leq f(x)$ for $x\in(0,\,\varepsilon]$. By the monotonicity of $f$ and Lemma~\ref{inv_lemma}, it holds $f^{-}(cx)\le x$ for  $x\in(0,\,\varepsilon]$. Thus, for all $x\in(0,\,c\varepsilon)$, it holds
\begin{equation*}
\Phi_f(x):= \int_x^{\delta}\frac{1}{f^{-}(t)}dt \ge  \int_{x}^{c\varepsilon}\frac{1}{f^{-}(t)}dt = c\int_{x/c}^{\varepsilon}\frac{1}{f^{-}(cy)}dy \ge c\int_{x/c}^{\varepsilon}\frac{1}{y}dy = c\ln(\varepsilon) - c\ln(x/c),
\end{equation*}
which proves $\Phi_f(x)\to\infty$ as $x\to 0_+$.
\end{itemize}
This completes the proof.
\end{proof}

\begin{remark}[Condition in Lemma~\ref{pf_inv}]
We list two special conditions contained  in Lemma~\ref{pf_inv}~(ii)~\ref{case_b}: 
\begin{itemize}
\item[{\rm (i)}] $\lim_{x\to0_+}\frac{f(x)}{x} = \infty$. This includes the entropic error bound function in \cite[Section~4.2.1]{LLP20} and \cite[Section~6.2]{LL20}: $f(x) \coloneqq -x\ln(x)$, $x\in(0,\,a]$ for some $a > 0$ and we note that $f$ belongs to $\RVz_1$.
\item[{\rm (ii)}] $\lim_{x\to0_+}\frac{f(x)}{x} = \mu$ for some $\mu > 0$. This corresponds to $f(x)\overset{c}{\sim} x$, and further implies that $f\in\RVz_1$, thanks to \eqref{asym_eq}.
\end{itemize}
One the other hand, the condition $f\in\RVz_1$ alone is not enough to guarantee $\Phi_f(x)\to\infty$ as $x\to 0_+$. Consider the following function:
\begin{equation*}
g(x):= x(1 + x)\left(\ln(1 + 1/x)\right)^2, \ \ \ x\in(0,\,1/(e^2 - 1)].
\end{equation*}
We have $g\in\RVz_1$ and $g$ is  increasing and continuous. Then the usual inverse $g^{-1}$ exists. Let $f := g^{-1}$. We then have $f^{-} = f^{-1} = g$. Moreover, we have from \eqref{eq:rvz_inv} that $f\in\RVz_1$ and $f$ is  increasing. However, when $x,\,\delta\in(0,\, 1/(e^2 - 1)]$,
\begin{equation*}
    \Phi_f(x):= \int_x^{\delta}\frac{1}{f^{-}(t)}dt = \int_{x}^{\delta}\frac{1}{t(1 + t)\left(\ln(1 + 1/t)\right)^2}dt = \frac{1}{\ln(1 + 1/\delta)} - \frac{1}{\ln(1 + 1/x)},
\end{equation*}
which implies that $\Phi_f(x)\not\to\infty$ as $x\to0_+$.
\end{remark}

All pieces are now in place for the main result of this subsection.
Here, we recall that $F$ is as in \eqref{fcp_prob}, i.e., it is the intersection of the fixed point sets of the $T_i$'s, where each $T_i$ is assumed to be $\alpha$-averaged.
Also, $\nu$ and $I_+(k)$ are as described after \eqref{alg_iter}.
\begin{theorem}\label{theo:rate} 
Let sequence $\{x^k\}$ be generated by quasi-cyclic algorithm \eqref{alg_iter}. Then  $\{x^k\}$ is bounded. 
Let $B$ be a bounded set containing $\{x^k\}$ and suppose that the following assumptions hold:
\begin{itemize}
\item[{\rm (a)}] $T_i$ $(i = 1,\ldots,m)$ are jointly Karamata regular (\emph{resp.}  Karamata regular when $m = 1$)  over $B$ with regularity function $\psi_B$  as in Definition~\ref{def_jcrv};  

\item[{\rm (b)}] there exists some $s > 0$ such that for each $k\in\N$,
\begin{equation*}
 I_+(k)\cup I_+(k+1)\cup\cdots\cup I_+(k + s -1) = \{1,\ldots,m\}.   
\end{equation*}
\end{itemize}
 Fix any $\widehat{a} > \dist^2(x^0,\,F)$ and $\delta > 0$ and define 
\begin{equation}\label{def_pf}
  \widehat{\phi}(u) := \psi_B^2\Big(\sqrt{\frac{2\alpha(1 + 4\nu s)}{\nu(1-\alpha)}u}\Big),  \ \ \ \phi(u) = \widehat{\phi}(u)|_{(0,\,\widehat{a}]}, \ \ \  \Phi_{\phi}(u)  :=\int_u^{\delta}\frac{1}{\phi^{-}(t)}dt, \ u > 0.
\end{equation}
Then $\{x^k\}$ converges to some $x^*\in F=\bigcap_{i=1}^m\F\,T_i$ finitely or the convergence rate is given by 
\begin{equation}\label{gener_rate}
    \dist(x^k,\,F) \le \sqrt{\Phi_{\phi}^{-1}\left(\Phi_{\phi}\left(\dist^2\left(x^{0},\,F\right)\right) + \lfloor{k/s}\rfloor\right)},\ \ \forall\ k\in\N.
\end{equation}
\end{theorem}

\begin{proof}
Let $y\in F$. Then for all $k\in\N$, the nonexpansiveness of $T_i$ (see from Lemma~\ref{lm_alpha_av}) gives
\begin{equation*}
  \|x^{k+1} - y\| = \Big\|\sum_{i=1}^mw_i^kT_i(x^k) - y\Big\| \le \sum_{i=1}^mw_i^k\|T_i(x^k) - y\| \le \sum_{i=1}^mw_i^k\|x^k - y\| = \|x^k - y\|,
\end{equation*}
which simultaneously proves that $\{x^k\}$ is bounded and Fej\'{e}r monotone. In particular, whenever 
$r \geq k$ and $y \in F$ we have
\begin{align}
\norm{x^{r} - y} & \leq \norm{x^{k} - y}, \label{eq:fejer2}\\
\dist(x^{r},\,F) & \leq \dist(x^{k},\,F). \label{eq:fejer}
\end{align}
If there exists some $\widehat{k}$ such that $x^{\widehat{k}}\in F$,  we then see from \eqref{alg_iter} and $\sum_{i=1}^m w_i^k = 1$ that $x^k = x^{\widehat{k}}$ for all $k\ge \widehat{k}$.
In this case, $\{x^k\}$ converges to some $x^*\in F =\cap_{i=1}^m\F\,T_i$ finitely. Next, we consider the case that $x^k\notin F$ holds for all $k\in \N$. 

By assumption (a) and  $\{x^k\}\subseteq B$, we know that for  all $k\in\N$,
\begin{equation}\label{assp_a}
    \dist^2(x^{ks},\,F) \le \psi_B^2\Big(\sqrt{\max_{1\le i\le m}\|x^{ks} - T_i(x^{ks})\|^2}\Big).
\end{equation}
Fix any $t\in\{1,\ldots,m\}$. By assumption (b),  for any $k$ there exists some $t_k\in\{ks,\ldots,(k+1)s -1\}$ such that $t\in I_+(t_k)$. Thus, we have from the nonexpansiveness of $T_t$ that
\begin{equation}\label{main_ineq}
\begin{split}
\|x^{ks} - T_t(x^{ks})\|^2 & \le \left(\|x^{ks} - x^{t_k}\| + \|x^{t_k} - T_t(x^{t_k})\|  + \|T_t(x^{t_k}) - T_t(x^{ks})\|\right)^2 \\
& \le \left( \|x^{t_k} - T_t(x^{t_k})\|  + 2\|x^{ks} - x^{t_k}\|\right)^2\\
     & \overset{(a)}{\le} \Big(\|x^{t_k} - T_t(x^{t_k})\| + 2\sum_{j = ks}^{t_k - 1}\|x^j - x^{j+1}\|\Big)^2\\
     & \overset{(b)}{\le} 2\|x^{t_k} - T_t(x^{t_k})\|^2 + 8(t_k - ks)\sum_{j = ks}^{t_k-1}\|x^j - x^{j+1}\|^2\\
     & \le 2\|x^{t_k} - T_t(x^{t_k})\|^2 + 8s\sum_{j = ks}^{(k+1)s-1}\|x^j - x^{j+1}\|^2,
\end{split}
\end{equation}
where $(a)$ follows from repeated applications of the triangle inequality. For $(b)$, we consider two cases. If $t_k = ks$, then $(b)$ holds. For $t_k > ks$, we use the convexity of the square function so that $(a_1 + \cdots + a_r)^2 \leq r\sum _{i=1}^ra_i^2$ holds for arbitrary $a_i \in \R$. 
This inequality is first applied with $r \coloneqq 2$, $a_1 \coloneqq 2\|x^{t_k} - T_t(x^{t_k})\|^2$, $a_2$ being the remaining sum and, then, it is applied once more with $r \coloneqq t_k - ks$ to bound the remaining terms.

Next, we bound the two terms in the right-hand side of the last inequality of \eqref{main_ineq}. First, since each $T_i$ is $\alpha$-averaged, we have from Lemma~\ref{lm_alpha_av} that for all $x\in\E$ and $y\in \F\,T_i$,
\begin{equation*}
    \|T_i(x) - y\|^2 + \frac{1 - \alpha}{\alpha}\big\|(I - T_i)(x) \big\|^2 \le \|x - y\|^2
\end{equation*}
holds\footnote{Each $T_i$ may be in fact $\alpha_i$-averaged with a different $\alpha_i \in (0,1)$, but without loss of generality we may assume that the inequality holds with $\alpha = \max_{1\leq i \leq m}\{\alpha_i\}$.}. Therefore, for all $r\in\N$, $x\in\E$ and $y\in F$,
\begin{equation}\label{key_ineq}
    \begin{split}
       \Big\|\sum_{i=1}^mw_i^rT_i(x) - y\Big\|^2 =  \Big\|\sum_{i=1}^mw_i^r\left(T_i(x) - y\right)\Big\|^2 &\le \sum_{i=1}^mw_i^r\big\|T_i(x) - y\big\|^2\\
      &  \le \|x - y\|^2 - \frac{1 - \alpha}{\alpha}\sum_{i=1}^mw_i^r\|x - T_i(x)\|^2.
    \end{split}
\end{equation}
Some extra algebraic acrobatics leads to
\begin{equation}\label{ieq_1}
\begin{split}
    \frac{\nu(1 - \alpha)}{\alpha}\big\|x^{t_k} - T_t(x^{t_k})\big\|^2 & \overset{(a)}{\le} \frac{1 - \alpha}{\alpha}\sum_{i=1}^mw_i^{t_k} \big\|x^{t_k} - T_i(x^{t_k})\big\|^2 \\
   & \overset{(b)}{\le} \big\|x^{t_k} - P_F(x^{ks})\big\|^2 - \Big\|\sum_{i=1}^mw_i^{t_k}T_i(x^{t_k}) - P_F(x^{ks})\Big\|^2\\
   & = \big\|x^{t_k} - P_F(x^{ks})\big\|^2 - \big\|x^{t_k+1} - P_F(x^{ks})\big\|^2\\
   & \overset{(c)}{\le} \big\|x^{ks} - P_F(x^{ks})\big\|^2 - \big\|x^{(k+1)s} - P_F(x^{ks})\big\|^2\\
   & \le \dist^2(x^{ks},\,F) - \dist^2(x^{(k+1)s},\,F), 
\end{split}
\end{equation}
where $(a)$ follows from $t\in I_+(t_k)$ and the definition of $\nu$ which implies that $w_{t}^{t_k} \geq \nu$, (b) follows from \eqref{key_ineq} by letting 
$r \coloneqq  t_k$, $x \coloneqq  x^{t_k}$ and $y \coloneqq  P_F(x^{ks})$ in  \eqref{key_ineq}.
Then, $(c)$ follows from $ks \leq t_k$, $t_k+1 \leq (k+1)s$ and two applications of \eqref{eq:fejer2}: first with $x^{t_k}, x^{ks}$ and $P_F(x^{ks})$ and second with $x^{(k+1)s},x^{t_k+1}$ and $P_F(x^{ks})$.

For each $j\in\{ks,\ldots,(k+1)s - 1\}$,  we let $r\coloneqq j$, $x \coloneqq x^j$ and $y \coloneqq P_F(x^{ks})$ in \eqref{key_ineq}, and obtain
\begin{align*}
    \|x^j - x^{j+1}\|^2 = \Big\|x^j - \sum_{i=1}^mw_i^jT_i(x^j)\Big\|^2 &\le \sum_{i=1}^mw_i^j\|x^j -T_i(x^j)\|^2 \\ &\le \frac{\alpha}{1-\alpha}\big(\|x^j - P_F(x^{ks})\|^2 - \|x^{j+1} - P_F(x^{ks})\|^2\big).
\end{align*}
This further implies that
\begin{equation}\label{ieq_2}
\begin{split}
   \sum_{j = ks}^{(k+1)s-1}\|x^j - x^{j+1}\|^2 &\le \frac{\alpha}{1 - \alpha}\big(\|x^{ks} - P_F(x^{ks})\|^2 - \|x^{(k+1)s} - P_F(x^{ks})\|^2\big)\\
   & \le \frac{\alpha}{1 - \alpha}\Big(\dist^2(x^{ks},\,F) - \dist^2(x^{(k+1)s},\,F)\Big).
\end{split}
\end{equation}
Let $\triangle_k: = \dist^2\left(x^{ks},\,F\right) - \dist^2\left(x^{(k+1)s},\,F\right)$. Then we have from \eqref{eq:fejer} that $\triangle_k \le \widehat{a}$. 
Now,  we combine \eqref{assp_a}, \eqref{main_ineq}, \eqref{ieq_1}, \eqref{ieq_2} and the arbitrariness of $t\in\{1,\ldots,m\}$ to obtain
\begin{equation}\label{recu}
\begin{aligned}
\dist^2(x^{ks},\,F) & \leq \psi_B^2\Big(\max_{1\le t\le m}\sqrt{\|x^{ks} - T_t(x^{ks})\|^2}\Big) \\
& \leq \psi_B^2\Big(\max_{1\le t\le m}\sqrt{ 2\|x^{t_k} - T_t(x^{t_k})\|^2 + \frac{8s\alpha\triangle_k}{1-\alpha}}\Big) \\
& \leq \psi_B^2\Big(\sqrt{ \frac{2\alpha \triangle_k}{\nu(1-\alpha)}+ \frac{8s\alpha\triangle_k}{1-\alpha}}\Big) \\
& \le \psi_B^2\Big(\sqrt{\frac{2\alpha(1 + 4\nu s)}{\nu(1-\alpha)}\triangle_k}\Big) = \widehat{\phi}(\triangle_k) = \phi(\triangle_k).
\end{aligned}
\end{equation}
We see from \eqref{eq:fejer} and the nonnegativity of $\dist(x^k,\,F)$ that the sequence $\{\dist(x^k,\,F)\}$ converges to some $c^*\ge 0$. Letting $k\to\infty$ on both sides of \eqref{recu}, recalling $\lim_{x\to 0_+}\phi(x) = 0$ (due to $\lim_{x\to 0_+}\psi_B(x) = 0$) we have $\dist(x^k,\,F) \to c^*= 0$. Since $\{x^k\}$ is bounded, there exists a subsequence $\{x^{k_i}\}$ which converges to some point $x^*\in\E$. Therefore, $\dist(x^{k_i},\,F) \to 0$ together with the closedness of $F$ implies that $x^*\in F$. We note from the Fej\'{e}r monotonicity of $\{x^k\}$ with respect to $F$ (or \eqref{eq:fejer2}) that the nonnegative sequence $\{\|x^k - x^*\|\}$ is nonincreasing and thus convergent. This together with $\|x^{k_i} - x^*\|\to 0$  implies that $\{x^k\}$ converges to $x^*\in F = \bigcap_{i=1}^m\F\,T_i$.

We note that $\lim_{x\to 0_+}\phi(x) = 0$ and $\phi$ is  nondecreasing, thanks to the same properties of the regularity function $\psi_B$. Consequently,
\begin{equation}\label{iter_dec_one}
\begin{split}
& \Phi_{\phi}\left(\dist^2\big(x^{(k+1)s},\,F\big)\right) - \Phi_{\phi}\left(\dist^2\left(x^{ks},\,F\right)\right)   \\
& = \int_{\dist^2\left(x^{(k+1)s},\,F\right)}^{\dist^2\left(x^{ks},\,F\right)}\frac{1}{\phi^{-}(t)}dt \ge \frac{\triangle_k}{\phi^{-}\left(\dist^2\left(x^{ks},\,F\right)\right)} \ge 1,
\end{split}
\end{equation}
where the first inequality follows from the monotonicity of $\phi^{-}$ and the second inequality follows from \eqref{recu} and Lemma~\ref{inv_lemma}.
Moreover, for any $\ell > 0$, summing both sides of \eqref{iter_dec_one} for $k = 0,\ldots, \ell - l$, we obtain
\begin{equation*}
    \Phi_{\phi}\left(\dist^2\left(x^{\ell s},\,F\right)\right) - \Phi_{\phi}\left(\dist^2\left(x^{0},\,F\right)\right)  \ge \ell.
\end{equation*}
This together with the monotonicity of $\Phi_{\phi}$ and the Fej\'{e}r monotonicity of $\{x^k\}$ further implies that for any $k\in\N$,
\begin{equation*}
\begin{split}
\Phi_{\phi}\left(\dist^2\left(x^{k},\,F\right)\right) & \ge \Phi_{\phi}\left(\dist^2\big(x^{\lfloor{k/s}\rfloor \cdot s},\,F\big)\right)  \ge \Phi_{\phi}\left(\dist^2\left(x^{0},\,F\right)\right) + \lfloor{k/s}\rfloor.
\end{split} 
\end{equation*}
Notice from Lemma~\ref{pf_inv} that $\Phi_{\phi}$ is continuous and  decreasing, and therefore the usual inverse $\Phi_{\phi}^{-1}$ exists. Also,
 $\Phi_{\phi}\left(\dist^2\left(x^{0},\,F\right)\right) + \lfloor{k/s}\rfloor$ is in the interval 
$[\Phi_{\phi}\left(\dist^2\left(x^{0},\,F\right)\right),\,\Phi_{\phi}\left(\dist^2\left(x^{k},\,F\right)\right)]$, so continuity implies that it is also in the domain of 
$\Phi_{\phi}^{-1}$.
Then, applying $\Phi_{\phi}^{-1}$ to the above inequality and rearranging the terms finally 
leads to \eqref{gener_rate}. This completes the proof.
\end{proof}

\begin{remark}
In Theorem~\ref{theo:rate}, the convergence rate of the quasi-cyclic algorithm \eqref{alg_iter} is measured using the distance of the iterates to the common fixed point set $\dist(x^k,\,F)$. 
We note, however, that since $\{x^k\}$ is Fej\'er monotone, we have  $\|x^k - x^*\| \le 2\,\dist(x^k,\,F)$ from \cite[Theorem~3.3(iv)]{BB1993}. Consequently, a convergence rate  in terms of $\|x^k - x^*\|$, as was used in \cite[Theorem~3.1]{BLT17}, can also be obtained directly. 
\end{remark}

\subsection{Rates based on the index of regular variation}
\label{sec:ind}
Theorem~\ref{theo:rate} is a general result on convergence rates. However, typically 
we would like to obtain more concrete results and say, for example, whether 
the rate is linear, sublinear and etc. 
A direct application of Theorem~\ref{theo:rate} would require one to 
compute the function $\Phi_{\phi}$ and its inverse, which can be both highly nontrivial and 
devoid of closed forms.

In this paper, we will show two techniques for obtaining ``concrete'' convergence rates and avoid the direct computation $\Phi_{\phi}$ and $\Phi_{\phi}^{-1}$.
The first, presented in this subsection, is based solely on the index of regular variation of 
the error bound function $\psi_B$. We note that if we are given $\psi_B$, computing the index of regular variation is a significantly easier task, since it is just a limit computation as in Definition~\ref{def:rv}.
However, even if we have access to $\psi_B$, computing $\Phi_{\phi}^{-1}$ can be highly nontrivial.

Before we proceed we need to review more tools from regular variation.
First, we need a result that is a part of \emph{Karamata's theorem}, which tells us about the behavior of regular varying functions under taking integrals. Suppose that
$f:[a,\,\infty) \to (0,\,\infty) \in \RVp{\rho}$ is locally bounded. If 
$\sigma \geq -(\rho +1)$, then 
\begin{equation}\label{eq:karamata}
\lim _{x \to \infty} \frac{x^{\sigma+1}f(x)}{\int_{a}^x t^{\sigma}f(t)dt } = \sigma + 1 + \rho,
\end{equation}
see \cite[Theorem~1.5.11]{BGT87}. Karamata's theorem is a crucial result of this  body of theory and the case $\sigma = 0$ and $\rho > -1$ represents a remarkable property of regularly varying functions: as far as the behavior at infinity is concerned, functions in $\RV_{\rho}$ behave as polynomials of degree $\rho$ in that $\int _{a}^x f(t)dt$ is asymptotically equivalent  to $\frac{x}{\rho+1}f(x)$.  
This foreshadows why this will be useful for us and immediately suggests how to bypass the computation of the hard integral that appears in Theorem~\ref{theo:rate}.

Next, suppose that $f:(0,\,a] \to (0,\,\infty) \in \RVzp{\rho}$ is locally bounded away from zero. Then $f(1/\cdot)\in\RV_{-\rho}$ is locally bounded on $[\widehat{a},\,\infty)$ for $\widehat{a} \coloneqq 1/a$.
In order to derive analogous statements for regular variation at $0$, we make the substitution $x = 1/y$ in \eqref{eq:karamata},  and change the variable inside the integrals.
Recalling that $f(1/\cdot) \in \RVp{-\rho}$, we conclude that if 
$\sigma \geq -1+\rho$, then 
\begin{equation}\label{eq:karamata0}
\lim _{y \to 0_+} \frac{y^{-(\sigma+1)}f(y)}{\int_{y}^{a} t^{-\sigma-2}f(t)dt } = \sigma + 1 - \rho.
\end{equation}
We note that $y^{-(\sigma+1)}f(y) \in \RVz_{\rho-\sigma-1}$ and that  \eqref{asym_eq} and  \eqref{eq:karamata0} imply that if $\sigma+1-\rho > 0$, then $g$ defined by $g(y) \coloneqq {\int_{y}^{a} t^{-\sigma-2}f(t)dt } $ belongs to $\RVz_{\rho-\sigma-1}$ as well. When $\sigma = -2$, we have the following important special case:
\begin{equation}\label{eq:karamata0_3}
f \in \RVzp{\rho}\,\, \text{ and } -1 > \rho \quad \Rightarrow\quad {\int_{y}^a f(t)dt } \overset{c}{\sim} yf(y) \in \RVzp{\rho+1}.
\end{equation}
A similar result holds for $\sigma = -2$, $\rho = -1$. Let $\ell(s) \coloneqq f(1/s)s^{-1}$, so that $\ell \in \RVp{0}$, i.e., is a function of slow variation. Then, since $f$ is measurable and locally bounded away from zero, $\ell$ is measurable, locally bounded over $[1/a,\,\infty)$. In particular, $\ell$ is locally integrable over $[1/a,\,\infty)$ and we can invoke \cite[Proposition~1.5.9a]{BGT87} to conclude 
that ${\int_{1/a}^{x} s^{-1} \ell(s) ds } \in \RVp{0}$. 
Observe that  
\begin{equation}\label{eq:karamata_aux}
{\int_{1/x}^{a}  f(t)dt } = {\int_{1/a}^{x} s^{-2}f(1/s)ds } = {\int_{1/a}^{x} s^{-1} \ell(s) ds}.
\end{equation}
Therefore,  as a function of $x$, 
the left-hand-side of \eqref{eq:karamata_aux} belongs to $\RVp{0}$. Performing the substitution $x = 1/y$, we obtain the following implication:
\begin{equation}\label{eq:karamata0_4}
f \in \RVzp{-1}  \quad \Rightarrow\quad {\int_{y}^a f(t)dt } \in \RVzp{0}.
\end{equation}
Finally,  we need a result on the integral of rapidly varying functions. 
Suppose that $f:[a,\infty) \to (0,\,\infty)$ belongs to $\KRV$ and is locally bounded, then
\begin{equation}\label{eq:rav_karamata}
\lim _{x\to \infty} \frac{f(x)}{\int _{a}^x f(t)dt/t} \to \infty,
\end{equation}
see \cite[Proposition~2.6.9]{BGT87}\footnote{Strictly speaking, Proposition~2.6.9 is about function classes $BD$ and $MR_{\infty}$ which are defined in Chapter~2 therein. However, $\KRV$ functions are both $BD$ and $MR_{\infty}$ by \cite[Equation~(2.4.3) and Proposition~2.4.4]{BGT87}.}.
With that, we have the following proposition, see also \cite[Corollary~1.1]{ED13} for a related result.
\begin{proposition}\label{prop:int_rav}
	Let $f:[a,\,\infty) \to (0,\,\infty) \in \KRV$ (where $a > 0$) be a  nondecreasing function, then 
	\begin{equation*}
	 \int _{a}^x f(t) dt \in\KRV.
	\end{equation*}
	
\end{proposition}
\begin{proof}
	Let $\lambda > 1$ and let 
	$F(x)  \coloneqq \int _{a}^x f(t) dt$, ($x > a$). Then
	\begin{equation*}
	F(\lambda x) - F(x) = \int _{x}^{\lambda x} f(t) dt \geq (\lambda x - x) f(x),
	\end{equation*}
	where the inequality follows from the monotonicity of $f$. Dividing by $F(x)$ and readjusting the terms, we have
	\begin{equation}\label{eq:int_rav_aux}
	\frac{F(\lambda x)}{F(x)} \geq 1+ (\lambda-1)\frac{x f(x)}{F(x)}
	\end{equation}
	Since $1/t \geq 1/x$ for $t \in [a,\,x]$, we have${\int _{a}^x f(t)dt/t}\geq {\int _{a}^x f(t)dt/x}  = \frac{F(x)}{x} $ and therefore,
	\begin{equation}\label{eq:aux_kr}
	\frac{x f(x)}{F(x)} \geq \frac{f(x)}{\int _{a}^x f(t)dt/t}.
	\end{equation}
	By \eqref{eq:rav_karamata} the right-hand-side of \eqref{eq:aux_kr} goes to $\infty$ as $x \to \infty$. Thus, the right-hand-side of \eqref{eq:int_rav_aux} goes to $\infty$ as $x \to \infty$.
	Consequently, $F(\lambda x)/F(x)$ goes to $\infty$ when $x \to \infty$ and $\lambda > 1$. For $\lambda < 1$, 
	we have $1/\lambda > 1$ and  $\lim _{x \to \infty} F(\lambda x)/F(x) = \lim _{y \to \infty} F(y)/F(y/\lambda) = 0$ by what was just proved. This shows that $F \in \RVp{\infty}$.
	Since $F$ is the integral of a positive monotone function, it must be  nondecreasing as well\footnote{As a reminder, $F(y)- F(x) = \int _{x}^y f(t) dt \geq f(x)(y-x) > 0$, if $y > x$.}, so, in fact, 
	$F \in \KRV$ by \eqref{eq:krv_mon}.
\end{proof}

%

We now have all pieces for our first results on the asymptotic properties of $\Phi_f$.

\begin{theorem}[Index of regular variation and asymptotic behavior of $\Phi_f$]\label{theo:rv}
	Let $f:(0,\,a] \to (0,\,\infty)$ be a  nondecreasing function in $\RV^0_{\rho}$ for 
	$\rho \in [0,\,1]$ such that $\lim _{t \to 0_+}f(t) = 0$ holds and consider the function $\Phi_f$ in \eqref{pf_f}.
Then the	following statements hold.
	\begin{enumerate}
		\item[{\rm (i)}] If $\rho = 0$, then $\Phi_{f} \in \RVzp{-\infty}$. In particular,  $\Phi_{f}^{-1} \in \RVp{0}$ and  for every $r > 0$, 
		$x^{-r} = o(\Phi_{\phi}^{-1}(x))$ as $x \to \infty$.
		\item[{\rm (ii)}]  If $\rho \in (0,\,1)$, then $\Phi_{f} \in \RVzp{1-1/\rho}$. In particular, $\Phi_{f}^{-1} \in \RVp{\rho/(\rho-1)}$ and 
		$\Phi_{f}^{-1}(x) = o(x^{-r})$ as $x \to \infty$ for every $0<r<-\rho/(\rho-1)$.
		\item[{\rm (iii)}]  If $\rho = 1$ and there exists some $c > 0$ such that $f(x)\ge cx$  as $x\to 0_+$,  then $\Phi_{f} \in \RVzp{0}$. 
		In particular, $\Phi_{f}^{-1} \in \RVp{-\infty}$ and for 
		every $r > 0$, $\Phi_{f}^{-1}(x) = o(x^{-r})$ as $x \to \infty$.
	\end{enumerate}
\end{theorem}
\begin{proof}
	Let $F : [1,\,\infty)\to (0,\,\infty)$ be such that
	\begin{equation}\label{eq:F}
	F(x) \coloneqq \Phi_f(1/x).
	\end{equation}
	By Lemma~\ref{pf_inv},  $F$ is  increasing and continuous, and thus locally bounded. Furthermore,  $F^{-1}(y) = 1/(\Phi_{f}^{-1}(y))$ holds for $y \in [1,\,\infty)$.
	
	We start with item~(i), where $\rho = 0$.
	Let $g \coloneqq 1/f(1/\cdot)$. Then, we see from \eqref{eq:rv_rvz} and \eqref{eq:rv_calc} that $g \in \RVp{0}$. Due to the monotonicity of $f$, both $f$ and $g$ are locally bounded. Moreover,  $g(t)$ goes to $\infty$ as 
	$t \to \infty$. By \eqref{eq:rav_inv}, 
	$g^{\leftarrow} \in \KRV$ and by 
	\eqref{eq:rvz_inv_def}  we have
	\begin{equation*}
	f^{-}(x) = \frac{1}{g^{\leftarrow}(1/x)}.
	\end{equation*}
	We note that $F$ can be written as follows
	\begin{equation*}
	F(y) = \int _{1/y}^\delta \frac{1}{f^{-}(s)}   ds  = \int _{1/\delta}^y \frac{t^{-2}}{f^{-}(1/t)  } dt =  \int _{1/\delta}^y t^{-2}g^{\leftarrow}(t) dt,
	\end{equation*}
	where the second equality is obtained by the substitution $s= t^{-1}$.
	Then, 	by \eqref{eq:krv_power}, $(\cdot)^{-2}g^{\leftarrow}(\cdot)$ belongs to $\KRV   $
	and invoking Proposition~\ref{prop:int_rav}, we conclude that $F$ belongs to $\KRV$
	which implies that $\Phi_{f} \in \RVzp{-\infty}$ by  \eqref{eq:rv_rvz} and \eqref{eq:F}.
	This proves the first half of item~(i).

	Due to \eqref{eq:potter_rapid2} and \eqref{eq:krv}, $F(x)$ goes to $\infty$ as $x\to\infty$. 
	Since $F$ belongs to $\KRV$ we further conclude that $F^{\leftarrow} \in \RVp{0}$ by \eqref{eq:rav_inv2}. Since $F^{\leftarrow}(y) = F^{-1}(y)$ holds for large $y$ by Proposition~\ref{lb_proposition}, 
	we see that $\Phi_{f}^{-1} \in \RVp{0} $ by  \eqref{eq:rv_calc}.
	
	Next, let $r > 0$ and let $h$ be the function given 
	by $h(x) = x^{-r}/\Phi_{f}^{-1}(x) $. By the calculus rules in \eqref{eq:rv_calc}, $h \in \RVp{-r}$.
	Since the index of $h$ is negative,  $h(x)$ goes to $0$ as $x \to \infty$ by  \eqref{eq:rv_inf2}. That is, 
	$x^{-r} = o(\Phi_{f}^{-1}(x))$ as $x \to \infty$. This concludes the proof of item~(i).

	We now move on to the case $\rho \in (0,\,1]$. By \eqref{eq:rvz_inv}, $1/f^{-}$ is regularly varying at $0$ with index $-1/\rho$. We note $1/f^{-}$ is monotone by Lemma~\ref{inv_lemma} and hence locally bounded away from zero.
	Since $-1 \geq -1/\rho$ and $\Phi_f(x) =  \int _{x}^\delta \frac{1}{f^{-}(s)} ds$, 
	we have 	$\Phi_{f} \in \RVzp{1-1/\rho}$, which follows from	\eqref{eq:karamata0_3} and \eqref{eq:karamata0_4}. This proves the first halves of items~(ii) and (iii).
	
	By \eqref{eq:F} and \eqref{eq:rv_rvz}, $F \in \RVp{1/\rho - 1}$.
	We now verify that $F \to \infty$ as $x \to \infty$.
	If $\rho \in (0,\,1)$ it follows from \eqref{eq:rv_inf} and 
	if $\rho = 1$ it follows from the assumption 
	in item~(ii) and Lemma~\ref{pf_inv}~(b). 
	The conclusion is that in either case, we have $F^{\leftarrow}(y) = F^{-1}(y)$ for large $y$, by Proposition~\ref{lb_proposition}.
	By the calculus rule  for the inverse in \eqref{eq:rv_inv} and \eqref{eq:rav_inv}, together with \eqref{eq:rv_calc} and the definition of $\RV_{-\infty}$ in Definition~\ref{def:rv}, we conclude that 
	$\Phi_{f}^{-1} = 1/F^{-1} $ belongs to $\RVp{\rho/{(\rho-1)}}$ if $\rho \in (0,\,1)$ and to $\RVp{-\infty}$ if $\rho = 1$.
	
	Suppose that $\rho \in (0,\,1)$ and let $r $
	be such that $0 < r < -\rho/(\rho-1)$. Let $\tau > 0$ 
	be such that $r+\tau <  -\rho/(\rho-1) $. The index of $\Phi_{f}^{-1}$ is $\rho/(\rho-1)$ and we apply Potter's bounds (see \eqref{eq:potter_rv}) to
	$\Phi_{f}^{-1}$ with $\epsilon = -\rho/(\rho-1)-r-\tau$. Fixing $y$, we see that for large $x$, we have
	\begin{equation*}
	\Phi_{f}^{-1}(x) \leq \hat{A} x^{-r-\tau},
	\end{equation*}
	where $\hat{A}$ is some fixed constant. This implies that 
	$\Phi_{f}^{-1}(x)/ (x^{-r}) \leq \hat{A} x^{-\tau}$ and that the quotient goes to $0$ as $x\to \infty$. Therefore, $\Phi_{f}^{-1}(x) = o(x^{-r})$. This concludes the proof of item~(ii).
	
	Next, suppose that $\rho = 1$ and let $r > 0$. We apply 
	\eqref{eq:potter_rapid} to $\Phi_{f}^{-1}$.
	Since \eqref{eq:potter_rapid}  is valid for all $r > 0$, it is valid in particular for $2r$. So for sufficiently large $x$, we have $\Phi_{f}^{-1}(x)/{x^{-r}}\leq x^{-r}$, which implies that the quotient goes to $0$ as $x \to \infty$.
	Therefore, $\Phi_{f}^{-1}(x) = o(x^{-r})$ and this proves item~(iii).
	This completes the proof.
\end{proof}

Using Theorem~\ref{theo:rv} we can analyze the quasi-cyclic iteration as follows.

\begin{theorem}[Index of regular variation and convergence rates]\label{theo:rv_old}
	Under the setting of Theorem~\ref{theo:rate}, let $\rho$ denote the index of $\psi_B|_{(0,\,a]}$ and suppose that the convergence is not finite.
	Then	 the following statements hold.
	\begin{enumerate}
		\item[{\rm (i)}] If $\rho = 0$, then $\Phi_{\phi} \in \RVzp{-\infty}$. In particular,  $\Phi_{\phi}^{-1} \in \RVp{0}$ and  for every $r > 0$, 
		$k^{-r} = o(\Phi_{\phi}^{-1}(k))$ as $k \to \infty$.
		\item[{\rm (ii)}]  If $\rho \in (0,\,1)$, then $\Phi_{\phi} \in \RVzp{1-1/\rho}$. In particular, $\Phi_{\phi}^{-1} \in \RVp{\rho/(\rho-1)}$ and 
		$\Phi_{\phi}^{-1}(k) = o(k^{-r})$ as $k \to \infty$ for every $0<r<-\rho/(\rho-1)$.
		\item[{\rm (iii)}]  Suppose that $\rho = 1$  and $B$ is closed and connected. 
		Then, $\Phi_{\phi} \in \RVzp{0}$. 
		In particular, $\Phi_{\phi}^{-1} \in \RVp{-\infty}$ and for 
		every $r > 0$, $\Phi_{\phi}^{-1}(k) = o(k^{-r})$ as $k \to \infty$.
	\end{enumerate}
\end{theorem}
\begin{proof}
	By the calculus rules in \eqref{eq:rvz_calc}, we see from $\psi_B|_{(0,\,a]}\in\RVz_{\rho}$ and the definition of $\phi$ in \eqref{def_pf} that $\phi$ has index $\rho$. Moreover, $\phi$ is  nondecreasing with $\lim_{t\to0_+}\phi(t) = 0$, thanks to the monotonicity of $\psi_B$ and $\lim_{t\to0_+}\psi_B(t) = 0$ in Definition~\ref{def_jcrv}.
	With that, items (i), (ii)  are a consequence of Theorem~\ref{theo:rv} applied to $f \coloneqq \phi$.

	Item~(iii) also follows from Theorem~\ref{theo:rv}, but we need to check that the assumption that there exists $c > 0$ such that $\phi(t)\ge ct$ holds as $t$ goes to $0_+$. Let $d(y) \coloneqq \max_{1\leq i \leq m}{\norm{y-T_i(y)}}$. 
Fixing any $y\in B$, we let $z \coloneqq P_{F}(y)$ and note from $F:=\bigcap_{i=1}^m\F\,T_i$ that $z = T_i(z)$ holds for all $i\in\{1,\ldots,m\}$, which together with the nonexpansiveness of $T_i$  (see Lemma~\ref{lm_alpha_av}) implies that for all $i$,
	\begin{equation*}
	\norm{y-T_i(y)} = \norm{y - z +z - T_i(y)} \leq 2\norm{y-z} = 2\,\dist(y,\,F).
	\end{equation*}
	This together with the arbitrariness of $i$ and the definition of $\psi_B$ further implies that
	\[
 d(y)/2 \leq \dist(y,\,F) \leq \psi_B(d(y)), \ \forall\ y \in B.
	\] 
	By assumption, $B$ is closed, so it contains the limiting point $x^*$ which satisfies $d(x^*) = 0$. Also because the convergence is assumed to be not finite, no $x^k$ can be a common fixed point. In particular, since $B$ contains the sequence $\{x^k\}$, there exists at least another $\bar{y} \in B$ with $d(\bar{y}) > 0$. 	Now, $d(\cdot)$ is a continuous function and $B$ is connected so $d(B)$ is a connected subset of $\R$, i.e., $d(B)$ is an interval. In particular, over $B$, $d$ assumes all values between $0$ and $d(\bar{y})$. This tells us that 
	\[
	t/2 \leq \psi_B(t)
	\]
	holds for all sufficiently small $t$. Letting $\kappa \coloneqq \sqrt{2\alpha(1+4\nu s)/(\nu(1-\alpha))}$ and recalling the definition of $\phi$, we have
	\[
	\sqrt{\phi(t)} = \psi_B(\kappa \sqrt{t}) \geq \kappa \sqrt{t}/2,
	\]
	for small $t$. This implies that $\phi(t) \geq \kappa^2 t/4$ holds for small $t$, which is what we wanted to show.
\end{proof}
The three cases in Theorem~{\ref{theo:rv_old}} can be interpreted as follows.
When $\rho = 1$, the convergence of the quasi-cyclic algorithm is \emph{almost linear}, which means that, for any $r$, the iterates of algorithm converges to a fixed point faster than $k^{-r}$ goes to $0$ as $k \to \infty$. This includes cases where the convergence is, in fact, linear but it also includes the possibility that the convergence is slower than linear (i.e., slower than $c^{-k}$ for any $c$) as observed empirically in \cite[Figure~1]{LL20}.
When $\rho  \in (0,\,1)$, the convergence rate is at least sublinear and is faster than $s^{-r}$ for all $r$ with $0 < r < -\rho/2(\rho-1)$, which it means that the converge rate is \emph{almost} the rate that would be afforded if $\psi_B$ were ``purely H\"olderian'' of the form  $t^\rho$ as we will discuss in Section~\ref{sec:previous}.   

The case $\rho = 0$ is the least informative, which only tells us that $\Phi_{\phi}^{-1}$ is going to $0$ slower than any sublinear rate, where we recall that $\Phi_{\phi}^{-1}$ gives an \emph{upper bound} on the true convergence rate. So, in essence, we are getting a lower bound to an upper bound, which is only helpful if we suspect that the upper bound is actually tight. 
In order to get more information on the case $\rho = 0$, we need extra assumptions as in Theorem~\ref{theorem_int}.

\subsection{Tighter rates}
In this section, we discuss a tool to obtain tighter rates than the one described in 
Theorem~\ref{theo:rv}.
The drawback is that although we do not need to compute the integral appearing in 
\eqref{pf_f}, we still need to be able to say something about a certain function 
$g$ that will appear in Theorem~\ref{theorem_int}.
We start with the following lemma about preservation of asymptotic equivalence under 
the arrow inverse.

\begin{lemma}\label{lemma_inv_eq}
Suppose that $f_1:[a,\,\infty) \to (0,\,\infty)$ and $f_2:[b,\,\infty) \to (0,\,\infty)$ are measurable locally bounded functions such that $f_1\in\RV_{\rho}$ with $\rho > 0$. If $f_1(t) \sim f_2(t)$ as $t \to \infty$, then $f_1^{\leftarrow}(t)\sim f_2^{\leftarrow}(t)$ as $t \to \infty$. 
Similarly, if $f_1 \overset{c}{\sim} f_2$, then $f_1^{\leftarrow}\overset{c}{\sim} f_2^{\leftarrow}$.
\end{lemma}

\begin{proof}
Let $\ell(t) \coloneqq f_1(t)/t^{\rho}$, so that $\ell \in \RV_{0}$. By assumption, 
$f_2\sim f_1$ holds, so we have $f_2(t) \sim t^{\rho}\ell(t)$ as well. 
Then, \cite[Proposition~1.5.15]{BGT87} implies that $f_1^{\leftarrow}$ and $ f_2^{\leftarrow}$ both 
satisfy 
\[
f_1^{\leftarrow}(t) \sim t^{1/\rho}\ell^{\sharp}(t^{1/\rho}) \,\, \text{and}\,\, f_2^{\leftarrow}(t) \sim t^{1/\rho}\ell^{\sharp}(t^{1/\rho}), 
\]
for a certain function $\ell^{\sharp} \in \RV_{0}$, which leads to $f_1^{\leftarrow}{\sim} f_2^{\leftarrow}$.

Now, if $f_1 \overset{c}{\sim} f_2$, then there exists a positive constant $\mu$ such that 
$f_1 {\sim} \mu f_2$. By what we just proved, we have $f_1^{\leftarrow} \sim (\mu  f_2)^{\leftarrow}$ and, by the definition of the arrow inverse \eqref{eq:arrow}, $(\mu  f_2)^{\leftarrow}(t) = f_2^{\leftarrow}(t/\mu)$ holds. Since, $f_2^{\leftarrow} \in \RV_{1/\rho}$, the definition of regular variation implies $f_2^{\leftarrow}(t/\mu)/f_2^{\leftarrow}(t) \to (1/\mu)^{1/\rho}$ as $t \to \infty$ so that $f_2^{\leftarrow}(t/\mu) \overset{c}{\sim} f_2^{\leftarrow}(t)$ as $t\to \infty$.
We conclude that $f_1^{\leftarrow} \overset{c}{\sim} f_2^{\leftarrow}$ holds.
\end{proof}
If $f,\,g \in \RVz_{\rho}$ are measurable, locally bounded away from zero and satisfy $f \sim g$ then, letting $\bar f \coloneqq 1/f(1/\cdot)$ and $\bar g \coloneqq 1/g(1/\cdot)$, the functions $\bar f, \,\bar g$ are locally bounded, belong to $\RV_{\rho}$ and satisfy $\bar f \sim \bar g$.
Recalling \eqref{eq:rvz_inv_def} and applying Lemma~\ref{lemma_inv_eq} to 
$\bar f, \,\bar g$, we conclude that $f^{-} \sim g^{-}$.
With that, we conclude that, under the stated hypothesis, if $f,\,g:(0,\,a] \to (0,\,\infty)$ are such that $f,\,g \in \RVz_{\rho}$ with $\rho > 0$ then
\begin{equation}\label{eq:inv_equiv_minus}
f \overset{c}{\sim} g \Rightarrow f^{-} \overset{c}{\sim} g^{-}.
\end{equation}

The following lemma is inspired by \cite[Theorem~1]{dt07}, but  here we relax the monotonicity assumption imposed therein. 
\begin{lemma}\label{lem:inv_o}
	Suppose that $f_1:[a,\,\infty) \to (0,\,\infty)$ and $f_2:[b,\,\infty) \to (0,\,\infty)$ are locally bounded functions such that $f_1\in \RV_{\rho}$ with $\rho \geq 0$, $f_1(t) \to \infty$, $f_2(t) \to \infty$ and $f_1(t) = o(f_2(t))$ as $t \to \infty$, then $f_2^{\leftarrow}(t) = o(f_1^{\leftarrow}(t))$ as $t \to \infty$. 
\end{lemma}
\begin{proof}
Because 	$f_1(t) = o(f_2(t))$ holds, for every $\alpha \in (0,\,1)$, there exists $t_\alpha\geq \max\{a,\,b\}$ such that
\[
f_1(t) \leq \alpha f_2(t) ,\qquad \forall t \geq t_{\alpha}.
\]
Now, because $f_1$ is locally bounded,  $y_{\alpha} \coloneqq 1+ \sup \{ f_1(t) \mid a \leq t \leq t_{\alpha} \}$ is finite. In particular, if $f_1(t) > y$ holds with $y \geq y_\alpha$, we must have  $t > t_{\alpha}$, which further leads to $\alpha f_2(t) \geq f_1(t) > {y}$.
Consequently, for $y \geq y_\alpha$, the inequality $f_1(t) > y$ implies  $f_2(t) > \frac{y}{\alpha}$.
In view of the definition of the arrow inverse \eqref{eq:arrow}, we have that if $y \geq y_{\alpha}$,
then 
\[
\inf \left\{t \in [a,\,\infty) \mid f_1(t) > y \right\}= f_1^{\leftarrow}(y) \geq f_2^{\leftarrow}(y/\alpha) = \inf \left\{t \in [b,\,\infty) \mid f_2(t) > y/\alpha \right\}. 
\]
Put otherwise, $f_1^{\leftarrow}(\alpha u) \geq f_2^{\leftarrow}(u)$ holds for sufficiently large $u$ which leads to
\[
\limsup _{u \to \infty} \frac{f_2^{\leftarrow}(u)}{f_1^{\leftarrow}(u)} \leq \limsup _{u \to \infty} \frac{f_1^{\leftarrow}(\alpha u)}{f_1^{\leftarrow}(u)}.
\]
Next, we divide in two cases. 

\fbox{$\rho = 0$} Since $f_1 \in \RV_{0}$, we have $f_1 ^{\leftarrow}\in \RV_{\infty}$ (see \eqref{eq:rav_inv} and \eqref{eq:krv}) which implies that the $\limsup$ on the right-hand-side is zero for $\alpha \in (0,1)$ (see Definition~\ref{def:rv}).
That is, $f_2^{\leftarrow}(t) = o(f_1^{\leftarrow}(t))$ holds.

\fbox{$\rho > 0$} We have $f_1 ^{\leftarrow}\in \RV_{1/\rho}$ (see \eqref{eq:rv_inv}). By the definition of 
regular variation, we have
\[
\limsup _{u \to \infty} \frac{f_2^{\leftarrow}(u)}{f_1^{\leftarrow}(u)} \leq \limsup _{u \to \infty} \frac{f_1^{\leftarrow}(\alpha u)}{f_1^{\leftarrow}(u)} = \alpha^{1/\rho}
\]
and since $\alpha \in (0,\,1)$ is arbitrary, we conclude that the $\lim \sup$ on the left-hand-side is zero and $f_2^{\leftarrow}(t) = o(f_1^{\leftarrow}(t))$ holds.
\end{proof}

We are now positioned to state the main result of this subsection. In what follows, we use the notation 
$f(s) = 1/o(g(s))$ as $s \to \infty$ to indicate that $1/(f(s)g(s)) \to 0$ as $s \to \infty$.
 \begin{theorem}\label{theorem_int}
Suppose that $f:(0,\,a] \to (0,\,\infty) \in \RVzp{\rho}$ with $\rho\in[0,\,1]$  is  nondecreasing  and  $\lim _{x\to 0_+}f(x) = 0$.  Let $g(x) := \frac{1}{xf^{-}(1/x)}$ for $x\in[1/\delta,\,\infty)$ and $\Phi_f$ be defined as in \eqref{pf_f}.
Then $g$ is locally bounded. Moreover, the following statements hold.
\begin{enumerate}
    \item [{\rm (i)}] If $\rho = 0$ and $\ln(g)\in\RV_{q}$ with $q > 0$  then letting $\alpha \in \R$ and $\widehat{g}(x) \coloneqq x^{\alpha} g(x)$, we have $\Phi_f^{-1}(s)\sim \frac{1}{\widehat{g}^{\leftarrow}(s)}$ as $s\to\infty$. 
    
     \item [{\rm (ii)}] If $\rho\in(0,\,1)$, then $\Phi_f^{-1}(s)\sim \frac{1}{g^{\leftarrow}\big((1/\rho - 1)s\big)}$ as $s\to\infty$.
    
    \item [{\rm (iii)}] In case of  $\rho = 1$: if $f(t)\overset{c}{\sim} t$  as $t\to0_+$, then there exist $\tau_1,\,\tau_2 > 0$ and $0 < c_1 < c_2 < 1$ such that $\tau_1 c_1^s \le \Phi_f^{-1}(s) \le \tau_2 c_2^s$ whenever $s$ is large enough;  if  $t = o(f(t))$ as $t\to0_+$ then $\Phi_f^{-1}(s) = \frac{1}{o\left(g^{\leftarrow}(s)\right)}$ as $s\to\infty$.
          
\end{enumerate}
\end{theorem}

\begin{proof}
Initially, we need some bookkeeping to verify that the several functions involved in this theorem satisfy certain desirable properties.

First, we see from \eqref{eq:rv_rvz}, \eqref{eq:rv_calc}, \eqref{eq:rvz_inv} and \eqref{eq:ravz_inv} that $g\in\RV_{1/\rho - 1}$ when $\rho\in(0,\,1]$ and $g\in\RV_{\infty}$ when $\rho = 0$. 
Also, since local boundedness is preserved by taking products and  $1/x$ and $1/f^{-}(1/x)$ are locally bounded (both by monotonicity) over $[1/\delta,\, \infty)$, $g$ is also locally bounded. Over $[1/\delta,\,\infty)$ the functions $1/x$ and $1/f^{-}(1/x)$ are positive and monotone, so they are also locally bounded away from zero. So, similarly, $g$ is locally bounded away from zero over $[1/\delta,\, \infty)$.

We let $F(x) \coloneqq \Phi_f(1/x)$ and then have
\begin{equation}\label{int_form}
F(x) =\Phi_f(1/x) = \int_{1/x}^{\delta}\frac{1}{f^{-}(t)}dt = \int_{1/\delta}^x\frac{1}{s^2f^{-}(1/s)}ds = \int_{1/\delta}^{x}\frac{g(s)}{s}ds.
\end{equation}
Since $\lim_{x\to 0_+}f(x) = 0$,  we invoke Lemma~\ref{pf_inv} to conclude that  $\Phi_f$ is continuous and  decreasing. This together with  $F(\cdot) = \Phi_f(1/\cdot)$ implies that $F$ is continuous and  increasing.

We now consider \eqref{int_form} in three cases as follows.

\fbox{$\rho = 0$ and $\ln(g)\in\RV_{q}$ with $q > 0$.}  
In this case, first we observe that the same argument that showed that 
$g$ is locally bounded and locally bounded away from zero also applies to $\widehat g$.
This implies that $\ln(\widehat g)$ is locally bounded\footnote{Direct from definition. Over any compact set $C$, there are positive constants $M_1,M_2$ such that $M_1 \leq g(t) \leq M_2$ for $t \in C$. Therefore, $\ln(M_1) \leq \ln(g(t)) \leq \ln(M_2)$, so that $\abs{\ln(g(t))} \leq \max\{\abs{\ln(M_1)},\, \abs{\ln(M_2)} \}.$}.

Next, let $h(x) \coloneqq \ln(g(x)) - \ln(x)$.
Since $\ln(g) \in \RV_{q}$ with $q > 0$ holds, we have $ \ln(x)/\ln(g(x)) \to 0$ as
$x \to \infty$, which follows from \eqref{eq:rv_calc} and \eqref{eq:rv_inf2}.
Similarly, $\alpha \ln(x)/\ln(g(x)) \to 0$ as $x \to \infty$.
This implies that
\begin{equation}\label{eq:hlng}
h(x) \sim  \ln(g(x)) \sim \ln(\widehat g(x)) \ \ {\rm as} \ \ x\to\infty.
\end{equation}
 Since $\ln(g(x)) \to \infty $ as 
$x \to \infty$ (see \eqref{eq:rv_inf}), \eqref{eq:hlng} implies that  there exists 
$\widehat a > 1/\delta$ such that $h(x)$ is  positive over $[\widehat a, \, \infty)$
and the restriction of $h$  to $[\widehat a, \, \infty)$ belongs 
to $\RV_{q}$. 

Letting $b \coloneqq \int_{1/\delta}^{\widehat{a}}e^{h(s)}ds$ and recalling \eqref{int_form}, we apply \cite[Lemma~5.10]{LL20} to conclude that
\begin{align}\label{eq:fxb}
\ln(F(x)-b) = \ln\left(\int_{{1/\delta}}^{x}e^{h(s)}ds -b\right)=
\ln\left(\int_{\widehat{a}}^{x}e^{h(s)}ds \right)\sim h(x) \ \ {\rm as} \ \ x\to\infty.
\end{align}
This implies that $F(x)\to\infty$ as $x\to\infty$. Consequently, 
\begin{equation*}
\lim_{x\to\infty}\frac{\ln\left(F(x) - b\right)}{\ln(F(x))} = \lim_{x\to\infty}\frac{\ln(F(x)) + \ln\left(1 - b/F(x)\right)}{\ln(F(x))} = 1 + \lim_{x\to\infty}\frac{ \ln\left(1 - b/F(x)\right)}{\ln(F(x))} = 1
\end{equation*}
holds, namely, $\ln(F(x)-b) \sim \ln(F(x))$ holds.
From \eqref{eq:hlng} and \eqref{eq:fxb}  we obtain 
\begin{equation}\label{ln_prim}
\ln\left(F(x)\right) \sim \ln(F(x)-b) \sim h(x) \sim \ln(\widehat{g}(x))\ \ {\rm as} \ \ x\to\infty.
\end{equation}
Since $\ln(\widehat{g}(x))\to\infty$ as $x\to\infty$, the same is true of  $\ln(F(x))$ and there exists $\tilde{a} > 1/\delta$ such that both functions are positive on $[\tilde{a},\,\infty)$. 
Let 
\begin{equation}\label{restr_fun}
F_r(\cdot):= \ln(F(\cdot))|_{[\tilde{a},\,\infty)}, \ \ \ \ \widehat{g}_r(\cdot):= \ln(\widehat{g}(\cdot))|_{[\tilde{a},\,\infty)}.
\end{equation}
We then have from \eqref{ln_prim} and $\ln(\widehat{g})\in\RV_{q}$  that $F_r(x)\sim\widehat{g}_r(x)$ as $x\to\infty$ and $\widehat{g}_r\in\RV_{q}$ with $q > 0$. In addition, $F_r$ is locally bounded (due to the continuity of $F$), and $\widehat{g}_r$ is locally bounded, thanks to the local boundedness of $\ln(\widehat{g})$.  We apply Lemma~\ref{lemma_inv_eq} to the two functions in \eqref{restr_fun} and obtain
\begin{equation}\label{eq:fr}
 F_r^{\leftarrow}(s) \sim \widehat{g}_r^{\leftarrow}(s) \ \ {\rm as} \ \ x\to\infty.
\end{equation}
We recall that both $\ln(F(x))$ and $\ln(\widehat{g}(x))$ go to $\infty$ as $x\to\infty$ and $F_r$ is continuous and increasing. We apply Proposition~\ref{lb_proposition} to them and see from \eqref{restr_fun} and \eqref{eq:fr} that
\begin{equation*}
F^{-1}(e^s) = (\ln(F))^{-1}(s) = (\ln(F))^{\leftarrow}(s) = F_r^{\leftarrow}(s) \sim \widehat{g}_r^{\leftarrow}(s) = (\ln(\widehat{g}))^{\leftarrow}(s) = \widehat{g}^{\leftarrow}(e^s) \ \ {\rm as} \ \ s\to\infty,
\end{equation*}
where the first equation follows from the fact that $F$ is invertible and the last equation follows from the definition of arrow generalized inverse.
This further gives
\begin{equation*}
\Phi_f^{-1}(s) = 1/F^{-1}(s) \sim 1/\widehat{g}^{\leftarrow}(s) \ \ {\rm as} \ \  s\to\infty,
\end{equation*}
which proves item (i).

\fbox{$\rho\in(0,\,1)$}  In this case, we have $g\in\RV_{1/\rho - 1}$ with $1/\rho - 1 > 0$. 
Then, invoking \eqref{eq:karamata} with $g$ and $\sigma = -1$, we have from  \eqref{int_form} that
\begin{equation}\label{prim_eqv}
F(x) \sim \frac{\rho}{1 - \rho}g(x)  \ \ \ {\rm as}\ \ x\to\infty.
\end{equation}
Recalling that $g$ and $F$ are locally bounded,
applying Lemma~\ref{lemma_inv_eq} to \eqref{prim_eqv} and Proposition~\ref{lb_proposition}, we obtain
that 
\begin{equation*}
\Phi_f^{-1}(s) = 1/F^{-1}(s) \sim 1/\Big(\frac{\rho}{1 - \rho}g\Big)^{\leftarrow}(s) = 1/g^{\leftarrow}\big((1/\rho - 1)s\big)  \ \ {\rm as}\ \  s\to\infty,
\end{equation*}
which proves statement (ii).

\fbox{$\rho = 1$}   We first prove the first half of the statement. Since $f$ is nondecreasing, 
it is locally bounded away from zero, so $f(t) \overset{c}{\sim} t$ implies 
$f^{-}(t) \overset{c}{\sim} t$ as $t \to 0_+$.
In particular, 
there exists $\mu > 0$ such that $1/f^{-}(1/s) \sim s/\mu$ as $s\to\infty$.  This further implies
\begin{equation*}
\frac{g(s)}{s}  = \frac{1/f^{-}(1/s)}{s^2} \sim \frac{1}{\mu s}\ \ {\rm as}\ \  s\to\infty.
\end{equation*}
Therefore, fix any $\epsilon\in(0,\,1)$, there exists $M > 0$ such that for all $s \ge M$ we have
\begin{equation*}
\frac{1  - \epsilon}{\mu s} \le \frac{g(s)}{s} \le \frac{ 1 + \epsilon}{\mu s}.
\end{equation*}
This together with \eqref{int_form} implies that when $x \ge M$,
\begin{equation*}
F(x) = \int_{1/\delta}^{x}\frac{g(s)}{s}ds = \int_{1/\delta}^{M}\frac{g(s)}{s}ds + \int_{M}^{x}\frac{g(s)}{s}ds
\end{equation*}
Let $\kappa_1:= \int_{1/\delta}^{M}\frac{g(s)}{s}ds - (1 - \epsilon)\ln(M)/\mu$ and $\kappa_2= \int_{1/\delta}^{M}\frac{g(s)}{s}ds - (1 + \epsilon)\ln(M)/\mu$. Then for $x\ge M$,
\begin{equation}\label{bd_F}
(1 - \epsilon)\ln(x)/\mu + \kappa_1 \le F(x) \le (1 + \epsilon)\ln(x)/\mu + \kappa_2.
\end{equation}
Due to Lemma~\ref{pf_inv}, we see that  $\Phi_f$ and $F$ are continuous, and $F$ is  increasing.
We recall that if $f_1, f_2$ are increasing functions satisfying $f_1 \leq f_2$ we have $f_2^{-1} \leq f_1^{-1}$.
Applying this principle to \eqref{bd_F}, we conclude that whenever $s$ is large enough, $F^{-1}$ is sandwiched between the inverses of the functions appearing on the right-hand-side and the left-hand-side of \eqref{bd_F}, which leads to
\begin{equation*}
\Phi_f^{-1}(s) = 1/F^{-1}(s) \in\left[\tau_1 c_1^s,\, \tau_2  c_2^s\right],
\end{equation*}
where $\tau_1 := e^{\mu\kappa_1/(1 - \epsilon)}$, $\tau_2 :=  e^{\mu\kappa_2/(1 + \epsilon)}$, $c_1:= e^{-\mu/(1 - \epsilon)}$ and $c_2:= e^{-\mu/(1 + \epsilon)}$. This proves the first half of statement (iii).

Now, it remains the case of $t = o(f(t))$ as $t\to0_+$.  In this case, we have $g\in\RV_0$.  Using this, the local boundedness of $g$ and the definition of $F$, we invoke \cite[Proposition~1.5.9a]{BGT87} which tells us that 
\begin{equation}\label{prim_o}
F \in\RV_0  \ \ \ {\rm and}  \ \ \   g(x) = o(F(x)) \ \ {\rm as}\ \ x\to\infty.
\end{equation}
Due to $t = o(f(t))$ as $t\to0_+$, it holds that $\frac{1/f(1/x)}{x} = \frac{1/x}{f(1/x)}\to 0$ as $x\to\infty$. Namely, we have $1/f(1/x) = o(x)$.
We note from the monotonicity of $f$ and $\lim _{x\to 0_+}f(x) = 0$ that $1/f(1/x)$ is locally bounded and goes to $\infty$ as $x\to\infty$. Therefore, applying Lemma~\ref{lem:inv_o} to $1/f(1/x)$, $x$ and recalling \eqref{eq:rvz_inv_def},  we obtain
\begin{equation*}
\lim_{x\to\infty}g(x) = \lim_{x\to\infty}\frac{1}{xf^{-}(1/x)} = \lim_{x\to\infty}\frac{1/f^{-}(1/x)}{x} = \lim_{x\to\infty}\frac{(1/f(1/\cdot))^{\leftarrow}(x)}{x}   = \infty.
\end{equation*}
This together with \eqref{prim_o} implies that $F(x) \to \infty$ as $x \to \infty$.
Note that $F$ is locally bounded (due to its continuity) and $g$ is locally bounded. Applying Lemma~\ref{lem:inv_o} to $g$ and $F$ by recalling \eqref{prim_o}, we  have 
\begin{equation*}
\Phi_f^{-1}(s) = 1/F^{-1}(s) = 1/o(g^{\leftarrow}(s))  \ \ {\rm as}\ \  s\to\infty,
\end{equation*}
which proves the latter half of statement (iii). This completes the proof.
\end{proof}

Typically Theorem~\ref{theorem_int} would be invoked in the context of Theorem~\ref{theo:rate} with $f = \phi$. However, $\phi$ may have terms of different orders so it will be helpful to verify that we can focus on the important terms only, especially in item~(ii). 

\begin{proposition}\label{prop:asympt_equiv}
Suppose that $f$ and $\hat f$ satisfy the assumptions in Theorem~\ref{theo:rv} and $f \overset{c}{\sim} \hat f \in \RVz_{\rho}$ with $\rho \in (0,\,1)$. 
Then we have
$\Phi_f \overset{c}{\sim} \Phi_{\hat f}$ and $\Phi_f^{-1} \overset{c}{\sim} \Phi_{\hat f}^{-1}$.
\end{proposition}
\begin{proof}
	By \eqref{eq:rvz_inv} and \eqref{eq:inv_equiv_minus}, we have $f^{-}, \hat f^{-} \in \RVz_{1/\rho}$ and 
	\[
	f^{-} \overset{c}{\sim} \hat f^{-}.
	\]
	Now, $1/f^{-}$ and $1/\hat f^{-}$ both belong to $\RVz_{-1/\rho}$.
	Moreover, thanks to Lemma~\ref{inv_lemma}, $1/f^{-}$ and $1/\hat f^{-}$ are monotone and therefore locally bounded away from zero.  Applying \eqref{eq:karamata0_3}\footnote{We restrict the domain of $1/f^{-}$ and $1/\hat f^{-}$ to $(0,\,\delta]$.} (i.e., Karamata's theorem) to both $\Phi_{f}$ and $\Phi_{\hat f}$ we have
	\[
	\Phi_{f} \overset{c}{\sim} \frac{y}{f^{-}(y)}	\overset{c}{\sim}  \frac{y}{\hat f^{-}(y)} \overset{c}{\sim} \Phi_{\hat f} \in \RVz_{1-1/\rho}.
	\]
	Defining $F$ and $\hat F$ by $F(x)  \coloneqq \Phi_{f}(1/x)$ and $\hat F (x) \coloneqq \Phi_{\hat f}(1/x)$, we have $F, \hat F \in \RV_{1/\rho-1}$ and 
	$1/\rho -1 > 0$ because $\rho \in (0,\,1)$. 
	Then, we have $F \overset{c}{\sim} \hat F$. By Lemma~\ref{pf_inv}, $F$ and $\hat F$ are increasing, continuous and both go to $\infty $ as $x \to \infty$, since $\rho \in (0,1)$.
	By Lemma~\ref{lemma_inv_eq} and Proposition~\ref{lb_proposition}, we conclude
	that $F^{-1} \overset{c}{\sim} \hat F^{-1}$.
	Since $F^{-1} = 1/\Phi_f^{-1}$ and $\hat F^{-1} = 1/\Phi_{\hat f}^{-1}$, we conclude that 
	$\Phi_f^{-1} \overset{c}{\sim} \Phi_{\hat f}^{-1}$.
\end{proof}

\section{Applications and examples}\label{sec:examples}
In this section we take a look at some examples of  non-H\"olderian behavior that is covered by the definition of Karamata regularity. 
In sections~\ref{sec:ap} and \ref{sec:dr} we will focus on new results that are obtainable based on the techniques developed in this paper and we will contrast with the results described in \cite{BLT17} and \cite{LL20}.

\subsection{Exotic error bounds and convergence rate of alternating projections}\label{sec:ap}
First, we will see examples of error bounds among convex sets. 
We will consider two closed convex sets $C_1,\, C_2$ with non-empty intersection and apply Definition~\ref{def_jcrv} to the projection operators $L_1 \coloneqq P_{C_1}$, $L_2\coloneqq P_{C_2}$. 
Then, in each case we examine the corresponding convergence rate of the alternating projections (AP) algorithm for solving the common fixed point problem 
\begin{equation}\label{feas_prob}
{\rm find}\ \ x\in C\coloneqq ( \F\, L_1) \cap (\F\,L_2)  = C_1 \cap C_2.
\end{equation}

As a reminder, AP for \eqref{feas_prob} is obtained from the quasi-cyclic iteration described in \eqref{alg_iter} by letting 
$T_{i} \coloneqq L_i$, $w_i^{k} \coloneqq (k+i \mod 2)$, so 
that $w_1^k$ is $(0,1,0,1,\ldots)$, $w_{2}^k$ is $(1,0,1,0,\ldots)$ and therefore $\nu = 1$.
With that, item~(b) of Theorem~\ref{theo:rate} is satisfied with $s=2$.
Then, the assumption in item~(a) of Theorem~\ref{theo:rate}, i.e.,~joint  Karamata regularity of $P_{C_1}, P_{C_2}$ over a bounded set $B$ containing $\{x^k\}$, corresponds to an error bound condition on $C_1,\,C_2$ over $B$, namely, 
	\begin{equation}\label{eq:keb}
\dist(x,\,C) \le
\psi_B\Big(\max_{1\le i\le 2}\norm{x-L_i(x)}\Big)=
 \psi_B\Big(\max_{1\le i\le 2}\dist(x,\,C_i)\Big), \ \ \forall\ x\in B,
\end{equation}
where $\psi_B$ satisfies items~\ref{def_jcrv:2} and \ref{def_jcrv:3} in Definition~\ref{def_jcrv}.
Since projections are $1/2$-averaged, the function $\hat \phi$ in Theorem~\ref{theo:rate} is of the form
\begin{equation}\label{eq:hat_phi}
\hat \phi (u) = \psi _B^2(\sqrt{18u} ).
\end{equation}
In summary, in order to estimate the convergence rate of AP applied to $C_1$ and $C_2$, we have to estimate the function $\Phi_{\phi}^{-1}$. 
We will show how to do using the results developed so far.

We emphasize that there are many different algorithms that are covered by the quasi-cyclic algorithm and this was extensively showcased in \cite{BLT17}.
Although initially we will focus on the particular case of AP, we recall that Theorems~\ref{theo:rv} and \ref{theorem_int} apply in general to quasi-cyclic iterations, since the aforementioned results are also valid under the setting of Theorem~\ref{theo:rate}.
Our choice of using the AP algorithm is only to better emphasize the analysis technique. 
We will also focus on new results and insights that were not obtainable under our previous work \cite{LL20}.
In particular, later in Section~\ref{sec:dr} we will see an example of convergence rate for the DR algorithm, which is something that was not possible under the framework developed in \cite{LL20} and is also beyond the analysis done in \cite{BLT17} since one of the sets is not semialgebraic.

Before we move on we recall some properties of the Lambert W function which denotes 
the converse relation of the function $f(w) \coloneqq we^w$. It has two real branches $W_0$ and $W_{-1}$.
The principal branch  $W_0(x)$ is continuous and increasing on its domain  $[-e^{-1},\,\infty)$, and satisfies $W_0(x) \ge -1$, $W_0(0) = 0$, with $W_0(x)\to \infty$ as $x\to \infty$. The branch  $W_{-1}(x)$ is continuous and decreasing on its domain $[-e^{-1},\,0)$, and satisfies $W_{-1}(x) \le -1$, $W_{-1}(-e^{-1}) = -1$, with $W_{-1}(x)\to -\infty$ as $x\to 0_{-}$. 
More details on the Lambert W function and its applications to optimization can be seen on \cite{BL16}.

\paragraph{An example with  $\rho \in (0,\,1)$: a H\"older-entropic error bound.}

First we will construct two convex sets satisfying a non-H\"olderian error bound with index $\frac{1}{2}$. 
We start by letting the function $\hefun:[-0.5,\,0.5] \to \R$ be such that $\hefun(0) \coloneqq 0$ and
\begin{equation*}
\hefun(x) \coloneqq e^{2W_{-1}\big(-\frac{\abs{x}}{2}\big)}, \qquad \forall\ x \in [-0.5,\,0.5]\setminus \{0\}.
\end{equation*}
The inverse of the restriction of 
$\hefun$ to $[0,\,0.5]$ exists and, with a slight abuse of notation, we will denote it by $\gamma^{-1}$. The domain of $\gamma^{-1}$ is  $[0,\,\hefun(0.5)]$, where  $\hefun(0.5) < 1$,  and its expression on $(0,\,\gamma(0.5)]$ is given by
\begin{equation}\label{eq:ginv}
\hefun^{-1}(y) = - \sqrt{y}\ln(y).
\end{equation}
The function $\hefun^{-1}$ is increasing and concave, and satisfies 
$\hefun^{-1}(0) = 0$, which implies that for $x \ge 0$  and $\alpha > 0$ such that $(1+\alpha)x$ is in the domain of $\hefun^{-1}$ we have
\begin{equation}\label{eq:ginv:2}
\hefun^{-1}(x)  = \hefun^{-1}\Big(\frac{1}{1 + \alpha}\cdot(1 + \alpha)x + \frac{\alpha}{1 + \alpha}\cdot 0\Big) \geq  \frac{\hefun^{-1}((1+\alpha)x)}{1+\alpha}.
\end{equation}
Define
\begin{equation}\label{def_C1C2}
C_1 \coloneqq \{(x,\,\mu) \mid \hefun(x) \leq \mu \}, \ \ \ C_2 \coloneqq \{(x,\,0)\mid x \in \R\}.
\end{equation}
Let $C\coloneqq C_1\cap C_2 $, with that we have $C = \{(0,\,0)\}$.
The projection onto $C_1$ does not seem to have a straightforward closed form, but we will see in Proposition~\ref{prop:conic_rep} that $C_1$ can be represented using conic linear constraints involving an exponential cone and a rotated second-order cone. 
In particular, given $(\hat x, \hat \mu) \in \R^2$, its projection onto $C_1$  can be computed numerically by solving a conic linear program over exponential cones and second order cones.

Our goal is to show the following theorem.
\begin{theorem}
[A H\"older-entropic error bound]\label{Lip_examp}
Let $C_1,\, C_2$ be defined as in \eqref{def_C1C2} and let $B_b \coloneqq \left\{(x,\,\mu) \mid \abs{x} + \abs{\mu}\leq b \right\}$.
Then for small enough $b > 0$ there exists $\kappa > 0$ such that 
\begin{equation}\label{exp_non_eb}
\dist(w,\,C) \leq \kappa\, \hefun^{-1}\Big(\max_{1\le i\le 2}\dist(w,\, C_i)\Big), \ \ \ \forall\,w\in B_b,
\end{equation}
where  $\hefun^{-1}$ is defined as in \eqref{eq:ginv}. Furthermore, 
$\hefun^{-1}$ belongs to $\RVz_{1/2}$ and \eqref{exp_non_eb} is optimal in the following sense: over those $B_b$ with $b$ sufficiently small, $C_1,\,C_2$ do not satisfy a H\"olderian error bound with exponent $1/2$  nor \eqref{exp_non_eb} holds if $\hefun^{-1}$ is substituted with an $\RVz_{\rho}$ function with $\rho > 1/2$. \footnote{That means one cannot expect to obtain an error bound function which improves the current one $\gamma^{-1}$ by either increasing its exponent or keeping the exponent but removing the slowly varying term $-\ln(\cdot)$.}
\end{theorem}
\begin{proof}
First, we can extend $\gamma$ from its domain $[-0.5,\,0.5]$ to $(-2e^{-1},\,2e^{-1})$: let $\widehat{\gamma}: (-2e^{-1},\,2e^{-1})\to\R$ be such that $\widehat{\gamma}(0):= 0$ and
\begin{equation*}
\widehat{\gamma}(x) \coloneqq e^{2W_{-1}\big(-\frac{\abs{x}}{2}\big)}, \qquad \forall\ x \in (-2e^{-1},\,2e^{-1})\setminus \{0\}.
\end{equation*}
Note from the properties of $W_{-1}$ on $(-e^{-1}, \, 0)$  that $\widehat{\gamma}$ is continuous and decreasing on $(-2e^{-1},\,0)$. Hence, the inverse of the restriction of $\widehat{\gamma}$ to $(-2e^{-1},\,0)$ exists and, with a slight abuse of notation, we call it $\widehat{\gamma}^{-1}$.
By calculation we have $\widehat{\gamma}^{-1}(y) = \sqrt{y}\ln(y)$ for $y\in(0,\,e^{-2})$, which is convex and decreasing. By \cite[Proposition~2]{MM08} we have $\widehat{\gamma}$ is convex and decreasing on $(-2e^{-1},\,0)$. Since $\widehat{\gamma}(x) = \widehat{\gamma}(-x)$ and $\widehat{\gamma}$ is continuous at $0$, we have that $\widehat{\gamma}$ is convex on its domain\footnote{It suffices to prove that the epigraph $S:=\left\{(x,\,u) \mid \widehat{\gamma}(x)\le u\right\}$ is convex, i.e., for any $(x,\,u),\,(y,\,v)\in S$ we have $\widehat{\gamma}(\alpha x + (1 - \alpha)y) \le \alpha u + (1 - \alpha)v$ holds for all $\alpha\in(0,\,1)$. Without loss of generality, we assume that $x < 0 < y$ and $z:= \alpha x + (1 - \alpha)y < 0$. The convexity of $\widehat{\gamma}$ on $(-2e^{-1},\,0)$, together with $\widehat{\gamma}(0) = 0$ and its continuity  at $0$, gives $\widehat{\gamma}(z) \le \frac{z}{x}\widehat{\gamma}(x)$. Combing this with $\frac{z}{x}\le \frac{z-y}{x-y} = \alpha$, we obtain $\widehat{\gamma}(z)\le \alpha\widehat{\gamma}(x)\le \alpha u \le \alpha u + (1 - \alpha)v$, which completes the proof.}.
Given that $\gamma$ is the restriction of $\widehat{\gamma}$ on $[-0.5,\,0.5$] and recalling that a convex function is Lipschitz continuous over any compact set contained in the interior of its domain (e.g., \cite[Theorem~3.1.1]{HL93}), we conclude that $\hefun$ is Lipschitz continuous on its domain.

Let $L$ denote the Lipschitz constant of $\hefun$ on its domain, and define $\widehat{L} \coloneqq \max\{L,\,1\}$, $\widehat{b}:= \frac{\hefun(0.5)}{2\sqrt{2}\widehat{L}}$. 
For any $w = (x,\,\mu)\in B_{\widehat{b}}$,
\begin{equation}\label{nb_gua}
2\sqrt{2}\widehat{L}\max_{1\le i\le 2}\dist(w,\,C_i) \le 2\sqrt{2}\widehat{L}\,\dist(w,\, C)\le 2\sqrt{2}\widehat{L}(|x| + |\mu|) \le 2\sqrt{2}\widehat{L}\widehat{b} = \hefun(0.5).
\end{equation}
Moreover, let $(\bar{x},\, \hefun(\bar{x}))$ denote the projection of $(x,\,0)$ onto $C_1$\footnote{If the projection were of the form $(\bar{x},\, \bar{\mu})$ with $\gamma(\bar{x}) < \bar{\mu}$, we would be able to get a point closer to $(x,\,0)$ by replacing $\bar{\mu}$ with $\gamma(\bar{x})$.}, in particular 
$\bar{x} \in \dom\,\hefun$. 
From  $\widehat{L} \geq 1$ and $\hefun(0.5) < 1$ we have $\widehat b \leq 0.5$, and since 
$w \in B_{\widehat{b}}$, we have $x,\,\mu\in\dom\,\hefun$. 

Using the Lipschitz continuity of $\hefun$ on its domain, we then obtain
\begin{equation}\label{bf_g}
\begin{split}
\hefun(x) & \leq L\abs{x-\bar x} + \hefun(\bar{x}) \leq \widehat{L}(\abs{x-\bar x} + \hefun(\bar{x})) \leq \sqrt{2}\widehat{L}\,\dist\left((x,\,0),\,C_1\right) \\
&  \le \sqrt{2}\widehat{L}\left(\dist(w,\, C_1) + \|w - (x,\,0)\|\right) = \sqrt{2}\widehat{L}\left(\dist(w,\, C_1) + \dist(w,\,C_2)\right) \\
& \le 2\sqrt{2}\widehat{L}\max_{1\le i\le 2}\dist(w,\,C_i),\\
\hefun(\mu) & \le L|\mu| \le \widehat{L}|\mu| = \widehat{L}\,\dist(w,\,C_2) \le \widehat{L}\max_{1\le i\le 2}\dist(w,\,C_i).
\end{split}
\end{equation}
Due to \eqref{nb_gua}, the two right-hand side terms of \eqref{bf_g} are in the domain of $\hefun^{-1}$. Thus, for each inequality in \eqref{bf_g}, we can apply $\hefun^{-1}$ at both sides obtain
\begin{equation*}
|x| \le \hefun^{-1}\Big(2\sqrt{2}\widehat{L}\max_{1\le i\le 2}\dist(w,\,C_i) \Big), \ \ \ |\mu| \le \hefun^{-1}\Big(\widehat{L}\max_{1\le i\le 2}\dist(w,\,C_i)\Big).
\end{equation*}
This together with \eqref{eq:ginv:2}  further implies
\begin{equation*}
\begin{split}
\dist(w,\, C)   \le |x| + |\mu|  & \le \hefun^{-1}\Big(2\sqrt{2}\widehat{L}\max_{1\le i\le 2}\dist(w,\,C_i) \Big) +  \hefun^{-1}\Big(\widehat{L}\max_{1\le i\le 2}\dist(w,\,C_i)\Big) \\
& \le \big(2\sqrt{2} + 1\big)\widehat{L}\hefun^{-1}\Big(\max_{1\le i\le 2}\dist(w,\,C_i) \Big).
\end{split}
\end{equation*}
This means that \eqref{exp_non_eb} holds for any $b\in(0,\,\widehat{b})$ and  $\kappa \coloneqq \left(2\sqrt{2} + 1\right)\widehat{L}$. 

Next, we show that the error bound we have in \eqref{exp_non_eb} is optimal in the sense that no error bound function in $\RVz$ with an index {greater} than $1/2$ is admissible. 
We will also show that no H\"olderian error bound with exponent $1/2$ holds for $C_1$ and $C_2$. 

First, suppose that $\psi \in \RV_{\rho}^0$ is a  nondecreasing function such that for sufficiently small $b > 0$,
\begin{equation*}
\dist(w,\,C) \le \psi\Big(\max_{1\le i\le 2}\dist(w,\, C_i)\Big), \qquad \forall\, w= (x,\,\mu) \in B_b.
\end{equation*}
Let $w_k:= (t_k,\,0)$ with $t_k\rightarrow 0_+$ and $t_k\in(0,\,\min(b,\,0.5))$. Then we would have 
\begin{equation}\label{eq:hent_opt}
t_k = \dist(w_k,\,C) \leq \psi\left(\dist((t_k,\,0),\,C_1)\right) \le \psi\left(\norm{(t_k,\,0)-(t_k,\,\hefun(t_k))}\right) = \psi\left(\hefun(t_k)\right).
\end{equation}
Since $\hefun^{-1} \in \RV_{1/2}^0$, the restriction of $\hefun$ to $(0,\,0.5]$ belongs to 
$\RV_{2}^0$ by \eqref{eq:rvz_inv}. With that, the composition $\psi \circ \gamma$ belongs to $\RV^0$ and has index $2\rho$, see \eqref{eq:rvz_calc}. If $\rho > 1/2$,  we then have $\frac{t}{\psi(\hefun(t))}\in\RVz$ with a negative index by \eqref{eq:rvz_calc}. By \eqref{lim_rv0_neg}, this implies that $t/ \psi(\hefun(t)) \to + \infty$ as $t \to 0_+$, which contradicts \eqref{eq:hent_opt}.  Therefore, $\rho$ must be in $[0,\,1/2]$.

Also, if there exists a H\"olderian error bound with exponent $1/2$ for $C_1$ and $C_2$, then \eqref{eq:hent_opt} holds with $\psi(t) = c\cdot t^{1/2}$, where $c > 0$. Consequently,
recalling that $e^{W_{-1}(t)} = t/W_{-1}(t)$ holds, we have  $t_k/\psi(\hefun(t_k)) = 2|W_{-1}(-|t_k|/2)|/c\rightarrow \infty$, which contradicts \eqref{eq:hent_opt}. This completes the proof.
\end{proof}

Next, we consider the problem of estimating the convergence rate of 
the AP algorithm when applied to $C_1$ and $C_2$ as defined in \eqref{def_C1C2}.  
Denote the sequence generated by the AP algorithm by $\{x^k\}$.
Let $B$ be a bounded set containing $\{x^k\}$ and $r > 0$ be such that $B\subseteq\B_r$. Let $b > 0$ and $\kappa >0$ be given as in Theorem~\ref{Lip_examp} such that \eqref{exp_non_eb} holds.
One can see that there exists some $c\in(0,\,e^{-2})$ such that $w\in B_b$ whenever $w\in B$ and $\max\limits_{1\le i\le 2}\dist(w,\,C_i)\le c$ hold\footnote{Suppose that such a constant $c$ does not exist. Then we can construct a sequence $\{c_k\}$ with $c_k\to 0$, a sequence $\{w^k\}\subseteq B$ satisfying $\max\limits_{1\le i\le 2}\dist(w^k,\,C_i)\le c_k$ but $w^k\not\in B_b$, for all $k$. Since $B$ is bounded, without loss of generality, we may assume that $w^k$ converges to some $w^*$, which further leads to $w^*\in C_1\cap C_2 = (0,\, 0)$.
Since $B_b$ contains a neighborhood of $(0,\,0)$, we have $w^k \in B_b$ for large $k$ which is a contradiction.}. With that, we let $\psi_B$ be as follows
\begin{equation}\label{he_psi_b}
\psi_B(t) := \begin{cases}
0 & \text{if }\ t =0, \\
-\kappa\sqrt{t}\ln\,(t)  & \text{if }\ 0 < t \le c, \\
\max\left\{r,\, -\kappa\sqrt{c}\ln\,(c)\right\}  &  \text{if }\ t > c.
\end{cases}
\end{equation}
The function $\psi_B$ satisfies items \ref{def_jcrv:2} and \ref{def_jcrv:3}  of Definition~\ref{def_jcrv}. Now, we show that error bound condition \eqref{eq:keb} holds. 
Given any $x\in B$, we consider three cases: if $\max\limits_{1\le i\le 2}\dist(x,\,C_i) = 0$, then $x\in C$ and hence \eqref{eq:keb} holds, thanks to $\psi_B(0) = 0$; 
if $0 < \max\limits_{1\le i\le 2}\dist(x,\,C_i) \le c$, then $x\in B_b$, which together with  \eqref{exp_non_eb} and the second case in \eqref{he_psi_b}  implies that \eqref{eq:keb} holds; if $\max\limits_{1\le i\le 2}\dist(x,\,C_i)> c$, we see from $C = \{(0,\,0)\}$, $x\in B\subseteq \mathbb{B}_r$ and the third case in \eqref{he_psi_b}  that  \eqref{eq:keb} holds.

From what is discussed above, we then have that $T_1 = P_{C_1}$ and $T_2= P_{C_2}$ are jointly Karamata regular over $B$ with regularity function $\psi_B$ defined as in \eqref{he_psi_b}. 
We will analyze the convergence rate as follows. 
We will compute the function $\phi$ appearing in Theorem~\ref{theo:rate} and use Theorem~\ref{theorem_int} to estimate the asymptotic properties of 
$\Phi_{\phi}^{-1}$.
Recalling  \eqref{eq:hat_phi} and the relation between $\phi$ and $\hat{\phi}$ in \eqref{def_pf}, we observe that when $t$ is small enough,
\begin{equation}\label{eq:he_phi}
\phi(t) = \psi_B^2(\sqrt{18t}) = {3\sqrt{2}\kappa^2}\sqrt{t}(\ln(\sqrt{18t}))^2 \overset{c}{\sim} \sqrt{t}(\ln(t))^2 \in\RVz_{1/2}.
\end{equation}
Let $\hat f(t) \coloneqq \sqrt{t}(\ln(t))^2$. Note that both $\phi$ and $\hat f$ are  continuous and increasing on $(0,\,a]$ for some small enough $a$. After restricting them on $(0,\,a]$, we have that both $\phi(t)$ and $\hat f(t)$ satisfy the assumptions in Theorem~\ref{theo:rv}. On the other hand, let $\hat g(s) := \frac{1}{s\hat f^{-}(1/s)}$ for $s\in[1/\delta,\,\infty)$. Now,  we first apply Proposition~\ref{prop:asympt_equiv} to $\phi$ and $\hat f$ by recalling \eqref{eq:he_phi}, and then apply Theorem~\ref{theorem_int}~(ii) by letting $f = \hat{f}$ and $g = \hat g$, and therefore obtain
\begin{equation}\label{case_2_cite2}
\sqrt{\Phi_{\phi}^{-1}(s)}\overset{c}{\sim} \sqrt{\Phi_{\hat f}^{-1}(s)} \sim\sqrt{\frac{1}{\hat g^{\leftarrow}(s)}}   \ \ {\rm as}\ \ s\to\infty.
\end{equation}
That is, instead of evaluating $\sqrt{\Phi_{\phi}^{-1}(s)}$ directly which may be quite cumbersome, we may estimate the convergence rate through the (relatively)
simpler expression $\sqrt{\frac{1}{\hat g^{\leftarrow}(s)}}$.
So our next task is evaluating $\hat g^{\leftarrow}(s)$.

Since $\hat f$ is continuous and increasing on its domain $(0,\,a]$, its usual inverse exists with $\hat f^{-1}(u) = \hat f^{-}(u)$ for small enough positive $u$. Let $h := \hat{f}^{-1}$. For small enough $u$, we have
\begin{equation}\label{inv_phi_first}
u = \hat f(h(u)) = \sqrt{h(u)}\left(\ln(h(u))\right)^2.
\end{equation}
Let \[z_u\coloneqq \frac{\ln(h(u))}{4}.\] Hence, $z_u\to -\infty$ as $u\to 0_+$. 
We also have $e^{z_u} = {h(u)}^{1/4}$.
Then \eqref{inv_phi_first} together with $z_u < 0$ as $u\to 0_+$ implies that
\begin{equation}\label{eq:zs}
-\frac{\sqrt{u}}{4} = z_ue^{z_u},
\end{equation}
i.e., for sufficiently small positive $u$,
\[
W_{-1}\left(-\frac{\sqrt{u}}{4}\right) = z_{u},
\]
where we note that this must be indeed the $W_{-1}$ branch because the $W_0$ branch is lower bounded by $-1$ and $z_{u}$ goes to $-\infty$ as $u\to 0_+$.
Consequently, using \eqref{eq:zs} we have for sufficiently small positive $u$,
\begin{equation*}
h(u) = (e^{z_u})^4 = \left(-\frac{\sqrt{u}}{4 z_u}\right)^4 = \frac{u^2}{256z_u^4} = \frac{u^2}{256\big[W_{-1}(-\frac{\sqrt{u}}{4} )\big]^4}.
\end{equation*}
Therefore, for large enough $s$, we have 
\begin{equation}\label{g_form_2}
\hat g(s) = \frac{1}{s\hat f^{-}(1/s)} = \frac{1}{sh(1/s)} = 256s\left[W_{-1}\left(-\frac{1}{4\sqrt{s}}\right)\right]^4.
\end{equation}
By the properties of $W_{-1}$, we have that $\hat g$ is increasing and continuous on $[M,\,\infty)$ for some large enough $M > 1/\delta$ and $\hat g(s)\to\infty$ as $s\to\infty$. We note from Theorem~\ref{theorem_int} that $\hat g$ is locally bounded. 
By Proposition~\ref{lb_proposition}, $\hat g^{\leftarrow}(s) = \hat g^{-1}(s)$ holds for large enough $s$. 
Let $w(s):=\hat g^{-1}(s)$ for large enough $s$. Let $t_s := W_{-1}\big(-\frac{1}{4\sqrt{w(s)}}\big)$. Then we obtain from \eqref{g_form_2} that
\begin{equation}\label{g_form_2_sim}
s = \hat g(w(s)) = 256w(s)t_s^4.
\end{equation}
By definition of the Lambert W function, we 
have $t_se^{t_s} = -\frac{1}{4\sqrt{w(s)}}$, that is, 
$\sqrt{w(s)} = -\frac{e^{-t_s}}{4 t_s}$. Since $t_s < 0$, in combination with \eqref{g_form_2_sim}, we obtain
\begin{equation*}
\frac{\sqrt{s}}{4} = 4 \left(-\frac{e^{-t_s}}{4t_s}\right) \cdot \sqrt{t_s^4} = -t_se^{-t_s},
\end{equation*}
which implies that 
\[
W_{0}\left(\frac{\sqrt{s}}{4} \right) = -t_s.
\]
This together with \eqref{case_2_cite2} and \eqref{g_form_2_sim} implies that
\begin{equation*}
\sqrt{\Phi_{\phi}^{-1}(s)} \overset{c}{\sim} \sqrt{\frac{1}{\hat g^{\leftarrow}(s)}}  = \sqrt{\frac{1}{\hat g^{-1}(s)}} = \sqrt{\frac{1}{w(s)}} = \sqrt{\frac{256 t_s^4}{s}} = \frac{16t_s^2}{\sqrt{s}}  = \frac{16\big[W_0(\frac{\sqrt{s}}{4})\big]^2}{\sqrt{s}} \ \ {\rm as}\ \ s\to\infty.
\end{equation*}
Now, $se^s$ is a rapidly varying function, so $W_0$ is slowly varying by \eqref{eq:rav_inv2}, which implies that 
$W_0(\lambda s) \sim W_0(s)$ for every $\lambda > 0$. Therefore,
\begin{equation}\label{main_exp2}
\sqrt{\Phi_{\phi}^{-1}(s)}\overset{c}{\sim}  \frac{\big[W_0({\sqrt{s}})\big]^2}{\sqrt{s}} \ \ {\rm as}\ \ s\to\infty.
\end{equation}
This implies that when the AP algorithm is used to find a feasible point in the intersection of $C_1$ and $C_2$  defined as in \eqref{def_C1C2},
the sequence $\{x^k\}$ converges  at least at a rate proportional to $\frac{\big[W_0({\sqrt{k}})\big]^2}{\sqrt{k}}$. 
Using our previous techniques in \cite{LL20} we would only be able to conclude 
that, up to multiplicative constants, the convergence rate is at least as fast as $(1/k)^{\frac{1}{2}-\varepsilon}$ for all positive $\varepsilon$ satisfying $1/2 - \varepsilon > 0$, see \cite[item~(ii) Theorem~5.7]{LL20}. Notably, $\varepsilon$ cannot be taken to be zero in the context of \cite[Theorem~5.7]{LL20}.

The rate in \eqref{main_exp2}, however, is more explicit and gives the asymptotic equivalence class of the estimate $\sqrt{\Phi_{\phi}^{-1}(k)}$. 
It is interesting to note that $\big[W_0({\sqrt{s}})\big]^2/\sqrt{s}$ does not correspond to a sublinear rate, however, we can see from
\begin{equation*}
\lim_{s\to\infty}\frac{\big[W_0({\sqrt{s}})\big]^2/\sqrt{s}}{1/\sqrt{s}} = \infty\ \ \ \  \mbox{and}\ \ \ \  \lim_{s\to\infty}\frac{\big[W_0({\sqrt{s}})\big]^2/\sqrt{s}}{(1/s)^{1/2-\varepsilon}} = 0
\end{equation*}
that $\big[W_0({\sqrt{s}})\big]^2/\sqrt{s}$ is squeezed between 
$1/\sqrt{s}$ and $(1/s)^{1/2-\varepsilon}$ for all positive $\varepsilon$ such that $1/2-\varepsilon > 0$.
In this way, \eqref{main_exp2} gives a finer evaluation of the convergence rate.

Before we move on to the next example, we have the following result  which implies that the problem of projecting a point onto $C_1$ can be cast a conic linear program (CLP) over exponential cones and second-order cones.
We recall that the exponential cone $\expCone$ is defined as
\[
\expCone\coloneqq\left \{(x_1,x_2,x_3)\in \R^3\;|\;x_2>0,x_3\geq x_2e^{x_1/x_2}\right \} \cup \left \{(x_1,x_2,x_3)\;|\; x_1 \leq 0, x_2=0, x_3\geq 0  \right \}
\]
and the rotated second-order cone in $\R^3$ is defined as 
\[
\mathcal{Q}_r^3\coloneqq \{(x_1,x_2,x_3) \in \R^3 \mid 2x_1 x_2 \geq x_3^2, x_1,x_2 \geq 0\}.
\]

\begin{proposition}[Conic linear representation of $C_1$]\label{prop:conic_rep}
	$(x,\mu) \in C_1$ if and only if 
	there exists $s,v,t\in\R$ such that 
	\begin{equation}\label{eq:conic_rep}
	\begin{aligned}
	(t/2,v,1) \in \expCone,&\quad 0\leq s\leq \mu,&\quad s\leq \gamma(0.5),\\
	(0.5,s,v) \in \mathcal{Q}_r^3,&\quad-t \leq x \leq t.&
	\end{aligned}
	\end{equation}
\end{proposition}
\begin{proof}
	See Appendix~\ref{app:rep}.
\end{proof}
Let $\tilde C_1 \subseteq \R^5$ denote the set  consisting of the $(x,\mu, s,v,t)$ satisfying \eqref{eq:conic_rep}. 
Also let $\mathcal{Q}_2^3$ denote the usual second-order cone in $\R^3$ so that $\mathcal{Q}_2^3 \coloneqq 
\{(x_1,x_2,x_3) \mid x_1^2 \geq x_2^2 + x_3^2, x_1 \geq 0\}$.
Then, the projection of $(\hat x, \hat \mu) \in \R^2$ onto $C_1$ can be computed by solving 
\[
\min_{u,x,\mu, s,v,t\in \R} u \quad \text{subject to } (u,x-\hat x, \mu - \hat \mu) \in \mathcal{Q}_2^3,\,\, (x,\mu, s,v,t) \in \tilde C_1,
\]
which is a conic linear program over second-order cones and a exponential cone and can be solved numerically via a variety of solvers \cite{PY22,KT24,GC24,MC2020}.

\paragraph{An example with $\rho = 1$: Better estimates for the exponential cone.}

As shown in \cite[Section~4.2.1]{LLP20} and \cite[Section~6.2]{LL20}, 
 a so-called \emph{entropic error bound} holds between 
$C_1 \coloneqq \expCone$ and $C_2 \coloneqq  \left\{x \in \R^3 \mid x_2 = 0\right\}$.
More precisely, there exists a constant $\kappa_B > 0$ and a function $\psi_B$ such that \eqref{eq:keb} holds and $\psi_B$ satisfies
\begin{equation}\label{eq:psi_ent}
\psi_B(a) = -\kappa_B a \ln(a)
\end{equation}
for sufficiently small $a$.


Consider the AP  algorithm applied to the feasibility problem of finding $x\in C_1 \cap C_2$. 
Although the projection onto $\expCone$ does not seem to have a known closed form, there is a fast and reliable numerical algorithm available \cite{Fr23}.
We previously analyzed the convergence rate of AP applied to $C_1$ and $C_2$ in \cite[Section~6.2]{LL20} and we were able to conclude the rate is faster than any sublinear rate, but the estimate is worse than any linear rate. 
This was also observed empirically, see Figure~1 in \cite{LL20}. 
Here, however, we obtain a finer result that indicates where the rate is located in the gap between linear and sublinear rates.

Continuing our analysis, from \eqref{eq:psi_ent} and \eqref{eq:hat_phi}, we conclude that the function $\phi$ in Theorem~\ref{theo:rate} takes the form
\begin{equation*}
\phi(t) = \psi _B^2(\sqrt{18t} ) = \kappa t\left(\ln(t) + \ln(18)\right)^2
\overset{c}{\sim} t(\ln(t))^2\in\RVz_1,
\end{equation*}
for sufficiently small $t$ and $\kappa:= \frac{9}{2}\kappa_B^2$.
Note that $t = o(\phi(t))$ as $t\to 0_+$. Moreover,  $\phi$ is continuous and  increasing on $(0,\,a]$ for some small enough $a$. So, we restrict $\phi$ on $(0,\,a]$.  Let $g(s):=\frac{1}{s\phi^{-}(1/s)}$ for $s\in[1/\delta,\,\infty)$.  Applying the second half of Theorem~\ref{theorem_int}~(iii) and letting $f:=\phi$, we obtain
\begin{equation}\label{item3}
\sqrt{\Phi_{\phi}^{-1}(s)} = \frac{1}{o\left(\sqrt{g^{\leftarrow}(s)}\right)} \ \ {\rm as}\ \ s\to\infty.
\end{equation}
Let $h(s):= \phi^{-1}(1/s)$ be defined for large enough $s$.
Since $\phi$ is increasing and continuous on $(0,\,a]$, $h$ is continuous and decreasing for large enough $s$, e.g., see \cite[Remark~1 and Proposition~1]{EH13}.
Moreover, for large $s$, it holds that $\phi^{-1}(1/s) = \phi^{-}(1/s)$ and
\begin{equation*}
1/s = \phi(h(s)) = \kappa\,h(s)\left(\ln(h(s)) + \ln(18)\right)^2.
\end{equation*}
Consequently, for large enough $s$,
\begin{equation}\label{eq:1s}
g(s) = \frac{1}{s\phi^{-}(1/s)} = \frac{1}{s\phi^{-1}(1/s)} = \frac{1/s}{h(s)}  =\kappa\left(\ln(h(s)) + \ln(18)\right)^2.
\end{equation}
Note that $h$ is decreasing and $h(s)\to 0_+$ as $s\to\infty$. This implies that $g$ is increasing and continuous on $[M,\,\infty)$ for some large enough $M > 1/\delta$. Moreover, $g(s)\to\infty$ as $s\to\infty$. Note from Theorem~\ref{theorem_int} that $g$ is locally bounded. By Proposition~\ref{lb_proposition}, we conclude that $g^{\leftarrow}(s) =  g^{-1}(s)$ holds for large enough $s$.
Define \[w(s) \coloneqq g^{-1}(s)\] for large enough $s$. From \eqref{eq:1s} we have 
\begin{equation}\label{eq:gws}
s = g(w(s))  = \kappa\left(\ln\left(h(w(s))\right) + \ln(18)\right)^2.
\end{equation}
Note that when $s\to\infty$, we have $w(s)\to\infty$, $h(w(s))\to 0_+$ and therefore $\ln(h(w(s)))\to -\infty$.  Recalling that $h^{-1}(u) = \frac{1}{\phi(u)}$ holds for small enough $u$, we solve \eqref{eq:gws} in terms of $w(s)$ and conclude that for large enough $s$  we have
\begin{equation*}
w(s) = h^{-1}\Big(e^{-\sqrt{\frac{s}{\kappa}} - \ln(18)}\Big) = \frac{1}{\phi\Big(e^{-\sqrt{\frac{s}{\kappa}} - \ln(18)}\Big)} = \frac{18}{se^{-\sqrt{\frac{s}{\kappa}}}}.
\end{equation*}
Let $c:= e^{\frac{1}{2\sqrt{\kappa}}}$. We then have $c > 1$. 
Combining the above equation with \eqref{item3} and recalling that for large $s$ we have 
$w(s) = g^{-1}(s) = g^{\leftarrow}(s)$, we obtain
\begin{equation*}
\sqrt{\Phi_{\phi}^{-1}(s)} = \frac{1}{o\left(\sqrt{18}/\left(\sqrt{s}c^{-\sqrt{s}}\right)\right)} \ \ {\rm as}\ \ s\to\infty.
\end{equation*}
This is further equivalent to
\begin{equation}\label{main_exp3}
\sqrt{s}c^{-\sqrt{s}}= o\left(\sqrt{\Phi_{\phi}^{-1}(s)}\right) \ \ {\rm as}\ \ s\to\infty.
\end{equation}
As a reminder, we already knew from \cite[Theorem~4.7 and Proposition~6.9~(iii)]{LL20} that the convergence rate of 
AP applied to $C_1$ and $C_2$ is faster than any sublinear rate and 
the estimate obtained therein is slower than any linear rate.
We called such a rate \emph{almost linear}.
The relation \eqref{main_exp3} refines this result and tells us that 
the predicted rate is not only slower than any linear rate but, up to multiplicative constants, it is 
slower than $\sqrt{k}c^{-\sqrt{k}}$. This is a stronger 
statement because $\sqrt{k}c^{-\sqrt{k}}$ goes to zero slower 
than $d^{-k}$ for any $d > 1$, i.e.,
\begin{equation*}
\lim_{k\to\infty}\frac{d^{-k}}{\sqrt{k}c^{-\sqrt{k}}} = 0.
\end{equation*}

\paragraph{Another example with $\rho \in (0,1)$.} 
In Theorem~\ref{Lip_examp}, we saw an example involving the $-1$ branch of the Lambert W function.
Here we will take a look at an instance related to the $0$ branch.
Consider the following two sets:
\begin{equation}\label{lambert_exp}
C_1\coloneqq \left\{(x,\,t)\mid x\ge -e^{-1},\, t \le W_0(x)\right\}, \ \ \ \ C_2\coloneqq \left\{(-e^{-1},\,t)\mid t\in\R \right\},
\end{equation}
that is, $C_1$ is the hypograph of $W_0$.
We define $C\coloneqq C_1\cap C_2$, so 
that $C= \left\{(-e^{-1},\,t)\mid t\le -1\right\}$ and $P_C(w) = (-e^{-1},\,\min\{t,\,-1\})$ for all $w = (x,\,t)$.

The boundary of $C_1$ is the union of the graph of $W_0$ (i.e., pairs of the form $(x,W_0(x))$) with the half line $\{(-e^{-1},t) \mid t \leq W_0(-e^{-1})=-1\}$. 
For our analysis, it will be more convenient to use $W_0^{-1}$ to describe the boundary.
Indeed, we can also describe the boundary of $C_1$ as the union of a 
half-line with the graph of $W_0^{-1}$ (as a function of $t$, so that we have pairs of the form $(W_0^{-1}(t),t)$).
In terms of the original coordinates in $x$ and $t$, this corresponds to reflecting $C_1$ along the $x = t$ line, i.e., ``permuting $x$ and $t$''.

With that in mind, for convenience,  we first reflect $C_1$ and $C_2$ along the  line $x = t$, and then shift them along the $t$-axis with length $e^{-1}$ and along the $x$-axis with length $1$ to obtain
\begin{equation}\label{def_hat}
\widehat{C}_1 = \left\{(x,\,t)\mid t \ge \widehat{W}(x)\right\}, \ \ \ \ \widehat{C}_2 = \left\{(x,\,0)\mid x\in\R\right\}.
\end{equation}
where $\widehat{W}(x)$ is given by 
\begin{equation*}
\widehat{W}(x) =
\begin{cases}
\psi(x)\coloneqq (x - 1)e^{x - 1} + e^{-1}, & \text{if $x \ge 0$;}\\
0, & \text{otherwise}.
\end{cases}
\end{equation*}
In this way, $\widehat{C}_1$ and $\widehat{C}_2$ intersect at $(0,0)$ and we have $\widehat{C}\coloneqq\widehat{C}_1\cap \widehat{C}_2 = \left\{(x,\,0)\mid x\le 0\right\}$, see also Figure~\ref{fig:rotation}.
We have $T(C_1) =\widehat C_1$ and $T(C_2) = \widehat C_2$, where $T:\R^2 \to \R^2$ is the affine 
map such that 
\begin{equation}\label{eq:affine_t}
T(x,t) \coloneqq (t+1, x+e^{-1}).
\end{equation}
Since $T$ is an affine isometry, the error bounds that hold between the pair $C_1, C_2$ and  the pair $\widehat{C}_1, \widehat{C}_2$ are the same.

\begin{figure}
	\begin{subfigure}[b]{0.45\textwidth}
		\centering
		\begin{tikzpicture}
		[scale=1]
		\node[anchor=south west,inner sep=0] (image) at (0,0) {\includegraphics[width=.7\linewidth]{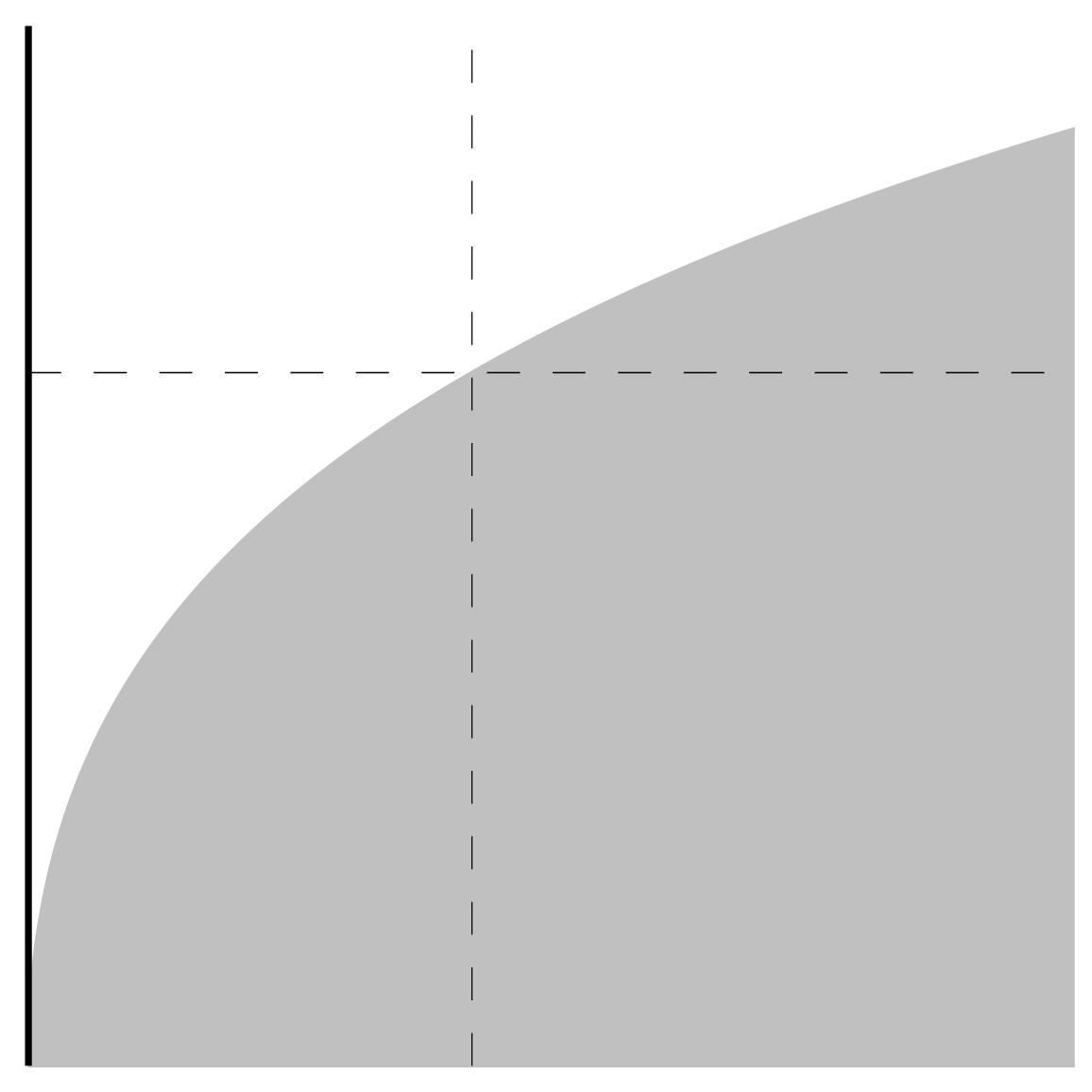}};
		\begin{scope}[x={(image.south east)},y={(image.north west)}]

		\node[black] at (0.08,0.9) {$C_2$};
		
		\node[black] at (0.7,0.4) {$C_1$};
		
		\node[black] at (0.40,0.8) {{\tiny$t$}};
		\node[black] at (0.6,0.63) {{\tiny$x$}};
		
		
		
		\end{scope}
		\end{tikzpicture}
	\end{subfigure}	
	\begin{subfigure}[b]{0.45\textwidth}
		\centering
		\begin{tikzpicture}
		[scale=1]
		\node[anchor=south west,inner sep=0] (image) at (0,0) {\includegraphics[width=.7\linewidth]{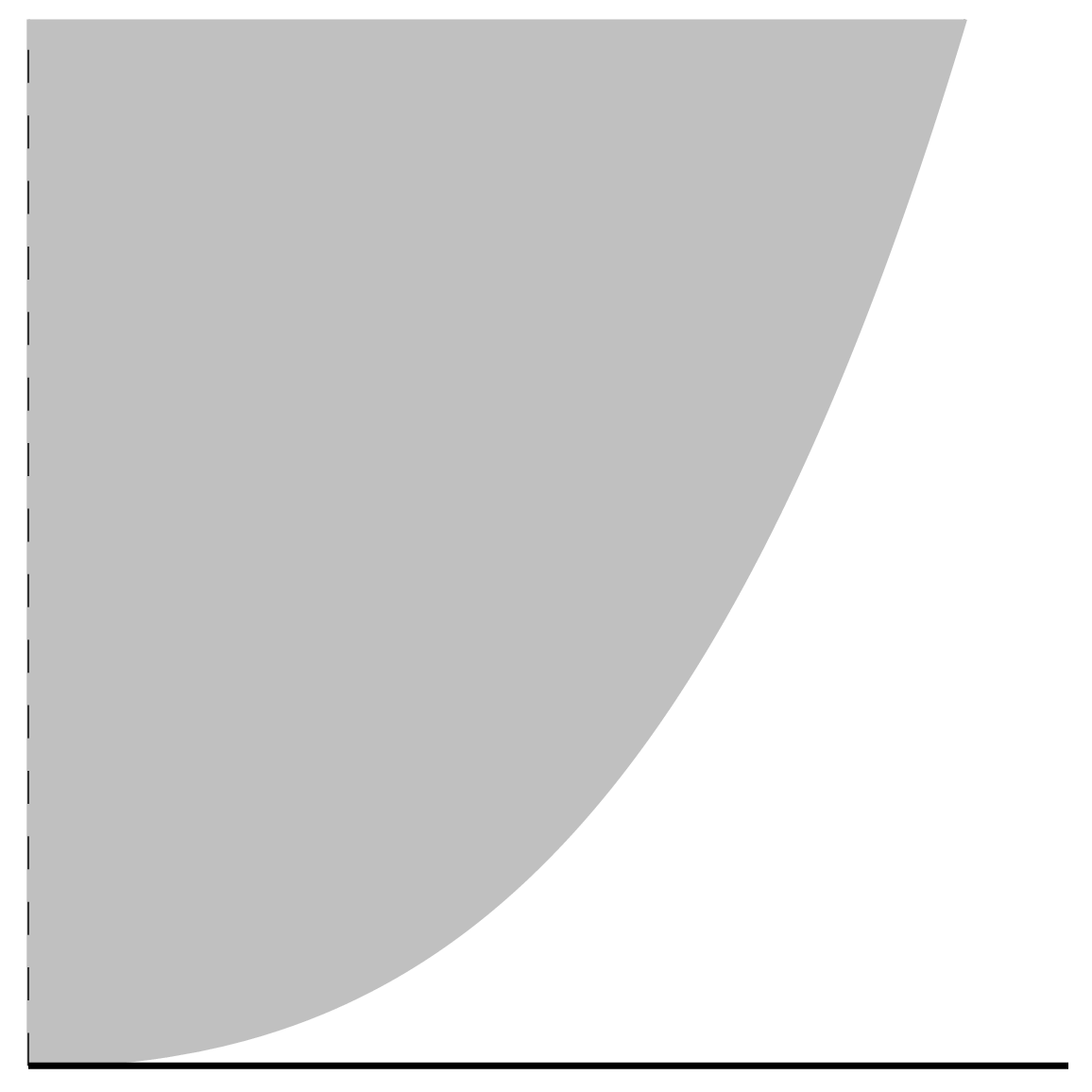}};
		\begin{scope}[x={(image.south east)},y={(image.north west)}]

		\node[black] at (0.3,0.5) {$\widehat{C}_1$};
		
		\node[black] at (0.7,0.08) {$\widehat{C}_2$};
		
		\node[black] at (0.0,0.8) {{\tiny$t$}};
		\node[black] at (0.5,0.0) {{\tiny$x$}};

		
		\end{scope}
		\end{tikzpicture}	
	\end{subfigure}		
	\caption{In the plot on the left, the nonlinear part of the boundary of $C_1$ is described by the graph of $W_0$, i.e., points of the form $(x, W_0(x))$. In the plot on the right, the sets are reflected and translated so that the nonlinear part of the boundary of $\widehat{C}_1$ is the described by $(x, \widehat{W}(x))$, where $\widehat{W}$ is $W_0^{-1}$ followed by a shift ensuring that $\widehat{W}(0) = 0$ holds.}\label{fig:rotation}
\end{figure}

Let $\psi^{-1}$ denote the inverse function of $\psi$. Then we have 
\begin{equation}
\psi^{-1}(t) = W_0(t - e^{-1}) + 1, \ \ \ t \ge 0.
\end{equation}
Note that $\psi^{-1}$ is increasing and concave, and satisfies $\psi^{-1}(0) = 0$. Similar to the inequality in \eqref{eq:ginv:2}, we have
\begin{equation}\label{ccv_prop}
\psi^{-1}(\tau x)\le \tau\psi^{-1}(x), \ \ \ \forall\, x\ge 0,\ \ \tau > 1.
\end{equation}
Moreover, since $\psi\in\RVz_2$, we see from \eqref{eq:rvz_inv} that $\psi^{-1}\in\RVz_{1/2}$.  In particular, 
$t/\psi^{-1}(t) \in \RVz_{1/2}$ by \eqref{eq:rvz_calc} and, therefore, we have
\begin{equation}\label{eq:t_psi_inv}
t = o\big(\psi^{-1}(t)\big) \text{ as } t\to0_+,
\end{equation}
by \eqref{lim_rv0_pos}.
Now we establish an error bound for $C_1$ and $C_2$, whose proof follows a similar line of argument of Theorem~\ref{Lip_examp}.
\begin{theorem}\label{theo:lambertw}
	Let $C_1$ and $C_2$ be defined by \eqref{lambert_exp}. Let $B_b \coloneqq \left\{(x,\,t) \mid \abs{x + e^{-1}} + \max\{t+1,\,0\}\leq b \right\}$. Then for small enough $b > 0$ there exists $\kappa > 0$ such that 
	\begin{equation}\label{theo:lambertw_eq}
	\dist\left(w,\,C\right) \le \kappa\psi^{-1}\Big(\max_{1\le i \le 2}\dist(w,\,C_i)\Big), \ \ \forall\,w\in B_b.
	\end{equation}
\end{theorem}

\begin{proof}
	Let $\widehat{C}_1$ and $\widehat{C}_2$ be defined as in \eqref{def_hat}.
	As discussed previously, since $\widehat{C}_1$ and $\widehat{C}_2$ are obtained from $C_1$ and $C_2$ through a reflection and translation, both pairs of set have the same error bounds, so we will analyze 
	$\widehat{C}_1$ and $\widehat{C}_2$ instead. 
	
	We recall that $\widehat{C}=\widehat{C}_1\cap \widehat{C}_2 = \left\{(x,\,0)\mid x\le 0\right\}$ and $P_{\widehat{C}}(w) = (\min\{x,\,0\},\,0)$ for all $w = (x,\,t)$.  Let $\widehat{B}_b \coloneqq \left\{(x,\,t) \mid \max\{x,\,0\} + \abs{t}\leq b \right\}$, so that $T(B_b) =  \widehat{B}_b $ holds where $T$ is the affine isometry in \eqref{eq:affine_t}.
	It suffices to show that  for small enough $b > 0$ there exists $\kappa > 0$ such that 
	\begin{equation}\label{eb_chat}
	\dist(w,\,\widehat{C}) \le \kappa\psi^{-1}\Big(\max_{1\le i \le 2}\dist(w,\,\widehat{C}_i)\Big), \ \ \forall\,w\in \widehat{B}_b.
	\end{equation}
	Note that for any $w = (x,\,t)\in \widehat{B}_b$, we have $P_{\widehat{C}_2}(w) = (x,\,0)$ and thus $\dist(w,\,\widehat{C}_2) = |t|$. Next, we prove \eqref{eb_chat} by considering two cases. 
	
	First, suppose that $x\le 0$.
	Then, we have $P_{\widehat{C}_1}(w) = (x,\,\max\{t,\,0\})$ and thus $\dist(w,\,\widehat{C}_1) = \max\{-t,\,0\} \le |t| = \dist(w,\,\widehat{C}_2)$. 
	This implies that $\max_{1\le i \le 2}\dist(w,\,\widehat{C}_i) = \abs{t}$.
	We also have $\dist(w,\,\widehat{C}) = |t|$.
	In view of \eqref{eq:t_psi_inv}, if $|t|$ is sufficiently small, we have $|t| \leq \psi^{-1}(\abs{t})$. 
	Therefore, for $b$ is sufficiently small, $|t|$ will also be small and \eqref{eb_chat} will be satisfied with any  $\kappa \geq 1$. 
	
	
	Next, we consider the case $x > 0$. Let $(\bar{x},\, \mu)$ denote the projection of $(x,\,0)$ onto $\widehat{C}_1$. Then we have
	$\mu = \psi(\bar{x})$ and  
	\begin{equation}\label{eq:case2_ineq}
	0 \le \bar{x}\le x\le b,
	\end{equation}
	see this footnote\footnote{  For $x > 0$, $(x,0) \not \in \widehat{C}_1$, so its projection  onto $\widehat C_1$ must be a boundary point of  $\widehat C_1$. The boundary of $\widehat C_1$ is $\{(y,\psi(y)) \mid y \geq 0\} \cup \{(y,0) \mid y \leq 0 \}$.
		If the projection were of the form $\big(\bar{x},\,0)$ with $\bar{x} < 0$, then $(0,0)$ would be a point of $\widehat{C}_1$ closer to $(x,0)$. 
		We conclude that the projection must be indeed of the form $(\bar{x},\, \psi(\bar{x}))$.
		Similarly, if it were the case that $x < \bar{x}$, we would be able to construct a point closer to $(x,0)$ by considering $(x, \psi(x))$ since $0 \leq \psi(x) \leq \psi(\bar{x})$ holds as $\psi$ is increasing. Finally, $x \leq b$ holds because $w =(x,t) \in \widehat B_b$ (by assumption) and $ x > 0$ holds.
	}.
	Since $\psi$ is convex and finite everywhere, its restriction to any bounded interval of $\R$ is Lipschitz continuous.  We let $L$ be the Lipschitz constant of $\psi$ restricted to interval $[0,\,b]$, and define $\widehat{L}\coloneqq\max\{L,\,1\}$, Therefore, we have
	\begin{equation*}
	\begin{split}
	\psi(x)  &  \le \psi(\bar{x}) + L|x - \bar{x}| \le \widehat{L}\left(\psi(\bar{x}) + |x - \bar{x}|\right) \le \sqrt{2}\widehat{L}\dist((x,\,0),\,\widehat{C}_1) \le  \sqrt{2}\widehat{L}\|(x,\,0) - P_{\widehat{C}_1}(w)\| \\
	& \le \sqrt{2}\widehat{L}\big(\|(x,\,0) - w\| + \|w - P_{\widehat{C}_1}(w)\|\big) = \sqrt{2}\widehat{L}\big(\dist(w,\,\widehat{C}_2) + \dist(w,\,\widehat{C}_1)\big).
	\end{split}
	\end{equation*}
	Note that $\psi$ is increasing on its domain and unbounded. 
	Applying $\psi^{-1}$ to both sides of the above inequality and recalling \eqref{ccv_prop}, we have
	\begin{equation}\label{eq:case2_psi_inv}
	x \leq 2\sqrt{2}\widehat{L}\psi^{-1}\Big(\max_{1\le i \le 2}\dist(w,\,\widehat{C}_i)\Big).
	\end{equation}
	Then, we have
	\begin{equation}\label{case2inq}
	\begin{split}
	\dist(w,\,\widehat{C}) & \le \|w - (x,\,0)\| + \dist((x,\,0),\,\widehat{C}) = \dist(w,\, \widehat{C}_2) + \dist((x,\,0),\,\widehat{C})\\
	& = \dist(w,\, \widehat{C}_2) + x  \\
	& \le \max_{1\le i \le 2}\dist(w,\,\widehat{C}_i)  + 2\sqrt{2}\widehat{L}\psi^{-1}\Big(\max_{1\le i \le 2}\dist(w,\,\widehat{C}_i)\Big),
	\end{split}
	\end{equation}
	where the last inequality follows from \eqref{ccv_prop}. Let $d(w) \coloneqq  \max_{1\le i \le 2}\dist(w,\,\widehat{C}_i)$.
	The inequalities in \eqref{eq:case2_ineq}, the Lipschitz continuity of $\psi$ and 
	$\dist(w, \widehat C_2) = |t|$ imply that 
	$d(w)$ is small when $b$ is small.
	Furthermore, recalling \eqref{eq:t_psi_inv}, we have that 
	$d(w) \leq \psi^{-1}(d(w))$ holds if $d(w)$ is sufficiently small. 
	In particular, for sufficiently small $b$, this inequality is valid. Therefore, \eqref{case2inq} implies that, for sufficiently small $b$, \eqref{eb_chat} holds for  $\kappa \geq (1+ 2\sqrt{2}\widehat{L})$.
\end{proof}

Although we will be more succinct this time, similar to the discussion following Theorem~\ref{Lip_examp}, we will analyze the behavior of the AP algorithm when applied to $C_1$ and $C_2$ as in \eqref{lambert_exp}.
We denote the sequence generated by the AP algorithm by $\{x^k\}$ and we let $b > 0 $ be sufficiently small so that Theorem~\ref{theo:lambertw} holds and let  $\kappa$ be such that \eqref{theo:lambertw_eq} holds. 

Denote the interior of $B_b$ by $\interior(B_b)$.
For every $b > 0$, $C_1 \cap C_2$  is contained $\interior(B_b)$. 
Since the sequence $\{x^k\}$  converges to some $\bar{x} \in C_1 \cap C_2$, we can find a ball $\bar{B}$ around $\bar{x}$ such that $\bar{B} \subseteq B_b$ and
$x^k$ belongs to $\bar{B}$ for sufficiently large $k$. 
Analogously to the discussion leading to \eqref{he_psi_b}, we 
conclude that if $B$ is a bounded set containing 
$\{x^k\}$, then
$C_1$ and $C_2$ are Karamata regular over $B$ and $\psi_B$ can be taken to be such that $\psi_B(t) = \kappa \psi^{-1}(t) = \kappa W_0(t-e^{-1}) + \kappa$ holds for sufficiently small $t > 0$.
In particular, we have $\psi_B \in \RVz_{1/2}$.

In view of \eqref{eq:hat_phi} and the calculus rules in \eqref{eq:rvz_calc}, we have $\phi \in \RVz_{1/2}$ as well. 
Then, item~$(ii)$ of Theorem~\ref{theo:rv_old}
leads to $\sqrt{\Phi_{\phi}^{-1}(k)} = o(k^{-r})$
for every $r \in (0,1)$.
That is, the convergence rate is at least as fast as $(1/k)^{\frac{r}{2}}$ for every $r \in (0,1)$.
For this example, we omit a more detailed analysis using Theorem~\ref{theorem_int}.

The set $C_1$ has a conic linear expression using two exponential cones as described in \cite[Section~5.2.10]{MC2020}. 
This is based on the observation that 
$(x,t) \in C_1$ if and only if there exists $u$ such that $x \geq u \ln(u)$ and $u \geq e^t$.
Therefore, analogously to the discussion following Proposition~\ref{prop:conic_rep}, the projection onto $C_1$ can be computed numerically by solving a conic linear program over exponential cones and second-order cones.

\subsection{A Douglas-Rachford splitting method example}
\label{sec:dr}

In this subsection, we will demonstrate an  example of convergence rate for the DR algorithm. Specifically, we consider the two sets $C_1$ and $C_2$ defined in \eqref{def_C1C2}, i.e., \begin{equation*}
C_1 \coloneqq \{(x,\,\mu) \mid \gamma(x) \leq \mu \}, \ \ \ C_2 \coloneqq \{(x,\,0)\mid x \in \R\},
\end{equation*}
and apply the DR algorithm to $C_1$ and $C_2$ as follows:
\begin{equation}\label{DR_w}
w^{k+1} = T_{\rm DR}(w^k) \coloneqq w^k + P_{C_1}\left(2P_{C_2}(w^k) - w^k\right) - P_{C_2}(w^k).
\end{equation}
Note that $C_1$ is not semialgebraic, so the convergence rate analysis in \cite{BLT17} is not applicable. However, we can obtain the convergence rate of $\{w^k\}$ by using the techniques developed in this paper.
Also, in view of Proposition~\ref{prop:conic_rep} and the following discusison, the operator $T_{\rm DR}$ can be computed numerically by solving a conic linear program over second-order cones and a exponential cone.

To proceed, we first compute $\F\,T_{\rm DR}$. Using the definition of $\F\,T_{\rm DR}$ and the closed-form expression of $P_{C_2}(\cdot)$, we have
\begin{equation*}
\begin{aligned}
 w \coloneqq (x,\,\mu) \in \F\,T_{\rm DR}  & \Longleftrightarrow \ P_{C_1}\left(2P_{C_2}(w) - w\right) = P_{C_2}(w)\\
& \Longleftrightarrow\ P_{C_1}\left((x,\,-\mu)\right)  = (x,\,0)\\
& \Longleftrightarrow\ x = 0,\, \mu \ge 0,
\end{aligned}
\end{equation*}
where we check the third equivalence as follows. Suppose that $P_{C_1}\left((x,\,-\mu)\right)  = (x,\,0) = P_{C_2}(w)$ holds. Then, $(x,\,0) \in C_1\cap C_2$, therefore $x = 0$. 
Now, $\gamma(0) = 0$, so if $\mu < 0$ holds, then $(x,\,-\mu )  = (0,\,-\mu) \in C_1$, which would lead to $P_{C_1}\left((x,\,-\mu)\right)  = (x,\,-\mu)$. This contradicts with $P_{C_1}\left((x,\,-\mu)\right)  = (x,\,0)$, so $\mu$ must be nonnegative.
Conversely, if $ x = 0$ and $\mu \geq 0$, then $P_{C_2}(w) = (0,\,0)$ and $P_{C_1}\left((0,\,-\mu)\right) = (0,\,0)$ holds because $\gamma$ is nonnegative.
%
This proves that
\begin{equation}\label{fix_dr}
\F\,T_{\rm DR} = \{(0,\,\mu) \mid \mu \ge 0\}.
\end{equation}
We will now discuss how to apply Theorem~\ref{theo:rate} to $T_{\rm DR}$.
The DR algorithm \eqref{DR_w} can be obtained from the quasi-cyclic iteration described in  \eqref{alg_iter} by  letting $m = 1$, $T_1 = T_{\rm DR}$ and $w_i^k \equiv 1$ (thus $\nu = 1$). With that, item (b) of Theorem~\ref{theo:rate} is satisfied with $s = 1$. 

In order to verify the assumption in item (a) of Theorem~\ref{theo:rate}, we 
need to show the Karamata regularity of $T_{\rm DR}$ over a bounded set $B$ containing $\{w^k\}$, namely, the existence of $\psi_B$ satisfying the last two items of Definition~\ref{def_jcrv} and the first item as below
\begin{equation}\label{phib_cond}
\dist\big(w,\,\F\,T_{\rm DR}\big) \le \psi_B\Big(\left\|T_{\rm DR}(w) -  w\right\|\Big), \ \ \forall\ w\in B.
\end{equation}
Once this is done and recalling that $T_{\rm DR}$ is $1/2$-averaged \cite[Lemma~4.1]{BLT17}, the function $\hat{\phi}$ in Theorem~\ref{theo:rate} is of the form 
\begin{equation}\label{hat_dr}
\hat{\phi}(u) = \psi_B^2(\sqrt{10u})
\end{equation}
and $\phi$ is obtained by restricting $\hat{\phi}$ to some interval of the form $(0,\,a]$. 

So the first step in order to obtain the convergence rate of the sequence $\{w^k\}$ is to compute $\psi_B$.   For this goal, we will first show the following theorem.  
In what follows, we consider the following notation for $r > 0$:
\begin{equation*}
\mathcal{B}_r \coloneqq \left\{w\in\R^2\mid \dist\left(w,\,\F\,T_{\rm DR}\right)\le r\right\}.
\end{equation*}

\begin{theorem}\label{ebf_dr}
Let $T_{\rm DR}$ be given in \eqref{DR_w}. There exist $r > 0$ and $\kappa > 0$ such that for all $w\in\mathcal{B}_r$,
\begin{equation}\label{dr_art}
\dist\big(w,\,\F\,T_{\rm DR}\big) \le \kappa \gamma^{-1}\Big(\left\|T_{\rm DR}(w) -  w\right\|\Big).
\end{equation}
\end{theorem}

\begin{proof}
For any $w = (x,\,\mu)$, we have from \eqref{fix_dr} that
\begin{equation}\label{lr_dr}
\begin{split}
\dist\big(w,\,\F\,T_{\rm DR}\big) & = \left\|(x,\,\mu) - \left(0,\,\max(\mu,\,0)\right)\right\| = \left\|(x, \,\min\left(\mu,\,0\right))\right\|, \\
\left\|T_{\rm DR}(w) -  w\right\| & = \left\|P_{C_1}\left(2P_{C_2}(w) - w\right) - P_{C_2}(w)\right\| = \left\|P_{C_1}\left((x,\,-\mu)\right) - (x,\,0) \right\|.
\end{split}
\end{equation}
Since $\lim_{t\rightarrow 0_+}\frac{t}{\gamma^{-1}(t)} = \lim_{t\to 0_+}\frac{t}{-\sqrt{t}\ln(t)}= 0$ and $\gamma^{-1}(0) = 0$, there exists some $c_1 > 0$ such that 
\begin{equation}\label{eq:tgamma_inv}
t \le \gamma^{-1}(t), \qquad \forall\ t \in[0,\,c_1].
\end{equation}
Due to the continuity of $\gamma$ and $\gamma(0) = 0$, there exists some $c_2 > 0$ such that 
\begin{equation}\label{eq:constants}
\sqrt{2}\widehat{L}\sqrt{4t^2 + \gamma^2(t)} \le \min\left(\gamma(0.5),\,c_1\right)
\end{equation} 
holds for all $t\in[0,\,c_2]$, where $\widehat{L} := \max(L,\,1)$ and $L$ is the Lipschitz constant of $\gamma$ on its domain, as shown in Theorem~\ref{Lip_examp}.
Let $r:= \min\{c_1,\,c_2,\,\gamma(0.5),\,0.5\}$.
For any  $w = (x,\,\mu)\in\mathcal{B}_r$, we consider the following two cases.
\begin{itemize}
\item[{\rm (a)}] $(x,\,-\mu)\in C_1$. In this case, $\gamma(x)\le -\mu$ holds. Therefore, we have $\mu\le 0$ and $\gamma(x)\le -\mu = |\mu|$. Furthermore, we note from \eqref{lr_dr} and $w\in\mathcal{B}_r$ that $\dist\big(w,\,\F\,T_{\rm DR}\big) = \left\|(x, \,\mu)\right\| \le r$. This  implies $|\mu|\le r\le \gamma(0.5)$ and $|\mu| \le r \le c_1$. Consequently,
 \begin{equation}\label{bdd_1}
\dist\big(w,\,\F\,T_{\rm DR}\big) \le |x| + |\mu|  \le \gamma^{-1}(|\mu|) + |\mu| \le 2\gamma^{-1}(|\mu|) = 2\gamma^{-1}\Big(\left\|T_{\rm DR}(w) -  w\right\|\Big),
\end{equation}
where the last inequality follows from \eqref{eq:tgamma_inv} and the equality follows from \eqref{lr_dr}.

\item[{\rm (b)}] $(x,\,-\mu)\notin C_1$. In this case,  we have $\gamma(x) > -\mu$. Note from \eqref{lr_dr} that  $|x| \le \dist\big(w,\,\F\,T_{\rm DR}\big)\le r\le 0.5$. Let $(\bar{x},\, \bar{\mu})$ denote the projection of $(x,\,-\mu)$ onto $C_1$. Since $\gamma$ is defined on $[-0.5,\,0.5]$, $|x|\le 0.5$, and $\gamma(a) < \gamma(b)$ whenever $|a| < |b|\le 0.5$, we then have $|\bar{x}|\le |x|\le 0.5$ and $\bar{\mu} = \gamma(\bar{x})$.\footnote{Suppose that $|\bar{x}| > |x|$. We then have $\bar{\mu} \ge \gamma(\bar{x}) > \gamma(x)$, which together with $\gamma(x) > -\mu$ implies that  $0 < \gamma(x) + \mu  = \|(x,\,-\mu) - (x,\,\gamma(x))\| < \mu + \bar{\mu}  < \|(x,\,-\mu) - (\bar{x},\,\bar{\mu})\|$. That means we can find $(x,\,\gamma(x))\in C_1$ which is closer to $(x,\,-\mu)$ than $(\bar{x},\,\bar{\mu})$. This leads to a contradiction, thus we conclude that $|\bar{x}| \le |x|$. Next we prove that $\bar{\mu} = \gamma(\bar{x})$. If this does not hold, we must have $\bar{\mu} > \gamma(\bar{x})$ since $(\bar{x},\,\bar{\mu})\in C_1$. Now, we consider two cases. If $-\mu\le\gamma(\bar{x})$, then $0\le \gamma(\bar{x}) + \mu < \bar{\mu} + \mu$. Thus, the point $(\bar{x},\,\gamma(\bar{x}))\in C_1$ satisfies $\|(x,\,-\mu) - (\bar{x},\,\gamma(\bar{x}))\| < \|(x,\,-\mu) - (\bar{x},\,\bar{\mu})\|$, which leads to a contradiction. If $-\mu > \gamma(\bar{x})$, since $-\mu <\gamma(x)$ and $\gamma$ is continuous with $\gamma(u)= \gamma(-u)$ for $u\in[-0.5,\,0.5]$, we can find some $\widehat{x}$ with the same sign with $x$ such that $\gamma(\widehat{x}) =  -\mu$. We then have $\gamma(\bar{x}) < \gamma(\widehat{x}) < \gamma(x)$, so $|\bar{x}| < |\widehat{x}| < |x|$. Thus, the point $(\widehat{x},\,\gamma(\widehat{x}))\in C_1$ satisfies $\|(x,\,-\mu) - (\widehat{x},\,\gamma(\widehat{x}))\| = |x - \widehat{x}| < |x - \bar{x}| \le \|(x,\,-\mu) - (\bar{x},\,\bar{\mu})\|$, which leads to a contradiction. This proves $\bar{\mu} = \gamma(\bar{x})$.}

We then obtain from the Lipschitz continuity of $\gamma$ on its domain that
\begin{equation}\label{Lipg}
\gamma(x) \le \gamma(\bar{x}) + L|x - \bar{x}| \le \widehat{L}\left(|x - \bar{x}| + \gamma(\bar{x})\right) \le \sqrt{2}\widehat{L}\left\|\left(x - \bar{x},\, \gamma(\bar{x})\right)\right\| .
\end{equation}
Moreover, we see from $|\bar{x}|\le |x| \le r \le c_2$ that
\begin{equation}\label{eq:l_gamma}
\sqrt{2}\widehat{L}\left\|\left(x - \bar{x},\, \gamma(\bar{x})\right)\right\| \le \sqrt{2}\widehat{L}\sqrt{4r^2 + \gamma^2(r)} \le \min\left(\gamma(0.5),\,c_1\right),
\end{equation}
where the last inequality follows from \eqref{eq:constants}.
This means that the right-hand-side of \eqref{Lipg} is in the domain of $\gamma^{-1}$.
Then, applying $\gamma^{-1}$ at both sides of \eqref{Lipg} and invoking \eqref{eq:ginv:2}, we obtain
\begin{equation}\label{bd_absx}
|x| \le \gamma^{-1}\left(\sqrt{2}\widehat{L}\left\|\left(x - \bar{x},\, \gamma(\bar{x})\right)\right\|\right) \le \sqrt{2}\widehat{L}\gamma^{-1}\left(\left\|(x - \bar{x},\,\gamma(\bar{x}))\right\|\right).
\end{equation}
Next, we consider two cases. If $\mu \ge 0$,  combining \eqref{lr_dr} and \eqref{bd_absx} we obtain
\begin{equation}\label{bdd_2}
\dist\big(w,\,\F\,T_{\rm DR}\big) =  |x| \le \sqrt{2}\widehat{L}\gamma^{-1}\left(\left\|(x - \bar{x},\,\gamma(\bar{x}))\right\|\right) = \sqrt{2}\widehat{L}\gamma^{-1}\left(\left\|T_{\rm DR}(w) -  w\right\|\right).
\end{equation}
Finally, suppose that $\mu < 0$. 
We  obtain from $\widehat{L}\ge 1$, \eqref{eq:tgamma_inv} and \eqref{eq:l_gamma} that
\begin{equation}\label{eq:aux}
\sqrt{2}\widehat{L}\left\|\left(x - \bar{x},\, \gamma(\bar{x})\right)\right\|\leq \sqrt{2}\widehat{L}\gamma^{-1}(\left\|\left(x - \bar{x},\, \gamma(\bar{x})\right)\right\|).
\end{equation}
We have from $(x,\,-\mu)\notin C_1$ that $|\mu| = -\mu < \gamma(x)$. In view of \eqref{Lipg} and \eqref{eq:aux}, we obtain
\begin{equation}\label{eq:mu}
|\mu| < \gamma(x) \leq \sqrt{2}\widehat{L}\left\|\left(x - \bar{x},\, \gamma(\bar{x})\right)\right\| \leq \sqrt{2}\widehat{L}\gamma^{-1}(\left\|\left(x - \bar{x},\, \gamma(\bar{x})\right)\right\|).
\end{equation}
Then, \eqref{lr_dr}, \eqref{bd_absx}, \eqref{eq:mu} imply that
\begin{equation}\label{bdd_3}
\begin{split}
\dist\big(w,\,\F\,T_{\rm DR}\big) & \le  |x| + |\mu| \\
& \le  2\sqrt{2}\widehat{L}\gamma^{-1}\left(\left\|(x - \bar{x},\,\gamma(\bar{x}))\right\|\right) = 2\sqrt{2}\widehat{L}\gamma^{-1}\left(\left\|T_{\rm DR}(w) -  w\right\|\right).
\end{split}
\end{equation}
\end{itemize}
Finally,  \eqref{bdd_1}, \eqref{bdd_2} and \eqref{bdd_3} taken together imply that \eqref{dr_art} holds with $\kappa = 2\sqrt{2}\widehat{L}$. This completes the proof.
\end{proof}

Next, we will use Theorem~\ref{ebf_dr} to obtain $\psi_B$ and further estimate the convergence rate of $\{w^k\}$ in \eqref{DR_w}. Let $B$ be a bounded set containing $\{w^k\}$ and $R$ be such that $B\subseteq\mathcal{B}_R$. Let $r > 0$ and $\kappa > 0$ be given as in Theorem~\ref{ebf_dr}. One can see that there exists some $c\in(0,\,e^{-2})$  such that $w\in\mathcal{B}_r$  whenever $w\in B$ and $\|T_{\rm DR}(w) - w\|\le c$ hold\footnote{Suppose that such a constant $c$ does not exists. Then we can construct a sequence $\{c_k\}$ with $c_k\to 0$, a sequence $\{u^k\}\subseteq B$ satisfying $\|T_{\rm DR}(u^k) - u^k\| \le c_k $ but $u^k\notin\mathcal{B}_r$. 
 Since $B$ is bounded, without loss of generality, we may assume that $u^k$ converges to some $u^*$. Now, $T_{\rm DR}$ is continuous, so $\|T_{\rm DR}(u^k) - u^k\| \le c_k $ and $c_k\to 0$ leads to $u^*\in\F\,T_{\rm DR}$. Now,  $\mathcal{B}_r$  contains a neighbourhood of  $\F\,T_{\rm DR}$ and, in particular, a neighbourhood of $u^*$, so $u^k \in \mathcal{B}_r$ for large $k$ which is a contradiction.}.
Let
\begin{equation}\label{eq:he_psi_b}
\psi_B(t) := \begin{cases}
0 & \text{if }\ t =0, \\
-\kappa\sqrt{t}\ln\,(t)  & \text{if }\ 0 < t \le c, \\
\max\left\{R,\, -\kappa\sqrt{c}\ln\,(c)\right\}  &  \text{if }\ t > c.
\end{cases}
\end{equation}
With that, $\psi_B$ satisfies items \ref{def_jcrv:2} and \ref{def_jcrv:3}  of Definition~\ref{def_jcrv}. Now we  check item \ref{def_jcrv:1} of Definition~\ref{def_jcrv}, i.e., condition \eqref{phib_cond}. Indeed, given any $w\in B$, we consider three cases: if $\|T_{\rm DR}(w) - w\| = 0$, then $w\in \F\,T_{\rm DR}$ and hence \eqref{phib_cond} holds, thanks to $\psi_B(0) = 0$; 
if $0< \|T_{\rm DR}(w) - w\|\le c$, then $w\in\mathcal{B}_r$, which together with \eqref{dr_art} and the second case in \eqref{eq:he_psi_b}  implies that \eqref{phib_cond} holds; if $\|T_{\rm DR}(w) - w\| > c$, we see from $w\in B\subseteq\mathcal{B}_R$ and the third case in \eqref{eq:he_psi_b}  that \eqref{phib_cond} holds.

Recalling \eqref{hat_dr}, we observe that when $t$ is small enough,
\begin{equation*}
\phi(t) =\hat{\phi}(t) = \psi_B^2(\sqrt{10t}) = \sqrt{10}\kappa^2\sqrt{t}(\ln(\sqrt{10t}))^2\overset{c}{\sim} \sqrt{t}(\ln(t))^2 \in\RVz_{1/2}.
\end{equation*}
So we are back to the same situation as in \eqref{eq:he_phi}. 
Following the same exact computations that culminate in \eqref{main_exp2}, we then obtain
\begin{equation}\label{exp_dr_art}
\sqrt{\Phi_{\phi}^{-1}(s)}\overset{c}{\sim}  \frac{\big[W_0({\sqrt{s}})\big]^2}{\sqrt{s}} \ \ {\rm as}\ \ s\to\infty.
\end{equation}
As a reminder, the set $C_1$ defined in \eqref{def_C1C2} is not semialgebraic and hence the convergence rate of the DR algorithm \eqref{DR_w} cannot be obtained from \cite[Proposition~4.1]{BLT17}, which requires all sets involved to be semialgebraic. 
To the best of our knowledge, the rate obtained in \eqref{exp_dr_art} is new for the DR algorithm.

\subsection{Recovering previous results}\label{sec:previous}
In this quick subsection we reap some extra fruits of the theory developed so far. 

\paragraph{Sublinear/linear rates under H\"older regularity.}
We briefly sketch how to recover the main convergence rate result in \cite{BLT17} as follows. 
For the iteration in \eqref{alg_iter} suppose that the $T_j$'s are bounded H\"older 
regular with uniform exponent $\gamma_{1,j} \in (0,\,1]$ and the intersection of the fixed point 
sets is nonempty and has a H\"olderian error bound of uniform exponent $\gamma _2 \in (0,\,1]$, 
see Remark~\ref{rmk_def} for a review of these notions.
Let $\gamma_1 \coloneqq \min_{1 \leq j \leq m} \gamma _{1,j} $, $\gamma \coloneqq \gamma_1 \gamma_2$ and 
suppose that we are under the assumptions of Theorem~\ref{theo:rate}.

The fact that the intersection of the fixed point sets has a H\"olderian error bound of uniform exponent $\gamma _2 \in (0,\,1]$ implies that they have a consistent error bound function of 
the form $\Phi(a,\,b) \coloneqq \sigma(b)a^{\gamma_2} $ for some some nondecreasing function $\sigma$, see \cite[Theorem~3.5]{LL20}.
Also, each $T_j$ is uniformly Karamata regular and, given a bounded set $B$, $T_j$ has a regularity function of the form $\Gamma_{B}^j(a) = \kappa _B a^{\gamma_{1,j}}$ for some $\kappa _B > 0$.
	
Then, applying Proposition~\ref{prop_psi}, we conclude that the $T_j$ are uniformly jointly Karamata regular  and the regularity function $\psi_B$ can be taken to be asymptotically equivalent (up to a constant) to the composition of 
$\sigma(b)a^{\gamma _2}$ with the sum of all the $\Gamma_{B}^j$'s. By the regular variation calculus rules 
we have $\psi_B \in \RVz_{\gamma}$ (see \eqref{eq:rvz_calc}) and  $\psi_B(t) \overset{c}{\sim} t^{\gamma}$ as $t \to 0_+$, since, intuitively, only the terms with smallest exponent matter as $t \to 0_+$.
Similarly, the $\phi$ in Theorem~\ref{theo:rate} also satisfies $\phi(t) \in \RVz_{\gamma}$ and 
$\phi(t)  \overset{c}{\sim} t^{\gamma}$ as $t \to 0_+$.

Next, we consider two cases.
First, if $\gamma \in (0,\,1)$, then applying item~(ii) of Theorem~\ref{theorem_int} and  Proposition~\ref{prop:asympt_equiv}, we conclude that  $\Phi_{\phi}^{-1}(s) \overset{c}\sim \frac{1}{g^{\leftarrow}\big((1/\gamma - 1)s\big)}$ as $s\to\infty$, where 
$g(t) \coloneqq \frac{1}{t\sqrt[\gamma]{1/t}} = t^{\frac{1-\gamma}{\gamma}}$. Now, 
$g$ is invertible over $(0,\,\infty)$, so we get $g^{\leftarrow}(s) = g^{-1}(s) = s^{\frac{\gamma}{1-\gamma}}$ for large $s$ by Proposition~\ref{lb_proposition}.
Overall, we obtain 
\begin{equation}\label{eq:holder_rate}
\sqrt{\Phi_{\phi}^{-1}(k)} \overset{c}{\sim} k^{-\frac{\gamma}{2(1-\gamma)}} \ \ {\rm as}\ \ k\to\infty.
\end{equation}

If $\gamma = 1$, then $\phi(t)  \overset{c}{\sim} t$ as $t \to 0_+$  so the first half of 
item~(iii) of Theorem~\ref{theorem_int} implies a linear convergence rate.
Both rates match what is given in \cite[Theorem~3.1]{BLT17}.

\paragraph{AP and logarithmic error bounds.}
We say  that convex sets $C_1,\ldots, C_m$ have a logarithmic error bound with exponent $\gamma$ if they admit a consistent error bound function $\Phi$ such that for every $b > 0$, there exist $\kappa_b > 0$ and $a_b > 0$ with $\Phi(a,\,b) = \kappa_b\left(-\frac{1}{{\rm ln}(a)}\right)^{\gamma}$ for $a\in(0,\,a_b)$,
see \cite[Definition~5.8]{LL20}. 

This corresponds to a very pathological kind of error bound that is worse than any H\"olderian error bound, see \cite[Example~5.9]{LL20} for a family of examples related to epigraph of the function mapping  a sufficiently small $x$ to $e^{-\frac{1}{\abs{x}^\beta}}$ for some $\beta \geq 2$. 
See also \cite[Example~4.10]{BDNP16} for a related discussion of the $\beta = 2$ case.
Another example of logarithmic error bound was found in \cite[Section~4.2.3]{LLP20} in the study of exponential cones, a highly expressive class of cones that can be used to model convex problems that require the exponential and logarithm functions, see \cite{CS17} and \cite[Chapter~5]{MC2020}.
In particular, a logarithmic error bound of exponent $1$ holds for the exponential cone and 
a certain subspace, 
see \cite[Section~6.2]{LL20} for more details. Other examples of logarithmic error bounds in the context of optimization over log-determinant cones can also be found in \cite{LLLP24}.

For simplicity, let $C_1$ and $C_2$ be two closed convex sets which satisfy a logarithmic consistent error bound $\Phi$ with exponent $\gamma >0$ and consider the method of alternating projections.
Let $B$ be a bounded set containing the iterates. Then the operators $P_{C_1}$ and $P_{C_2}$ are jointly Karamata regular over $B$, and the regularity function $\psi_B$ can be taken 
to be such that $\psi_B(u) = \Phi(u,\,b)$ for some large $b$ (see Remark~\ref{rmk_def}), which implies that 
$\psi_B(u) = \kappa\left(-\frac{1}{\ln(u)}\right)^{\gamma}$ for small $u$ and some constant $\kappa > 0$. 
Recalling  \eqref{eq:hat_phi} and the relation between $\phi$ and $\hat{\phi}$ in \eqref{def_pf}, we observe that when $u$ is small enough,
\begin{equation}\label{re_phi_log}
 \phi(u) = \psi_B^2(\sqrt{18u}) = \kappa^2\left(-\frac{1}{{\rm ln}(\sqrt{18u})}\right)^{2\gamma}\overset{c}{\sim} \left(-\frac{1}{\ln(u)}\right)^{2\gamma}\in\RVz_0.
\end{equation}
Note that $\phi$ is continuous and  increasing on $(0,\,a]$ for some small enough $a$. So we can restrict $\phi$ on $(0,\,a]$. Let $g(s):=\frac{1}{s\phi^{-}(1/s)}$ and $\widehat{g}(s):=sg(s)$ for $s\in[1/\delta,\,\infty)$. 
Then $\widehat{g}(s) = \frac{1}{\phi^{-}(1/s)} = \frac{1}{\phi^{-1}(1/s)} $ holds for large enough $s$. We see from this and the explicit expression of $\phi$ in \eqref{re_phi_log} 
that $\widehat{g}$ is continuous and increasing on $[M,\,\infty)$ for some large enough $M > 1/\delta$. Moreover, due to $\lim_{t\to 0_+}\phi(t) = 0$, we have $\widehat{g}(s) \to\infty$ as $s\to\infty$.
By Theorem~\ref{theorem_int}, we have that  $g$ is locally bounded and thus so is $\widehat{g}$.
Consequently, $\widehat{g}^{-1}(s) = \widehat{g}^{\leftarrow}(s)$ for large enough $s$, thanks to Proposition~\ref{lb_proposition}.
Now, for large enough $s$ we let $y_s:=\phi^{-1}(1/s)$. Then, $1/s =\phi(y_s) = \psi_B^2(\sqrt{18y_s})$, which implies that $1/\sqrt{s} = \psi_B(\sqrt{18y_s}) = \kappa\left(-\frac{1}{\ln(\sqrt{18y_s})}\right)^{\gamma}$. 
Solving this to obtain $\ln(y_s)$, we have
\[
\ln(y_s) = -\ln(18) -2\kappa^{1/\gamma} s^{1/(2\gamma)}
\]
for large $s$, which leads to
\begin{equation*}
\ln(g(s)) = \ln(\widehat g(s)) - \ln(s) = -\ln\left(\phi^{-1}(1/s)\right) - \ln(s) = -\ln(y_s) - \ln(s)\in\RV_{\frac{1}{2\gamma}}.
\end{equation*}
Applying Theorem~\ref{theorem_int}~(i) by letting $f = \phi$ and $\alpha = 1$, we have
\begin{equation}\label{main_exp1}
\sqrt{\Phi_{\phi}^{-1}(s)}\sim\sqrt{\frac{1}{\widehat{g}^{\leftarrow}(s)}}  = \sqrt{\frac{1}{\widehat{g}^{-1}(s)}}  = \sqrt{\phi(1/s)}  \overset{c}{\sim} \left(\frac{1}{\ln(s)}\right)^{\gamma} \ \ {\rm as}\ \ s\to\infty,
\end{equation}
where the last equivalence follows from \eqref{re_phi_log}. The rate in \eqref{main_exp1} matches what is given in \cite[Theorem~5.12]{LL20}.

\section{Definable operators and jointly  Karamata regular operators}
\label{sec: o-minimal}
In this section, we further explore the  class of jointly  Karamata regular operators  proposed in Definition~\ref{def_jcrv}. Our main goal is to show that, under mild assumptions, operators that are {definable} in some $o$-minimal structure are always jointly Karamata regular,  provided that their fixed points intersect. 

Next, we recall the definition of $o$-minimal structure, which is somewhat technical. 
That said, the reader may take heart from the fact that, besides some basic notions, the sole tool we 
need from this body of theory in order to derive the main result of this section is the so-called \emph{monotonicity lemma}. In view of this situation, we defer more detailed explanations to \cite[Section~4]{BDLS07}, \cite[pg.~452]{ABRS10} or to \cite{Dri98}. 
We define $o$-minimal structures and definable set/functions following \cite[Section~4]{BDLS07}.
\begin{definition}[$o$-minimal structure and definable sets/functions]\label{def:omin}
An $o$-minimal structure on $(\R,\, +,\,.)$	 is a sequence of Boolean algebras $\mathcal{O}_n$ of subsets of $\R^n$ such that for each $n$ we have
\begin{enumerate}[$(i)$]
	\item if $A$ belongs to $\mathcal{O}_n$ then $A \times \R$ and $\R \times A$ belong to $\mathcal{O}_{n+1}$;
	\item if $\pi : \R^{n+1} \to \R^n$ is the projection map such that $\pi(x_1,\ldots, x_{n+1}) = (x_1,\ldots, x_n)$,
	then $\pi(A) \in \mathcal{O}_n$ for all $A \in \mathcal{O}_{n+1}$;
	\item $\mathcal{O}_n$ contains the family of algebraic subsets of $\R^n$, i.e., every set of the form $\{x \in \R^n\mid p(x) = 0\}$ where $p: \R^n \to \R$ is a polynomial function;
	\item the elements of $\mathcal{O}_1$ are exactly the finite unions of intervals and points.
\end{enumerate}
A subset $A \subseteq \R^n$ is said to be \emph{definable} if $A \in \mathcal{O}_n$. A function $f:\R^n \to \R^m$ is said to be \emph{definable} if its graph $\{(y,\,z)\in \R^n\times \R^m\mid  z = f(y)\}$ belongs to $\mathcal{O}_{n+m}$. 
An extended-valued function $f : \R^n \to \R \cup \{+\infty\}$ is definable if its graph $\{(y,\,\alpha) \in \R^n\times \R\mid f(y) = \alpha\}$ is in $\mathcal{O}_{n+1}$.
\end{definition}


Key examples of $o$-minimal structures are given by the semialgebraic sets and the globally subanalytics sets. Here we recall that a set is semialgebraic if 
it can be written as a finite union of solution sets of polynomial equalities and polynomial inequalities. 
 In particular, the fact that the coordinate projection of a semialgebraic set is also a semialgebraic (item $(ii)$ of Definition~\ref{def:omin}) corresponds to the celebrated Tarski-Seidenberg theorem.

Many important optimization problems are defined over semialgebraic sets and this is useful in several ways.
For example, as a consequence of the so-called {\L}ojasiewicz inequality, the error bounds governing the intersection of convex semialgebraic sets can always be taken to be H\"olderian. Similarly, if the underlying operators describing an algorithm are semialgebraic, they must be bounded H\"older regular. For a proof, see, for example, \cite[Proposition~4.1]{BLT17}.
 
Unfortunately, the exponential function is not semialgebraic nor global subanalytic. Therefore, many problems that require exponentials and logarithms need to be defined over a larger $o$-minimal structure such as the \emph{log-exp structure}, which contains the aforementioned structures together with the graph of the exponential function. 

A troublesome aspect of the log-exp structure is that we can no longer expect that the underlying definable sets/functions have H\"older-like properties. Indeed, in the study of the error bounds appearing in problems defined over the exponential cone there is an example of an intersection having no H\"olderian error bounds, see \cite[Example~4.20]{LLP20}. A similar phenomenon occurs in the context of log-determinant cones \cite{LLLP24}.

In spite of this difficulty, in this section our goal is to show that continuous quasi-nonexpansive definable operators in any $o$-minimal structure are jointly  Karamata regular as in Definition~\ref{def_jcrv}.
This serves as a counterpart to the fact that H\"olderian behavior can no longer be expected over general $o$-minimal structures and also shows that theory developed in this paper applies quite generally.

%
%

We start with a natural extension of \cite[Proposition~3.3]{LL20}, which essentially says that there is always some function that describes the joint level of regularity of some given operators. Its proof follows the same line of arguments as \cite[Proposition~3.3]{LL20}, but we show the details here for the sake of self-containment. 

\begin{lemma}\label{lem:op}
	Let $L_i: \E\rightarrow\E$ $(i = 1,\ldots,n)$ be continuous operators such that $\bigcap_{i=1}^n\F\,L_i\neq\emptyset$. For  all $a,\, b\ge 0$, let $\Omega_{a,b}:= \Big\{y\in\E\mid \max\limits_{1\le i\le n}\norm{y- L_i(y)} \le a,\,\|y\| \le b\Big\}$. Define
	\begin{equation}\label{def_best_fun_t}
	\Phi(a,\,b):= 
	\begin{cases}
	\max\limits_{y\in\Omega_{a,b}}\,\dist\left(y,\, \bigcap_{i=1}^n\F\,L_i\right) & if\ \Omega_{a,b}\neq\emptyset;\\
	0 & otherwise.
	\end{cases}
	\end{equation}
Then $\Phi$ satisfies the following conditions:
\begin{enumerate}[$(i)$]
	\item[{\rm (i)}] for any $b\ge 0$, function $\Phi(\cdot,\,b)$ is  nondecreasing, $\lim\limits_{a\to 0+}\Phi(a,\,b) = 0$ and $\Phi(0,\,b) = 0$;	
	\item[{\rm (ii)}] for any  $a\ge 0$, function $\Phi(a,\,\cdot)$ is  nondecreasing;
	
	\item[{\rm (iii)}]  for any $x\in\E$, we have $\dist\left(x,\, \bigcap_{i=1}^n\F\,L_i\right) \le \Phi\Big(\max\limits_{1\le i\le n}\norm{x-L_i(x)},\, \|x\|\Big)$.
\end{enumerate}
\end{lemma}
\begin{proof}
Except for the continuity requirement in item~(i), all the other properties of $\Phi$ in items~(i), (ii), (iii) follow directly from the definition of $\Phi$. We only need to show that $\lim_{a\to 0+}\Phi(a,\,b) = 0$ holds for every $b \geq 0$. 

Suppose on the contrary that there exist some $\delta > 0$ and a sequence $\{a_k\} \subseteq [0,\infty)$ converging to $0$ such that $\Phi(a_k,\, b) \ge \delta$ holds for every $k$. We then see from \eqref{def_best_fun_t} that $\Omega _{ a_k,b} \neq \emptyset$. By the continuity of each $L_i$, we have that $\Omega _{ a_k,b}$ is compact. Since the distance function to $\bigcap_{i=1}^n\F\,L_i$ is continuous, for each $a_k$, there exists $y^k \in \Omega _{a_k,b}$ such that 
\begin{equation*}
\dist\Big(y^k,\, \bigcap_{i=1}^n\F\,L_i\Big) = \Phi(a_k,\,b) \ge \delta.
\end{equation*}
Note that $\{y^k\}$ is bounded. Then there exists a subsequence $\{y^{k_j}\}$ which converges to some limit $\bar{y}$. Since $a_{k} \to 0$, we see from $\max\limits_{1\le i\le n}\norm{y^k- L_i(y^k)} \le a_k$ that $\bar{y}\in \bigcap_{i=1}^n\F\,L_i$. Consequently,
 $\dist(y^{k_j},\, \bigcap_{i=1}^n\F\,L_i) \to \dist(\bar{y},\, \bigcap_{i=1}^n\F\,L_i) = 0$, which contradicts the fact that $\dist(y^{k_j},\, \bigcap_{i=1}^n\F\,L_i) = \Phi(a_{k_j},\,b) \geq \delta $ should hold for every $j$. This proves the continuity requirement in item~(i) and completes the proof.
\end{proof}

Next, we state formally the monotonicity lemma, which, in particular, ensures that a real definable function defined over an interval cannot oscillate infinitely over its domain. 
\begin{lemma}[Monotonicity lemma, {\cite[Chapter~3]{Dri98}}]\label{lem:mon}
Let $f:(a,\,b) \to \R$ be a definable function in some $o$-minimal structure. Then, there exists a finite subdivision $a = a_0 < a_1 < \ldots < a_k = b$ such that on each subinterval $(a_{i},\,a_{i+1})$, $f$ is continuous and either constant or strictly monotone (i.e., increasing or decreasing).
\end{lemma}

Before we move on to the main theorem of this section, we recall some basic properties of definable sets and functions, see \cite[Section~2.1]{DM96}, \cite[Section~1.3]{Coste00} and \cite{Io09} for more details. 
In summary, the class of definable functions/sets is remarkably stable: if $f,\,g$ are definable functions in some $o$-minimal structure, then whenever the functions $f \pm g$, $f\circ g$, $f^{-1}$, $f/g$ are well-defined, they must be definable over the same $o$-minimal structure. In addition, finite unions and intersections of definable sets are definable as well.
Naturally, inverse images of definable sets through definable functions are definable as well.

A very powerful technique to show that a given set $A$ is definable is to express $A$ as a solution of a ``first-order formula quantified over definable sets and functions''. A detailed proof is given in \cite[Theorem~1.13]{Coste00}, but this principle is referenced throughout the literature, e.g., \cite[Appendix~A]{DM96} and \cite[Section~2]{Io09}. A simple application  is that if $A_i$ are definable sets, $Q_i \in \{\exists,\, \forall \}$ are quantifiers, $\Delta \in \{<,\,\leq,\, =,\, \neq \}$ and $f$ is a definable function  then the set of $x$ satisfying 
\begin{equation}\label{eq:1storder}
Q_1 y_1 \in A_1, \ Q_2 y_2 \in A_2, \ldots, Q_m y_m \in A_m, \qquad f(x,\,y_1,\,\ldots, y_m) \Delta 0 \text{ holds.}
\end{equation}
is definable.
This principle gives, for instance, an easy proof of the fact that the set $C_f$ of continuity points of a definable function $f$ is definable as well.
After all, $C_f$ is exactly the set of solutions of the formula ``$\forall\, \epsilon > 0,\, \exists\, \delta > 0, \,\forall\, y \in \{z\mid \norm{x-z} \leq \delta \}, \qquad \norm{f(x)-f(y)} - \epsilon \leq 0$ holds''. 
This principle also leads to an easy proof of the following well-known lemma regarding partial minimization.
\begin{lemma}[Partial minimization preserves definability]\label{lem:partial_min}
Let $A \subset \R^n$ be a definable set and let $f : \R^n \times \R^m \to \R \cup \{+\infty\}$ be a definable function, all over the same $o$-minimal structure. Then, the function $\varphi:\R^n \to \R \cup \{+\infty\}$
given by $\varphi(x) = \inf_{y\in A} f(x,\,y)$ is definable.
\end{lemma}
\begin{proof}
The graph of $\varphi$ is the intersection of the sets of solutions $(x,\,\alpha)$ of two first-order formulae:
\begin{equation*}
\forall\, y \in A,\, f(x,\,y) \geq \alpha\quad \text{ and }\quad \forall\, \epsilon > 0,\, \exists\, y \in A,\,  f(x,\,y) < \alpha + \epsilon.
\end{equation*}
Due  to \eqref{eq:1storder},  these two sets are definable. Consequently, the graph of $\varphi$ is definable and $\varphi$ is definable. This completes the proof.
\end{proof}

We move on to the main result of this section and we recall that an operator $L:\E\to \E$ is said to be \emph{quasi-nonexpansive} if $\norm{Lx-y} \leq \norm{x-y}$ holds for every $x \in \E$ and $y \in \F\,L$.

\begin{theorem}\label{theo:cvo}
	Let $L_i: \E\rightarrow\E$ $(i = 1,\ldots,n)$ be continuous operators such that $\bigcap_{i=1}^n\F\,L_i\neq\emptyset$.   If all the $L_i$ are quasi-nonexpansive and definable in some $o$-minimal structure, then $L_1,\ldots,L_n$ are jointly Karamata regular  (\emph{resp.} Karamata regular when $n = 1$) over any bounded set  $B\subseteq\E$.
	In particular,  the corresponding regularity function $\psi_B$ can be taken to be $\Phi(\cdot,\,r)$ defined in Lemma~\ref{lem:op} with sufficiently large $r$. Furthermore, $\Phi(\cdot,\, r) \in \RVz_{\rho}$ with index $\rho \in [0,\,1]$ when restricted to $(0,\,1]$ ($\rho$ may depend on $r$).
		
\end{theorem}

\begin{proof}
	For simplicity, let $C \coloneqq \bigcap_{i=1}^n\F\, L_i$.
	First, we get rid of a trivial case. 
	If $C = \E$ and $B$ is an arbitrary bounded set, the conclusion holds if we take $\psi_B$ to be the identity map restricted to the interval $[0,\infty)$. 
	Henceforth, we next suppose that $C \neq\E$.
	
	Fix any bounded set $B\subseteq\E$.	
	 Then there exists some $r_B > 0$ such that $B\subseteq \B_{r_B}$, where we recall that $\B_{r_B}$ is the  closed ball centered on the origin with radius $r_B$. We choose some $\widehat{x}\notin C$ and let $r$ be such that \[r \geq \max\{r_B,\, \dist(0,\,C),\,\|\widehat{x}\|\}.\]
	Let $\Phi$ be defined as in \eqref{def_best_fun_t} and let 
	\begin{equation}\label{eq:psib:def}
	\psi_{B}(t):=\Phi(t,\,r),\ \ \  t\ge 0.
	\end{equation} We then know from Lemma~\ref{lem:op}~(ii) and (iii) that for any $x\in \B_r$ (in particular, $x\in B$),
	\begin{equation}\label{eb_ineq_t}
	\begin{aligned}
	\dist(x,\,C)  &\le \Phi\Big(\max_{1\le i\le n}\norm{x-L_i(x)},\, \|x\|\Big) \\ 
	&\le  \Phi\Big(\max_{1\le i\le n}\norm{x-L_i(x)},\, r\Big) = \psi_{B}\Big(\max_{1\le i\le n}\norm{x-L_i(x)}\Big).
	\end{aligned}
	\end{equation}
	One can see from Lemma~\ref{lem:op}~(i) that $\psi_B$ is  nondecreasing and satisfies
	\begin{equation*}
	\psi_{B}(0) = \lim_{t\to 0_+}\psi_{B}(t) = \lim_{t\to 0_+}\Phi(t,\,r) = 0.
	\end{equation*}
	It then remains to show that $\psi_{B}|_{(0,\,1]}\in\RVz_{\rho}$ with $\rho\in[0,\,1]$.
	
	First, we show that $\psi_B(t) >0$ holds for any $t > 0$.
	Let $d(x):= \max\limits_{1\le i\le n}\norm{x- L_i(x)}$.
	Notice that $d(P_C(0)) = 0$ and $d(\widehat{x}) > 0$.
	Since all the $L_i$ are continuous, $d$ is continuous as well.  By the intermediate value theorem, for any $t > 0$, there exists some $x_t = \alpha P_C(0) + (1 - \alpha)\widehat{x}$ with $\alpha\in[0,\,1)$ such that $d(x_t) = \min\{t,\,d(\widehat{x})\}$. Moreover, we have
	\begin{equation*}
	\|x_t\| = \left\|\alpha P_C(0) + (1 - \alpha)\widehat{x}\right\| \le \alpha\|P_C(0)\| + (1 - \alpha)\|\widehat{x}\| = \alpha\,\dist(0,\,C) + (1 - \alpha)\|\widehat{x}\| \le r,
	\end{equation*}
	which together with $0 < d(x_t) \le t$ implies that $x_t\not\in C$ and  $x_t\in\Omega_{t,r}$, where $\Omega_{t,r}$ is defined as in Lemma~\ref{lem:op}. As a result, for any $t > 0$ we have
	\begin{equation}\label{1st_form_t}
	\psi_B(t) = \max_{x\in\Omega_{t,r}}\,\dist(x,\, C) > 0.
	\end{equation}
	Next, we  prove $\psi_{B}\in\RVz$. Let
	\begin{equation*}
	h(x,\,t) := \dist(x,\,C) + \delta_{\Omega}(x,\,t)\ \ \ {\rm with}\ \ \ \Omega:=\Big\{(x,\,t)\mid \max_{1\le i\le n}\norm{x-L_i(x)} \le t,\,\|x\| \le r\Big\},
	\end{equation*}
	where $\delta_{\Omega}$ is the indicator function of $\Omega$ and we recall that $r$ is fixed.
	Since all the $L_i$ are definable,  each $\F\,L_i$ is definable\footnote{Note that $\F\,L_i$ is the inverse image of $\{0\}$ of the definable map $L_i-I$, where $I$ is the identity map. Therefore, it is definable, thanks to \cite[B.3]{DM96}.} and thus $C = \bigcap_{i=1}^n\F\,L_i$ is definable. By Lemma~\ref{lem:partial_min} applied to $\norm{x-y}$ and $C$   we see that $\dist(x,\,C)$ is definable, see also \cite[Proposition~2.8]{Hoang16}.
	
	Note that $\Omega$ can be written as the intersection of $n +1$ sets,  each of which is definable\footnote{Again, this can be proved in multiple ways. For example, by observing that the set of $x$ satisfying $\norm{x - L_i(x)} \leq t$ is the inverse image of the inverval $(-\infty,\,0]$ by the definable map $\norm{x - L_i(x)} -t$. }. 
        Consequently, $\Omega$ is definable and $h(x,\,t)$ is definable as well, since it is the sum of two definable functions. From \eqref{1st_form_t} we have $\psi_B(t) = \max_xh(x,\,t)$. Then $\psi_B(t)$ is definable by Lemma~\ref{lem:partial_min}. 
    
    In particular, the restriction $\psi_B$ to the interval $(0,\,1)$ is a definable positive function. 
   	Next, for 
   	every  $\mu\in(0,1)$, we define the function \[\psi_B^{\mu}(t): = \frac{\psi_B(\mu t)}{\psi_B(t)},\] for $t \in (0,\,1)$. 
   	Because compositions and quotients of definable functions are definable, we have  that $\psi_B^{\mu}(t)$ is  definable. By Lemma~\ref{lem:mon}, there is an interval $(0,\,c) \subseteq (0,\,1)$ over which  $\psi_B^{\mu}(t)$ is either strictly monotone  or constant, in particular 
	\begin{equation}\label{key_limit_t}
	\lim_{t\to 0_+}\psi_B^{\mu}(t)  \ \ \mbox{exists in } \R\cup\{-\infty,\, \infty\},
	\end{equation} 
	see also  \cite[Exercise~2.3]{Coste00}.
	Now, we let $f(x):= \psi_B(1/x)$, $x\in [1,\, \infty)$. 
	Then we see from the monotonicity of $\psi_B$ (recall \eqref{eq:psib:def} and Lemma~\ref{lem:op}) that $f$ is  nonincreasing and therefore measurable.  Moreover, for all  $x\ge 1$ and $\lambda \geq 1$ we have
	\begin{equation}\label{key_mono_t}
	0 < \frac{f(\lambda x)}{f(x)} \le 1.
	\end{equation}
	On the other hand, using \eqref{key_limit_t} we obtain for all $\lambda \in (1,\infty)$,
	\begin{equation*}
	\lim_{x\to\infty}\frac{f(\lambda x)}{f(x)}  = \lim_{x\to\infty}\frac{\psi_B\left(1/(\lambda x)\right)}{\psi_B(1/x)} = \lim_{t\to 0_+}\frac{\psi_B\left(\frac1{\lambda} t\right)}{\psi_B(t)} = \lim_{t\to 0_+}\psi_B^{1/\lambda}(t) \ \ \mbox{exists in } \R\cup\{-\infty,\, \infty\}.
	\end{equation*}
	This together with \eqref{key_mono_t} implies that  for all $\lambda\geq 1$, $\lim_{x\to\infty}\frac{f(\lambda x)}{f(x)}$ exists in $\R$.
	Now let $\Lambda_0$ be the set of $\lambda \geq 1$ for which $\lim_{x\to\infty}\frac{f(\lambda x)}{f(x)} = 0$ holds. 	$\Lambda_0$ must be definable because it is the set of $\lambda \geq 1$ satisfying 
	\[
	\forall \epsilon > 0, \exists M > 0, \forall x \geq M \qquad \left|{\frac{f(\lambda x)}{f(x)}}\right| \leq \epsilon.
	\]
	Let $\Lambda_{>}$ be the complement of $\Lambda_0$ intersected with $[1,\,\infty)$. 
	With that, $\Lambda_{>}$ corresponds to the $\lambda \geq 1$ for which $\lim_{x\to\infty}\frac{f(\lambda x)}{f(x)} > 0$ holds. Since definability is preserved by complements and intersections, $\Lambda_{>}$ is definable and, in view of Definition~\ref{def:omin}, must be a finite union of intervals and points.
	
	Next we consider two cases. If $\Lambda_{>}$ contains an interval, then $\lim_{x\to\infty}\frac{f(\lambda x)}{f(x)}$ exists, is finite and positive for all $\lambda$ in a set 
	of positive measure. Since $f$ is measurable, 
	this is enough to conclude 
	that $f\in\RV$, see \cite[Theorem~1.4.1~(ii)]{BGT87}. In view of \eqref{eq:rv_rvz} and the definition of $f$, we 
	conclude that $\psi_B|_{(0,\,1]}\in\RVz_{\rho}$ with some $\rho\in\R$.
	
	If $\Lambda_{>}$ does not contain an interval, then it must be a union of finitely many points. However, one may verify that if $\lambda \in \Lambda_{>}$ then $\lambda^n \in \Lambda_{>}$ for all $n \in \mathbb{N}$, e.g., see \cite[Section~1.4]{BGT87} or this footnote\footnote{It can also be proved by induction by observing that letting $g_{\lambda} \coloneqq \lim _{x \to \infty} \frac{f(\lambda x)}{f(x)}$ and assuming that $g_{\lambda} \in (0,\,\infty)$, we have $\frac{f(\lambda^2 x)}{f(x)} = \frac{f(\lambda^2 x)}{f(\lambda x)} \frac{f(\lambda x)}{f(x)} \to  g_{\lambda}^2$.}.
	So the only way that $\Lambda_{>}$ can be a finite union of points is if $\Lambda_{>} = \{1\}$ holds.
	In this case, we have $\Lambda_0 = (1,\,\infty)$ so $f \in \RV_{-\infty}$ and 
	$\psi_B$ belongs to $\RVz_{\infty}$.

	Overall, we conclude that $\psi_B \in \RVz \cup \RVz_{\infty}$. 
	Our next step is to show that $\psi_B$ cannot be in $\RVz_{\infty}$.
	For that, given any $x$ and every $i$, we let $y \coloneqq P_{C}(x)$ and we see from the quasi-nonexpansiveness of $L_i$ that
	\begin{equation*}
	\norm{x-L_i(x)} = \norm{x - y +y - L_i(x)} \leq 2\norm{x-y} = 2\,\dist(x,\,C).
	\end{equation*}
	This together with \eqref{eb_ineq_t} and $d(x)\coloneqq \max\limits_{1\le i\le n}\norm{x-L_i(x)}$ implies that for all $x\in\B_r$,
	\begin{equation}\label{ex_set_t2}
	\frac{d(x)}{2}\le \dist(x,\,C) \le \psi_B(d(x)).
	\end{equation}
	Recall that $P_C(0)$ and $\widehat{x}$ are in $\B_r$ and they satisfy $d(P_C(0)) = 0$ and $d(\widehat{x}) > 0$, respectively. 
	Since $d(\cdot)$ is a continuous function, $d(\cdot)$ takes all values between 0 and $d(\widehat{x})$ over the ball $\B_r$. This together with \eqref{ex_set_t2} implies that for sufficiently small $t$ we have
	\begin{equation}\label{eq:no_rapid}
	t/2 \le \psi_B(t).
	\end{equation}
	That is, $1/2 \leq \psi_B(t)/t$ holds for all sufficiently small $t$. Or, put 
	otherwise, 
	\[
	\frac{1}{2} \leq \frac{\psi_B(1/x)}{1/x} = \frac{f(x)}{x^{-1}}
	\]
	holds for all sufficiently large $x > 0$.
	This implies that $x^{-1}\frac{1}{f(x)} \leq 2$ for $x$ sufficiently large. 
	For the sake of obtaining a contradiction, suppose that $\psi_B \in \RVz_{\infty}$.
	Then $1/f \in \RV_{\infty}$ and is nondecreasing, since $f$ is nonincreasing. With 
	that $1/f \in \KRV$ by  \eqref{eq:krv_mon}. 
	Therefore, the function that maps $x$ to $x^{-1}\frac{1}{f(x)}$ is in $\RV_{\infty}$, by \eqref{eq:krv} and \eqref{eq:krv_power}. However this implies that $x^{-1}\frac{1}{f(x)} \to \infty $ as 
	$x\to \infty$, e.g., see \eqref{eq:potter_rapid2}.
	This is a contradiction and it shows that $\psi_B$ cannot be rapidly varying, 
	so it must be in $\RVz_{\rho}$ for some $\rho \in \R$.

	For the last part, we will prove that  $\rho\in[0,\,1]$ holds. First, for $\lambda > 1$, we see from the monotonicity of $\psi_B$ that $\lambda^{\rho} = \lim_{t\to 0_+}\psi_B(\lambda t)/\psi_B(t) \ge 1$, which proves that $\rho \ge 0$.
	
	It remains to show that $\rho \le 1$. Suppose to the contrary that $\rho > 1$. We let $\delta \in(0,\,\rho - 1)$ and conclude from Potter's bound \eqref{eq:potter} (set $A = 2$) that for sufficiently small $t$ and $s$,
	\begin{equation}\label{potter_use_t}
	\psi_B(t) \le 2\psi_B(s)\max\Big\{\Big(\frac{t}{s}\Big)^{\rho - \delta},\, \Big(\frac{t}{s}\Big)^{\rho + \delta}\Big\}.
	\end{equation}
	In view of \eqref{eq:no_rapid} and \eqref{potter_use_t}, the following inequalities hold for sufficiently small $t$ and $s$,
	\begin{equation*}
	\frac{1}{2}\le \frac{\psi_B(t)}{t} \le 2\psi_B(s)\max\Big\{t^{\rho - \delta - 1}\Big(\frac{1}{s}\Big)^{\rho - \delta},\, t^{\rho + \delta -1}\Big(\frac{1}{s}\Big)^{\rho + \delta}\Big\}.
	\end{equation*}
	If we fix $s$ and let $t\to 0_+$, the right-hand side will converge to 0 (because $\rho - \delta - 1 > 0$), which gives a contradiction. Consequently, we must have $\rho \le 1$. This completes the proof.
\end{proof}
Theorem~\ref{theo:cvo} has some useful corollaries that we will discuss next.
The first is that in the setting of quasi-cyclic algorithms  as in \eqref{alg_iter}, we can always obtain joint Karamata regularity.
\begin{corollary}\label{col:quasi}
Let $T_i:\E \to \E$ $(i = 1,\ldots,m)$ be $\alpha$-averaged operators as in problem \eqref{fcp_prob}.
If $T_i$ are definable in some {o-minimal structure}, then they  are jointly Karamata regular (\emph{resp.}  Karamata regular when $m = 1$) over any bounded set $B\subseteq\E$.
\end{corollary}
\begin{proof}
Each $T_i$ is $\alpha$-averaged, so, in particular, continuous and quasi-nonexpansive. Moreover,  $\bigcap_{i=1}^m\F\, T_i \neq \emptyset$ by assumption. 
Then applying Theorem~\ref{theo:cvo} with $L_i = T_i$ and $n = m$, we conclude that $T_i$ $(i = 1,\ldots,m)$ are jointly Karamata regular (\emph{resp.}  Karamata regular when $m = 1$) over any bounded set $B\subseteq\E$.
\end{proof}
As a consequence of Corollary~\ref{col:quasi}, whenever the problem is definable in some $o$-minimal structure, we can in principle use Theorem~\ref{theo:rv} and Theorem~\ref{theorem_int}  to analyze convergence rates of algorithms for solving the common fixed point problem \eqref{fcp_prob}.

Another consequence of Theorem~\ref{theo:cvo} is that the error bounds that describe intersections of definable convex sets can be taken to be regularly varying, which is a result we were not aware when \cite{LL20} was written. In what follows, we will use the notion of consistent error bound see Remark~\ref{rmk_def} and \cite{LL20}.
\begin{corollary}\label{col:consistent}
Let $C_1,\ldots, C_m \subseteq \E$ be closed convex sets definable in some $o$-minimal structure with nonempty intersection. Then, there exists a consistent error bound function $\Phi$ for $C_1,\ldots, C_m$ such that for all sufficiently large $r > 0$, we have $\Phi(\cdot,\, r) \in \RVz_{\rho}$ with index $\rho \in [0,\,1]$ when restricted to $(0,\,1]$ ($\rho$ may depend on $r$). In particular, for every bounded set $B \subseteq \E$, there exists  a nondecreasing function $\psi_B \in \RVz_{\rho}$ with $\rho \in [0,\,1]$ when restricted to $(0,\,1]$ and $\lim_{t\to0_+}\psi_B(t) = 0$ such that 
\begin{equation*}
\dist\Big(x,\,\bigcap_{i=1}^mC_i\Big) \le \psi_B\Big(\max_{1\le i\le m}\dist(x,\,C_i)\Big), \ \ \forall\ x\in B.
\end{equation*}
\end{corollary}
\begin{proof}
Apply Theorem~\ref{theo:cvo} by taking $m = n$ and $L_i$ to be the projection operator onto $C_i$. With that the $\Phi$ in Lemma~\ref{lem:op} becomes a consistent error bound function for the $C_1, \ldots, C_m$. Then we see from  Theorem~\ref{theo:cvo} that for sufficiently large $r > 0$, we have $\Phi(\cdot,\,r)  \in \RVz_{\rho}$ with $\rho \in [0,\,1]$ when restricted to $(0,\,1]$.
Finally, let $B$ be any bounded set and let $b$ be large enough so that $\sup _{x \in B} \norm{x} \leq b$ and $\Phi(\cdot,\,b)  \in \RVz_{\rho}$ with $\rho \in [0,\,1]$ when restricted to $(0,\,1]$. Define $\psi_B \coloneqq \Phi(\cdot,\, b)$. Then it follows from Lemma~\ref{lem:op}~(i) that $\psi_B$ is nondecreasing and it satisfies $\lim_{t\to0_+}\psi_B(t) = 0$.
Moreover, for every $x \in B$, letting $d(x)  \coloneqq \max_{1\le i\le m}\dist(x,\,C_i)$ we have from Lemma~\ref{lem:op}~(iii) that
\begin{equation*}
\dist\Big(x,\,\bigcap_{i=1}^mC_i\Big) \le \Phi(d(x),\,\norm{x}) \le
\Phi(d(x),\,b) =\psi_B(d(x)).
\end{equation*}
This completes the proof.
\end{proof}
Corollary~\ref{col:consistent} implies that the results of this paper and \cite{LL20} can be used to analyze several different types of algorithms for feasibility problems over definable convex sets.

\paragraph{Polynomially bounded $o$-minimal structures.}
The preceding results show that, under suitable assumptions, the $T_i$'s as in \eqref{fcp_prob}  are jointly Karamata regular and the corresponding regularity function can be taken as described in Theorem~\ref{theo:cvo}.
We also know that the index of regular variation is some $\rho \in [0,1]$. 

In view of Theorem~\ref{theo:rv}, if $\rho = 0$ we are somewhat out of luck, as item~$(i)$ only gives a lower bound to an upper bound. 
In order to extract more information in the $\rho = 0$ case, we need extra assumptions as in item~$(i)$ of Theorem~\ref{theorem_int} which was used to analyze convergence rates under logarithmic error bounds in Section~\ref{sec:previous}.

On the other hand, if $\rho \in (0,1]$ holds then we know from items~(ii) and (iii) of Theorem~\ref{theo:rv} that the convergence rate should be at least sublinear. To conclude this paper, we take a look at conditions ensuring that $\rho \in (0,1]$ holds.

There is a result called \emph{growth dichotomy} that states that a $o$-minimal structure either contains the graph of the exponential function or it is \emph{polynomially bounded}.
Here, \emph{polynomially bounded} means that for every definable function $f:\mathbb{R} \to \mathbb{R}$ there exists some natural number $N$ such that $|f(x)| \leq x^N$ holds for sufficiently large $x$, see \cite{Miller94}, \cite[Section~4.12]{Dri98}.

The $o$-minimal structures associated to semialgebraic sets and globally subanalytic sets are both polynomially bounded, e.g., see  \cite[Proposition~2.6.1]{BCR98} and Sections 5.3 and 5.4 of \cite{DM96} in view of the growth dichotomy, respectively. 
On the other hand, no $o$-minimal structure containing the exponential cone as a definable set is polynomially bounded, again by the growth dichotomy.

For polynomially bounded $o$-minimal structures we have the following lemma.
\begin{lemma}[\cite{Miller94}]\label{lem:dichotomy}
	Let $f: \R \to \R$ be definable in a polynomially bounded $o$-minimal structure. Then either $f(x)$ is zero for all sufficiently large $x$ or there exist $c,r \in \R$, $c\neq 0$ such that $\frac{f(x)}{x^r} \to c$ as $x\to +\infty$. 
\end{lemma}
This  implies the following result.
\begin{lemma}\label{lem:dichotomy2}
Let $g: \R_+ \to \R_+$ be definable in a polynomially bounded $o$-minimal structure. Suppose that $g(t) > 0$ for all sufficiently small positive $t$ and $\lim_{t\to 0_+} g(t) = 0$. Then, there exist $c > 0$ and $r \neq 0$ such that  $\frac{g(t)}{t^r} \to c$ as $t\to 0_+$. 
\end{lemma}
\begin{proof}
Suppose that $g(t) > 0$ for all $0 < t \leq t_0$.
Then, we define $f: \R \to \R$ such that $f(x) \coloneqq g(1/x)$ if $x \geq 1/t_0$ and $f(x) \coloneqq 0$ otherwise.
Since $g$ is definable, $f$ is definable as well\footnote{The union of finitely many definable sets is definable and the graph of $f$ is the union $\{(x,g(1/x)) \mid x \geq 1/t_0 \} \cup \{(x,0) \mid x < 1/t_0\}$. The first set in the union is definable because $g$ is definable.} and Lemma~\ref{lem:dichotomy2} implies the existence of $c,r \in \R$, $c \neq 0$ such that $\frac{f(x)}{x^r} \to c$ as $x\to +\infty$. 
This implies that $\frac{g(t)}{t^{-r}} \to c$ as $t\to 0_+$.

Because $g(t)$ is positive for small $t$, we have $c > 0$. Finally, $\lim_{t\to 0_+} g(t) = 0$ together with $c > 0$ tells us that $r \neq 0$.
\end{proof}

We can now refine Theorem~\ref{theo:cvo} in the polynomially bounded case.
\begin{theorem}\label{theo:krpoly}
Consider the setting of Theorem~\ref{theo:cvo} and suppose that the underlying $o$-minimal structure is polynomially bounded. Then, for sufficiently large $r > 0$, we have $\Phi(t, r) \overset{c}{\sim} t^\rho$ as $t \to 0_+$, with $\rho \in (0,1]$. In particular,  $\Phi(\cdot,\, r)\in \RVz_{\rho}$ with index $\rho \in (0,\,1]$ when restricted to $(0,\,1]$ ($\rho$ may depend on $r$). 
\end{theorem}
\begin{proof}
	We go back to the proof of Theorem~\ref{theo:cvo} and consider the same choice of $r$.
	Then, the function $\psi_B$ is such that $\psi_B(t) = \Phi(t,r)$ holds for $t \geq 0$. 
	Furthermore,
	from Lemma~\ref{lem:op}, we have
$\lim _{t \to 0_+} \psi_B(t) = 0$.
	Later in the proof it is shown that $\psi_B$ is definable and that $\psi_B(t) > 0$ holds for all $t > 0$, see \eqref{1st_form_t}.
	At the end, it is shown that $\psi_B \in \RVz_{\rho}$ with $\rho \in [0,1]$.
	
	Since the $o$-minimal structure is polynomially bounded, we apply Lemma~\ref{lem:dichotomy2} and conclude that there exists $c > 0$ and $\alpha \neq 0$ such that  
	$\Phi(t, r) = \psi_B(t)\overset{c}{\sim} t^\alpha$ as $t \to 0_+$, which implies that the index of regular variation of $\Phi(\cdot,r)$ is $\alpha$. Therefore, $\rho = \alpha$ and $\rho \neq 0$.
\end{proof}
Essentially, Theorem~\ref{theo:krpoly} tells us that if we are in a polynomially bounded $o$-minimal structure, the regularity functions can be assumed to be H\"olderian. 
This ensures that convergence rates are at least sublinear.

\begin{corollary}\label{col:poly_conv}
Under the setting and assumptions of Theorem~\ref{theo:rate}, suppose in addition that the $T_i$'s are all definable in a polynomially bounded $o$-minimal structure.
Then, either $\{x^k\}$ converges to some $x^*\in\bigcap_{i=1}^m\F\,T_i$ finitely or there exists some $r > 0$, $\kappa > 0$ such that convergence rate is given by 
\begin{equation*}
\dist(x^k,\,F) \leq \kappa k^{-r}  \ \ {\rm as}\ \ k\to+\infty.
\end{equation*}
\end{corollary}
\begin{proof}
By Theorem~\ref{theo:krpoly}, the regularity function in item~$(a)$ of Theorem~\ref{theo:rate} can be taken to satisfy $\psi_B(t) \overset{c}{\sim} t^\rho$ as $t \to 0_+$ with $\rho \in (0,1]$. 
Following the same arguments  that lead to 
\eqref{eq:holder_rate}, we conclude that $\phi$ in Theorem~\ref{theo:rate} also satisfies $\phi(t) \overset{c}{\sim} t^\rho$ as $t \to 0_+$.
Furthermore if $\rho \in (0,1)$, we have
\begin{equation}
\sqrt{\Phi_{\phi}^{-1}(k)} \overset{c}{\sim} k^{-\frac{\rho}{2(1-\rho)}} \ \ {\rm as}\ \ k\to\infty,
\end{equation}
which, in view of \eqref{gener_rate}, implies the desired result. If $\rho = 1$, we have a linear rate which, similarly, also implies the result.%
\end{proof}
The results in \cite[Section~4]{BLT17} imply that if $C,D$ are two convex semialgebraic sets with nonempty intersection, then the Douglas-Rachford algorithm converges at least sublinearly, see Corollary~4.1 therein.
Corollary~\ref{col:poly_conv} allow us to extend this to the polynomially bounded case.

\begin{corollary}
Suppose that $C,D \subseteq \E$ are convex sets definable in a polynomially bounded $o$-minimal structure such that $C\cap D \neq \emptyset$. Let $\{x^k\}$ be a sequence generated by the DR algorithm.  Then either $x^k$ converges to some $x^*\in \F\,T_{C,D}$ finitely or  there exist $ r > 0$, $\kappa > 0$ such that
\begin{equation*}
\dist(x^k,\F\,T_{C,D}) \leq \kappa k^{-r}  \ \ {\rm as}\ \ k\to\infty.
\end{equation*}
\end{corollary}
\begin{proof}
If $C,D$ are definable, then the projections $P_C, P_D$ are also definable in the same $o$-minimal structure. This implies that the DR operator is also definable. 
As discussed in Section~\ref{sec:dr}, the DR algorithm can be obtained from the quasi-cyclic iteration described in  \eqref{alg_iter} by  letting $m = 1$, $T_1 = T_{C,D}$ and $w_i^k \equiv 1$ (thus $\nu = 1$). With that, item $(b)$ of Theorem~\ref{theo:rate} is satisfied with $s = 1$. 
We can then apply Corollary~\ref{col:poly_conv} to obtain the result.
\end{proof}

Due to the growth dichotomy, the price to pay for imposing polynomial boundedness is that we can no longer capture certain phenomena involving exponentials and logarithms as in Sections~\ref{sec:ap}, \ref{sec:dr} and \cite{LL20,LLP20,LLLP24}. 
In the presence of exponentials and related functions, as discussed in Section~\ref{sec:examples}, we may still use Theorem~\ref{theorem_int} if we have enough information about the underlying regularity function or Theorem~\ref{theo:rv_old} if we are able to deduce through other considerations that the index is nonzero.

\appendix
\section{Conic representation of $C_1$ in \eqref{def_C1C2}}\label{app:rep}

First, we observe that  $C_1$ satisfies
\begin{equation}\label{C1_rew}
C_1 = \left\{(x,\,\mu) \mid \gamma(x) \leq \min\{\mu,\, \gamma(0.5)\}\right\}.
\end{equation}
Let $\widehat{C}_1$ be defined as
\begin{equation}\label{eq:wideC1}
\begin{split}
\widehat{C}_1 \coloneqq & \left\{(x,\,\mu,\,s,\,v,\,t)\mid |x|\le t,\, t/2 \le -v\ln(v),\, 0 < v \le \sqrt{s},\, 0 < s \le \min\{\mu,\, \gamma(0.5)\} \right\}\\
& \cup \left\{(x,\,\mu,\,s,\,v,\,t)\mid x = 0,\, v = 0,\, t = 0,\, 0 \le s \le \min\{\mu,\, \gamma(0.5)\}  \right\}.
\end{split}
\end{equation}
The following lemma establishes the connection between $C_1$ and $\widehat{C}_1$.
\begin{lemma}\label{lem:reform}
	$(x,\,\mu)\in C_1$ if and only if there exist $s,v,t\in\R$ such that  $(x,\,\mu,\,s,\, v,\, t)\in \widehat{C}_1$.
\end{lemma}

\begin{proof}
	Suppose that $(x,\,\mu)\in C_1$. We let $s\coloneqq \min\{\mu,\,\gamma(0.5)\}$, $v\coloneqq \sqrt{s}$ and $t\coloneqq |x|$. In case of $\mu = 0$, we have $s =0$ and see from $\gamma(x) \le \min\{\mu,\, \gamma(0.5)\} = 0$ that $x = 0$. Thus, $v = 0$ and $t = 0$. Consequently, we have $(x,\,\mu,\,s,\, v,\, t)\in \widehat{C}_1$. In case of $\mu > 0$, we have $0 < s = \min\{\mu,\, \gamma(0.5)\} $ and $0 < v = \sqrt{s}$. Moreover, note that 
	\begin{equation*}
	\gamma(|x|) = \gamma(x) \le \min\{\mu,\, \gamma(0.5)\} = s \le \gamma(0.5).
	\end{equation*}
	This further implies that
	\begin{equation*}
	t = |x| \le \gamma^{-1}(s) = -\sqrt{s}\ln(s) = -2v\ln(v),
	\end{equation*}
	which proves that $(x,\,\mu,\,s,\, v,\, t)\in \widehat{C}_1$.
	
	Conversely, suppose that $(x,\,\mu,\,s,\, v,\, t)\in \widehat{C}_1$. In case of $v = 0$, then $(x,\,\mu,\,s,\, v,\, t)$ must fall into the second piece of $\widehat{C}_1$ and thus we have $\gamma(x) = \gamma(0) = 0 \le s \le \min\{\mu,\, \gamma(0.5)\}$, which proves that $(x,\,\mu)\in C_1$. In case of $v > 0$, then $(x,\,\mu,\,s,\, v,\, t)$ must fall into the first piece of $\widehat{C}_1$. Consequently, we have 
	\begin{equation*}
	t\le f(v)\coloneqq -2v\ln(v),\, 0 < v\le \sqrt{s}  \le \sqrt{\gamma(0.5)} < e^{-1}.
	\end{equation*}
	Note that $f$ is monotone increasing on $(0,\,e^{-1})$. Therefore, we have 
	\begin{equation*}
	|x|\le t\le f(\sqrt{s}) = -\sqrt{s}\ln(s) = \gamma^{-1}(s),
	\end{equation*}
	which further implies that
	\begin{equation*}
	\gamma(x) = \gamma(|x|) \le s \le \min\{\mu,\, \gamma(0.5)\}. 
	\end{equation*}
	This proved that $(x,\,\mu)\in C_1$ and completes the proof.
\end{proof}

\begin{proof}[Proof of Proposition~\ref{prop:conic_rep}]
	In view of Lemma~\ref{lem:reform} it is enough to show that 
	$(x,\mu,s,v,t) \in \widehat C_1$ if and only if $x,\mu,s,v,t$ satisfy \eqref{eq:conic_rep}.
	
	First, we assume that $(x,\mu,s,v,t) \in \widehat C_1$.
	Suppose that $(x,\mu,s,v,t) $ belongs to the first piece of $\widehat{C}_1$ in \eqref{eq:wideC1}. 
	Then, $0<v \leq \sqrt{s}$ implies that $(0.5,s,v) \in \mathcal{Q}_r^3$, while $t/2 \leq -v\ln(v)$ together with $v > 0$ imply that $(t/2,v,1) \in \expCone$. 
	Overall, $(x,\mu,s,v,t) $ satisfies \eqref{eq:conic_rep}.
	Next, suppose that the  $(x,\mu,s,v,t) $ belongs to the second piece of $\widehat{C}_1$ in \eqref{eq:wideC1}, so that 
	$x =  v = t = 0$. Since 
	$(0,0,1) \in \expCone$ and $(0.5,s,0) \in \mathcal{Q}_r^3$ for $s \geq 0$, we conclude that 
	$(x,\mu,s,v,t) $ satisfies \eqref{eq:conic_rep}.
	
	Conversely, suppose that $(x,\mu,s,v,t) $ satisfies \eqref{eq:conic_rep}.
	Then $(t/2,v,1) \in \expCone$ implies that either $v > 0$ and
	$t/2 \leq -v\ln(v)$ hold or 
	$v = 0$ and $t \leq 0$ hold. 
	In the former case, $(x,\mu,s,v,t) $ belongs to the first piece of $\widehat{C}_1$. 
	In the latter case, in view of the other constraints in \eqref{eq:conic_rep} we have 
	$x= v= t= 0$ and 	$(x,\mu,s,v,t) $ belongs to the 
	second piece of $\widehat{C}_1$.
\end{proof}


{\small{
		\subsection*{Acknowledgements}
		We thank the referees and the editors for their  comments, which helped to improve the paper.
}}

\bibliographystyle{abbrvurl}
\bibliography{bib_plain}

\begin{thebibliography}{10}

\bibitem{ABC83}
R.~Aharoni, A.~Berman, and Y.~Censor.
\newblock An interior points algorithm for the convex feasibility problem.
\newblock {\em Advances in Applied Mathematics}, 4(4):479--489, Dec. 1983.

\bibitem{al06}
K.~B. Athreya and S.~N. Lahiri.
\newblock {\em Measure Theory and Probability Theory}, volume~19.
\newblock Springer, 2006.

\bibitem{ABRS10}
H.~Attouch, J.~Bolte, P.~Redont, and A.~Soubeyran.
\newblock Proximal alternating minimization and projection methods for
  nonconvex problems: an approach based on the {K}urdyka-{{\L}}ojasiewicz
  inequality.
\newblock {\em Mathematics of Operations Research}, 35(2):438--457, 2010.

\bibitem{BB1993}
H.~H. Bauschke and J.~M. Borwein.
\newblock On the convergence of von {N}eumann's alternating projection
  algorithm for two sets.
\newblock {\em Set-Valued Analysis}, 1:185--212, 1993.

\bibitem{BB96}
H.~H. Bauschke and J.~M. Borwein.
\newblock On projection algorithms for solving convex feasibility problems.
\newblock {\em SIAM Review}, 38(3):367--426, 1996.

\bibitem{BC11}
H.~H. Bauschke and P.~L. Combettes.
\newblock {\em Convex Analysis and Monotone Operator Theory in Hilbert Spaces}.
\newblock CMS Books in Mathematics. Springer International Publishing, 2017.

\bibitem{BNP15}
H.~H. Bauschke, D.~Noll, and H.~M. Phan.
\newblock Linear and strong convergence of algorithms involving averaged
  nonexpansive operators.
\newblock {\em Journal of Mathematical Analysis and Applications},
  421(1):1–20, 2015.

\bibitem{BGO06}
N.~H. {Bingham}, C.~M. {Goldie}, and E.~{Omey}.
\newblock {Regularly varying probability densities}.
\newblock {\em {Publications de l'Institut Math\'ematique. Nouvelle S\'erie}},
  80:47--57, 2006.

\bibitem{BGT87}
N.~H. Bingham, C.~M. Goldie, and J.~L. Teugels.
\newblock {\em Regular Variation}.
\newblock Encyclopedia of Mathematics and its Applications. Cambridge
  University Press, 1987.

\bibitem{BCR98}
J.~Bochnak, M.~Coste, and M.-F. Roy.
\newblock {\em Real Algebraic Geometry}.
\newblock Springer Berlin Heidelberg, 1998.

\bibitem{BDLS07}
J.~Bolte, A.~Daniilidis, A.~Lewis, and M.~Shiota.
\newblock Clarke subgradients of stratifiable functions.
\newblock {\em {SIAM} Journal on Optimization}, 18(2):556--572, Jan. 2007.

\bibitem{bdlm10}
J.~Bolte, A.~Daniilidis, O.~Ley, and L.~Mazet.
\newblock Characterizations of {{{\L}}ojasiewicz} inequalities: subgradient
  flows, talweg, convexity.
\newblock {\em Transactions of the American Mathematical Society},
  362(6):3319--3363, 2010.

\bibitem{BNPS17}
J.~Bolte, T.~P. Nguyen, J.~Peypouquet, and B.~W. Suter.
\newblock From error bounds to the complexity of first-order descent methods
  for convex functions.
\newblock {\em Mathematical Programming}, 165(2):471--507, 2017.

\bibitem{BLT17}
J.~M. Borwein, G.~Li, and M.~K. Tam.
\newblock Convergence rate analysis for averaged fixed point iterations in
  common fixed point problems.
\newblock {\em SIAM Journal on Optimization}, 27(1):1--33, 2017.

\bibitem{BL16}
J.~M. Borwein and S.~B. Lindstrom.
\newblock Meetings with {L}ambert {W} and other special functions in
  optimization and analysis.
\newblock {\em Pure and Applied Functional Analysis}, 1(3):361--396, 2016.

\bibitem{Censor84}
Y.~Censor.
\newblock {\em Iterative Methods for the Convex Feasibility Problem}, pages
  83--91.
\newblock Elsevier, 1984.

\bibitem{CS17}
V.~Chandrasekaran and P.~Shah.
\newblock Relative entropy optimization and its applications.
\newblock {\em Mathematical Programming}, 161(1):1--32, 2017.

\bibitem{Coste00}
M.~Coste.
\newblock {\em An Introduction to O-minimal Geometry}.
\newblock Pisa: Istituti editoriali e poligrafici internazionali., 2000.

\bibitem{dt07}
D.~Djurčić and A.~Torgašev.
\newblock Some asymptotic relations for the generalized inverse.
\newblock {\em Journal of Mathematical Analysis and Applications},
  335(2):1397--1402, 2007.

\bibitem{ED13}
N.~Elez and D.~Djur\v{c}i\'{c}.
\newblock Some properties of rapidly varying functions.
\newblock {\em Journal of Mathematical Analysis and Applications},
  401(2):888--895, 2013.

\bibitem{EH13}
P.~Embrechts and M.~Hofert.
\newblock A note on generalized inverses.
\newblock {\em Mathematical Methods of Operations Research}, 77(3):423--432,
  2013.

\bibitem{FZ90}
S.~D. Fl{\aa}m and J.~Zowe.
\newblock Relaxed outer projections, weighted averages and convex feasibility.
\newblock {\em BIT}, 30(2):289–300, June 1990.

\bibitem{Fr23}
H.~A. Friberg.
\newblock Projection onto the exponential cone: a univariate root-finding
  problem.
\newblock {\em Optimization Methods and Software}, pages 1--17, 2023.

\bibitem{GC24}
P.~J. Goulart and Y.~Chen.
\newblock {Clarabel}: {An} interior-point solver for conic programs with
  quadratic objectives.
\newblock {\em arXiv preprint arXiv:2405.12762}, 2024.

\bibitem{BDNP16}
D.~N. Heinz H.~Bauschke, Minh H.~Dao and H.~M. Phan.
\newblock Proximal point algorithm, {D}ouglas-{R}achford algorithm and
  alternating projections: a case study.
\newblock {\em Journal of Convex Analysis}, 23(1):237--261, 2016.

\bibitem{HL93}
J.-B. Hiriart-Urruty and C.~Lemaréchal.
\newblock {\em Convex Analysis and Minimization Algorithms I}.
\newblock Springer Berlin Heidelberg, 1993.

\bibitem{Hoang16}
P.~D. Ho{\`a}ng.
\newblock {\L}ojasiewicz-type inequalities and global error bounds for
  nonsmooth definable functions in o-minimal structures.
\newblock {\em Bulletin of the Australian Mathematical Society}, 93(1):99--112,
  2016.

\bibitem{Io09}
A.~D. Ioffe.
\newblock An invitation to tame optimization.
\newblock {\em {SIAM} Journal on Optimization}, 19(4):1894--1917, Jan. 2009.

\bibitem{Io13}
A.~D. Ioffe.
\newblock Nonlinear regularity models.
\newblock {\em Mathematical Programming}, 139(1–2):223–242, Mar. 2013.

\bibitem{KT24}
M.~Karimi and L.~Tun{\c{c}}el.
\newblock {Domain-Driven Solver (DDS) Version} 2.1: a {MATLAB}-based software
  package for convex optimization problems in domain-driven form.
\newblock {\em Mathematical Programming Computation}, 16(1):37--92, 2024.

\bibitem{LP18}
G.~Li and T.~K. Pong.
\newblock Calculus of the exponent of {K}urdyka--{{\L}}ojasiewicz inequality
  and its applications to linear convergence of first-order methods.
\newblock {\em Foundations of Computational Mathematics}, 18(5):1199--1232,
  2018.

\bibitem{LLLP24}
Y.~Lin, S.~B. Lindstrom, B.~F. Louren\c{c}o, and T.~K. Pong.
\newblock Tight error bounds for log-determinant cones without constraint
  qualifications.
\newblock {\em Journal of Optimization Theory and Applications}, 205(3), Apr.
  2025.

\bibitem{LLP20}
S.~B. Lindstrom, B.~F. Louren{\c{c}}o, and T.~K. Pong.
\newblock Error bounds, facial residual functions and applications to the
  exponential cone.
\newblock {\em Mathematical Programming}, 200(1):229--278, 2023.

\bibitem{LL20}
T.~Liu and B.~F. Louren\c{c}o.
\newblock Convergence analysis under consistent error bounds.
\newblock {\em Foundations of Computational Mathematics}, 24(2):429--479, 2024.

\bibitem{LTT18_2}
D.~R. Luke, N.~H. Thao, and M.~K. Tam.
\newblock Implicit error bounds for {P}icard iterations on {H}ilbert spaces.
\newblock {\em Vietnam Journal of Mathematics}, 46(2):243–258, Feb. 2018.

\bibitem{Miller94}
C.~Miller.
\newblock Exponentiation is hard to avoid.
\newblock {\em Proceedings of the American Mathematical Society},
  122(1):257--259, 1994.

\bibitem{MC2020}
{{MOSEK} {ApS}}.
\newblock {\em {MOSEK} Modeling Cookbook Release 3.4.0}, 2025.

\bibitem{MM08}
M.~Mr{\v{s}}evi{\'c}.
\newblock Convexity of the inverse function.
\newblock {\em The Teaching of Mathematics}, (20):21--24, 2008.

\bibitem{Ottavy88}
N.~Ottavy.
\newblock Strong convergence of projection-like methods in {H}ilbert spaces.
\newblock {\em Journal of Optimization Theory and Applications},
  56(3):433–461, Mar. 1988.

\bibitem{PY22}
D.~Papp and S.~Y{\i}ld{\i}z.
\newblock Alfonso: {Matlab} package for nonsymmetric conic optimization.
\newblock {\em INFORMS Journal on Computing}, 34(1):11--19, 2022.

\bibitem{rockafellar}
R.~T. Rockafellar.
\newblock {\em {Convex Analysis}}.
\newblock Princeton University Press, 1997.

\bibitem{LTT18}
D.~Russell~Luke, N.~H. Thao, and M.~K. Tam.
\newblock Quantitative convergence analysis of iterated expansive, set-valued
  mappings.
\newblock {\em Mathematics of Operations Research}, 43(4):1143–1176, Nov.
  2018.

\bibitem{Se76}
E.~Seneta.
\newblock {\em Regularly Varying Functions}.
\newblock Lecture Notes in Mathematics. Springer Berlin Heidelberg, 1976.

\bibitem{DM96}
L.~van~den Dries and C.~Miller.
\newblock Geometric categories and o-minimal structures.
\newblock {\em Duke Mathematical Journal}, 84(2), Aug. 1996.

\bibitem{Dri98}
L.~P.~D. van~den Dries.
\newblock {\em Tame Topology and O-minimal Structures}.
\newblock Cambridge University Press, May 1998.

\bibitem{YLP21}
P.~Yu, G.~Li, and T.~K. Pong.
\newblock Kurdyka--{{\L}}ojasiewicz exponent via inf-projection.
\newblock {\em Foundations of Computational Mathematics}, 22(4):1171--1217,
  July 2021.

\end{thebibliography}
%
%
%
%
%
%
%
%
%

\end{document}